%% file: curvature_complexity.tex
\documentclass[final,12pt]{colt2023} 

\usepackage{array} %
\usepackage{caption} %
\usepackage{float} %
\usepackage{tablefootnote} %
\usepackage{multirow} %
\usepackage{enumitem}
\usepackage{ifthen}

        \hypersetup{
           breaklinks=true,   
           colorlinks=true,   
           pdfusetitle=true,  
        }

\usepackage{url} %
\include{preamble}

\title[Lower bounds for geodesically convex optimization]{Curvature and complexity:\\Better lower bounds for geodesically convex optimization}
\usepackage{times}

 \coltauthor{\Name{Christopher Criscitiello} \Email{christopher.criscitiello@epfl.ch}\AND
\Name{Nicolas Boumal} \Email{nicolas.boumal@epfl.ch}\\
\addr Ecole Polytechnique F\'ed\'erale de Lausanne (EPFL), Institute of Mathematics}

\begin{document}

\maketitle


\begin{abstract}%
We study the query complexity of geodesically convex (g-convex) optimization on a manifold.
To isolate the effect of that manifold's curvature, we primarily focus on hyperbolic spaces.
In a variety of settings (smooth or not; strongly g-convex or not; high- or low-dimensional), known upper bounds worsen with curvature.
It is natural to ask whether this is warranted, or an artifact.

For many such settings, we propose a first set of lower bounds which indeed confirm that (negative) curvature is detrimental to complexity.
To do so, we build on recent lower bounds \citep{hamilton2021nogo,bumpfctpaper} for the particular case of smooth, strongly g-convex optimization.
Using a number of techniques, we also secure lower bounds which capture dependence on condition number and optimality gap, which was not previously the case.

We suspect these bounds are not optimal.
We conjecture optimal ones, and support them with a matching lower bound for a class of algorithms which includes subgradient descent, and a lower bound for a related game.
Lastly, to pinpoint the difficulty of proving lower bounds, we study how negative curvature influences (and sometimes obstructs) interpolation with g-convex functions.
\end{abstract}


\section{Introduction and contributions}




Let $\calM$ be a $d$-dimensional Riemannian manifold.
We consider optimization problems of the form
\begin{align}
	\min_{x \in \calM} f(x), && \textrm{ knowing that } f \textrm{ attains its minimum in } B(\xorigin, r) = \{ x : \dist(x, \xorigin) \leq r \}.
	\tag{P}
	\label{eq:P}
\end{align}
Here, $f \colon \calM \rightarrow \reals$ is geodesically convex (g-convex) and
$\dist$ denotes Riemannian distance.
When $\calM$ is a Euclidean space, problem~\eqref{eq:P} amounts to convex optimization. 
Motivated by applications in statistics, machine learning and computer science
(see below),
it is natural to study algorithms for~\eqref{eq:P} (upper bounds), and to ask whether they are optimal (lower bounds).

Both in smooth and nonsmooth g-convex optimization, known upper bounds \emph{worsen} with the curvature of $\calM$~\citep{zhang2016complexitygeodesicallyconvex,bento2017iterationcomplexity}.
\emph{Is this effect of curvature warranted?}

To address this question, we focus on Hadamard manifolds, which have nonpositive curvature, because they are the most natural setting for geodesic convexity.
Non-Hadamard manifolds, even simple ones, often do not globally carry interesting g-convex functions.\footnote{For example, on compact manifolds, global g-convex functions are constant.  Moreover, if a complete Riemannian manifold admits a smooth strongly g-convex function, then it must be diffeomorphic to $\reals^d$~\citep[Prop.~5.10]{sakai1996riemannian}.}
Accordingly, most applications of geodesic convexity are in Hadamard manifolds.  
To isolate the effect of curvature, we further focus on the case where $\calM$ has constant negative curvature, i.e., $\calM$ is a \emph{hyperbolic space} of curvature $K < 0$.\footnote{We expect our results extend to the manifold of positive definite matrices with affine-invariant metric, because it contains a large totally geodesic submanifold isometric to hyperbolic space~\citep[App. J]{bumpfctpaper}.}
We prove lower bounds for algorithms with access to first-order information (function values and subgradients).
Dependence on curvature is captured by the quantity $\zeta = \zeta_{r \sqrt{-K}} := \frac{r \sqrt{-K}}{\tanh(r \sqrt{-K})}$, as it appears in known upper bounds.\footnote{If $r \sqrt{-K} \ll 1$, then $\zeta \approx 1$ and bounds match those from Euclidean space.  If $r \sqrt{-K} \gg 1$, then $\zeta \sim r \sqrt{-K}$ is large.}

Lower bounds in the Euclidean case are well understood.
For curved spaces, the only known result is due to~\citet{hamilton2021nogo} and~\citet{bumpfctpaper}, who prove the lower bound\footnote{We write $\tilde\Omega(\cdot)$, $\tilde O(\cdot)$ or $\tilde\Theta(\cdot)$ when we omit logarithmic terms.} $\tilde \Omega(\zeta)$ for smooth strongly g-convex optimization, showing that in this setting the effect of curvature is unavoidable.
However, \emph{this known lower bound has two major limitations}: 
\begin{enumerate}
\item[(a)] {It only holds for smooth, strongly g-convex optimization} (as opposed to other settings);
\item[(b)] {It does not depend on the condition number} $\kappa$ of the strongly g-convex function.
\end{enumerate}
These are real limitations.
Many g-convex problems in applications fail to be smooth or strongly g-convex, e.g., Fr\'echet medians~\citep{Fletcher2009TheGM} or operator scaling~\citep{allenzhuoperatorsplitting}.  
Moreover, most problems that are smooth and strongly g-convex have condition number $\kappa$ which is much greater than $\zeta$ (e.g., robust covariance estimation~\citep{franksmoitra2020}), so the complexity of the problem is determined by $\kappa$ rather than $\zeta$.
(We note however that the problem of Fr\'echet means~\citep{karcher1977riemannian} is strongly g-convex with condition number that naturally scales as $\Theta(\zeta)$.)


In this paper, we address these two limitations.  
We provide \emph{four main contributions}.
First, to address limitation (a), in Section~\ref{nonstronglyconvexcaseextension} we extend the $\tilde\Omega(\zeta)$ lower bound to the Lipschitz and the smooth (nonstrongly) g-convex settings.
This shows that ``full'' acceleration $O(\frac{1}{\sqrt{\epsilon}})$ (independent of curvature) for smooth g-convex optimization is impossible, and also justifies the presence of curvature terms in the upper bounds for subgradient descent proven by~\citet{zhang2016complexitygeodesicallyconvex}.

Second, to address limitation (b), in Sections~\ref{secnonsmoothlowerbounds} and~\ref{secsmoothlowerbounds} we
provide new techniques for proving lower bounds which depend on the problem class parameters (e.g., $\kappa$ or $\epsilon$).
In addition, in Appendix~\ref{carryovereuc} we show how upper and lower bounds from Euclidean space carry over to the Riemannian setting when $r$ is very small.
In the first four rows of Table~\ref{table:maintheorem}, we present these lower bounds, along with the best known upper bounds.  See Section~\ref{problemclasses} for the precise function classes and complexity measures.
For any fixed $\zeta$, the scalings of these lower bounds match those found in Euclidean space.

Our lower bounds do not match the best known upper bounds in Table~\ref{table:maintheorem}, and we strongly suspect the lower bounds can be improved.
Our third contribution is to provide a roadmap for improving on our lower bounds.
We \emph{conjecture} that the ``multiplicative'' upper bounds in the second, third and fourth rows of the table are optimal, and that the optimal upper bound for the first row is $\Theta(\zeta d)$.\footnote{The lower bounds we prove depend \emph{additively} in terms of $\zeta$ and the other parameters ($\tilde\Omega(\zeta + \sqrt{\kappa})$ for the strongly g-convex case).  On the other hand, the best known upper bounds depend \emph{multiplicatively} on the parameters ($O(\sqrt{\zeta \kappa})$).}
{We do not know if those are the correct bounds, but we believe they are reasonable targets.}

We provide two pieces of evidence for this conjecture.  
First, in Section~\ref{polyaknotapp} \emph{we prove the {multiplicative} lower bound $\Omega(\frac{\zeta}{\epsilon^2})$ for a class of (intuitively reasonable) algorithms which includes subgradient descent} (with Polyak step size).  In particular, this shows that the analysis of subgradient descent cannot be improved.
Second, in Section~\ref{feasibilitygame} we prove the $\tilde\Omega(\zeta d)$ lower bound for a ``cutting-planes game'' which serves as a proxy for low-dimensional Lipschitz g-convex optimization, and also provides a lower bound for g-convex ``cutting-planes schemes''~\citep[Sec.~3.2.6]{nesterov2004introductory}.\footnote{\citet[Sec.~3.2.5]{nesterov2004introductory} considers the analogous game as a proxy for Lipschitz convex optimization.}

To build lower bounds for a function class (e.g., g-convex functions), one (implicitly or explicitly) \emph{interpolates} a collection of function values and gradients with a function from that class.
The recent focus on \emph{necessary and sufficient interpolation conditions} for convex functions~\citep{taylorinterpolation2016}, has led to a number of insights into the complexity of convex optimization~\citep{taylorcomposite,acceleratingmirrordescentisnotpossible,taylorexact}.
Our fourth contribution is a study of interpolation by g-convex functions in Section~\ref{interpolation}.
We show that (unlike in the convex case), the naive necessary conditions for interpolation by g-convex functions are \emph{not} sufficient.  
This geometric obstruction partly explains the difficulties of showing lower bounds for g-convex optimization.
\bigskip

\begin{minipage}{0.95\linewidth}
{\centering
\begin{tabular} { |m{2.6cm}||m{1.3cm}|m{1.7cm}|m{1.5cm}|m{3cm}|m{1.6cm}| }
\hline
\multicolumn{6}{|c|}{Complexity of g-convex optimization} \\
\hline
g-convex setting & Function class & Lower bound & Upper bound & Algorithm & Result\\
\hline
Lipschitz, low-dimensional,~\ref{lowdimLipschitzproblem} &  $\mathcal{F}_{r, M, 0, \infty}^d$                                                                                            & \multicolumn{1}{c|}{$\tilde \Omega(\zeta + d)$}  & \multicolumn{1}{c|}{$O(\zeta d^2)$}                                             & centerpoint method \citep{rusciano2019riemannian} & Thm.~\ref{nicetheorem}, \text{     } Rmk.~\ref{remarkresult} \\ 
\hline
Lipschitz,~\ref{highdimLipschitzproblem} & $\mathcal{F}_{r, M, 0, \infty}$                                                                        & \multicolumn{1}{c|}{$\tilde\Omega(\zeta + \frac{1}{\zeta^2 \epsilon^2})$}  & \multicolumn{1}{c|}{$O(\frac{\zeta}{\epsilon^2})$}      & subgradient descent \citep{zhang2016complexitygeodesicallyconvex} & Thm.~\ref{nicetheorem},~\ref{nonsmoothlowerbound} \\ 
\hline
smooth,~\ref{smoothproblem} & $\mathcal{F}_{r, \infty, 0, L}$                           & $\tilde \Omega(\zeta + \frac{1}{\zeta \sqrt{\epsilon}})$     & \multicolumn{1}{c|}{$\tilde O\Big(\sqrt{\frac{\zeta}{\epsilon}}\Big)$ \footnote{The upper bounds for smooth g-convex and smooth strongly g-convex minimization have an important caveat. Those upper bounds come from \citet{kimyang} who \emph{assume} the iterates produced by their algorithm stay in the optimization domain $B(\xorigin, r)$.
Recent work by~\citet{davidmr} shows how to remove this assumption at the expense of somewhat worse complexity guarantees: $\tilde O(\zeta \sqrt{\frac{1}{\epsilon}})$ and $\tilde O(\zeta \sqrt{\kappa})$.}}  & RNAG-C \citep{kimyang} & Thm.~\ref{nicetheorem},~\ref{thmsmoothlowerbound} \\ 
\hline
smooth, strongly g-convex,~\ref{smoothstronglyproblem} & $\mathcal{F}_{r, \infty, \mu, L}$         & \multicolumn{1}{c|}{$\tilde\Omega(\zeta + \sqrt{\kappa})$}         & \multicolumn{1}{c|}{$O(\sqrt{\zeta \kappa})$ \footnote{\citet{kimyang} report the bounds $O(\zeta \sqrt{\frac{1}{\epsilon}})$ and $O(\zeta \sqrt{\kappa})$.  
We observe that the rate $O(\zeta \sqrt{\kappa})$ can be improved to $O(\sqrt{\zeta \kappa})$ by choosing the step size in their algorithm as $\Theta(\frac{1}{L})$ instead of $\Theta(\frac{1}{\zeta L})$.  More precisely, using that $\kappa \geq \zeta$ and the step size $\Theta(\frac{1}{L})$, one can check that their analysis (see Corollary F.1 in their paper) follows through unchanged.
Moreover, the rate $\tilde O(\sqrt{\frac{\zeta}{\epsilon}})$ can be achieved by a reduction from the $O(\sqrt{\zeta \kappa})$ strongly g-convex algorithm (see Proposition~\ref{reductionstronglyconvextoconvex}), using that $\epsilon \leq \frac{8}{\zeta}$ (see Proposition~\ref{usefulguy}).  }}                                                                                            & RNAG-SC \citep{kimyang} & Cor.~\ref{thisremark} \\ 
\hline
cutting-planes game, Sec. \ref{feasibilitygame} & N/A                                                                                                                 & \multicolumn{1}{c|}{$\tilde \Omega(\zeta d)$}                     & \multicolumn{1}{c|}{$O(\zeta d^2)$}                              & centerpoint method \citep{rusciano2019riemannian} & Thm.~\ref{lowerboundforfeasibilityproblem} \\ 
\hline
\end{tabular}
\captionof{table}{$\calM$ is a hyperbolic space of curvature $K < 0$, and $\zeta = \zeta_{r\sqrt{-K}} = \frac{r \sqrt{-K}}{\tanh(r \sqrt{-K})}$.
All entries in the column ``Lower bound'' are novel, except for the term $\zeta$ in the fourth row.} \label{table:maintheorem} }
\end{minipage}

\subsection{Problem classes and algorithms}\label{problemclasses}
Let $\mathbb{H}^d$ denote a $d$-dimensional hyperbolic space.  Without loss of generality, we can assume $\mathbb{H}^d$ has curvature $K = -1$, and so we do this from now on.\footnote{Let $\calM_1 = (\calM, g)$ be a hyperbolic space of curvature $K_1 < 0$, and scale the metric $g$ on $\calM$ to get a hyperbolic space $\calM_2 = (\calM, \frac{K_1}{K_2} g)$ of curvature $K_2$.
Then it is easy to see that the function classes $\calF_{r_1, M, \mu, L}^{\xorigin}(\calM_1)$ and $\calF_{r_2, M, \mu, L}^{\xorigin}(\calM_2)$ are identical provided $r_1 \sqrt{-K_1} = r_2 \sqrt{-K_2}$.  Stated differently: what matters is $r \sqrt{-K}$, not $r$ and $K$ separately.  We thus fix $K=-1$.}
We usually denote the underlying manifold by $\calM$ (i.e., $\calM = \mathbb{H}^d$), and its tangent bundle by $\T\calM$ (see Section~\ref{prelims}).
For each $d\geq 2$, fix a point $\xorigin^{d} \in \mathbb{H}^d$.  Throughout we usually drop the superscript on $\xorigin^{d}$ and write $\xorigin$.
For a function $f \colon \calM \to \reals$, we always denote $f^* = \min_{x \in \calM} f(x)$. 

For $r > 0$, $M \geq 0$, and $L \geq \mu \geq 0$, let $\calF_{r,  M, \mu, L}^{d}$ be the class of functions $f \colon \mathbb{H}^d \to \reals$ which (a) are globally $M$-Lipschitz, (b) are globally $\mu$-strongly g-convex, (c) are $L$-smooth (globally if $\mu = 0$, or in the ball $B(\xorigin, r)$ if $\mu > 0$),\footnote{For (c): if $\mu > 0$ and we require $f$ to be $L$-smooth globally, then the function class is empty, see Section~\ref{lrvslrsquared}.} and (d) have a global minimizer $x^*$ which is in $B(\xorigin, r)$.  
Lastly, define $\calF_{r, M, \mu, L} = \bigcup_{d=2}^\infty \calF_{r, M, \mu, L}^d$.

A deterministic first-order algorithm $\calA$ on $\calM$ is a sequence of functions $(\calA_{k} \colon (\reals \times \T \calM)^k \rightarrow \calM)_{k \geq 0}$; in particular, $\calA_0$ returns an initial point $x_0$.
Such an algorithm has access to an oracle $\calO_f\colon \calM \rightarrow \reals \times \T \calM$ which for each query $x \in \calM$ returns the function value $f(x)$ and a subgradient $g \in \partial f(x)$.
Running $\calA$ with an oracle $\calO_f$ produces iterates $x_0, x_1, \ldots$ as follows.\footnote{We often implicitly assume there is an oracle associated to $f$, and simply say ``running $\calA$ on $f$ produces iterates \ldots.''}
Let $\calH_0 = \emptyset$.
After already making $k \geq 0$ queries $x_0, \ldots, x_{k-1}$, the algorithm uses the known information $\calH_k$ to compute the next query $x_k = \calA_k(\calH_k)$.
The oracle $\calO_f$ gives the algorithm $F_k = f(x_k)$ and $g_k \in \partial f(x_k)$, and we update the known information $\calH_{k+1} = (F_\ell, x_\ell, g_\ell)_{\ell=0}^k$.

Given a tolerance $\delta > 0$, we define $T_{\delta}(\mathcal{A}, f)$ to be the first $k$ for which $f(x_k) - f^* \leq \delta$, where $(x_k)_{k\geq 0}$ is the sequence produced by running $\calA$ on $f$.
Given a function class $\mathcal{F}$, the complexity of optimization on that function class is defined by infimizing over deterministic first-order algorithms:
$$T_{\delta}(\mathcal{F}) = \inf_{\mathcal{A}} \sup_{f \in \mathcal{F}} T_{\delta}(\mathcal{A}, f).$$

We focus on four function classes.  In the following, $M > 0, L > 0, d \geq 2$.
\begin{enumerate} [label=\textbf{P\arabic*}]
\item\label{lowdimLipschitzproblem} (\emph{low-dimensional Lipschitz}) 
For $0 < \epsilon < 1$ , define $\delta = \epsilon \cdot M r$, and $T_{\epsilon, r, d} = T_{\delta}(\calF_{r, M, 0, \infty}^d)$.\footnote{With this choice of $\delta$, by scaling the functions in the class, one can check that $T_{\delta}(\calF_{r, M, 0, \infty}^d)$ is independent of $M$, and only depends on $\epsilon, r, d$.  Similar considerations hold for the other function classes.}

\item\label{highdimLipschitzproblem} (\emph{high-dimensional Lipschitz}) 
For $0 < \epsilon < 1$ , define $\delta = \epsilon \cdot M r$, and $T_{\epsilon, r} = T_{\delta}(\calF_{r, M, 0, \infty})$.  

\item\label{smoothproblem} (\emph{high-dimensional smooth}) 
For $0 < \epsilon < \min\{1, \frac{8}{\zeta_r}\}$, define $\delta = \epsilon \cdot \frac{1}{2} L r^2$, $T_{\epsilon, r} = T_{\delta}(\calF_{r, \infty, 0, L})$.\footnote{This is the right scaling for $\epsilon$ -- see Section~\ref{lrvslrsquared}.  Indeed, in Proposition~\ref{usefulguy} we prove if $f \in \calF_{r, \infty, 0, L}$ then $f(\xorigin) - f^* \leq \frac{1}{2} L r^2 \cdot \frac{8}{\zeta_r}$.  This is analogous to the observation that $\kappa \geq \zeta_r$, although a different proof is required to show this.}

\item\label{smoothstronglyproblem} (\emph{high-dimensional smooth strongly g-convex}) For $0 < \epsilon < \min\{1, \frac{8}{\zeta_r}\}, \mu > 0, L \geq \mu \zeta_r$, define $\delta = \epsilon \cdot \frac{1}{2} L r^2$, $\kappa = \frac{L}{\mu}$, and $T_{\epsilon, r, \kappa} = T_{\delta}(\calF_{r, \infty, \mu, L})$.  
\end{enumerate}

We expect complexity of problems~\ref{highdimLipschitzproblem} and~\ref{smoothproblem}  (both denoted by $T_{\epsilon, r}$) to scale \emph{polynomially} in $\epsilon^{-1}$.  
For those, Table~\ref{table:maintheorem} reports bounds on $T_{\epsilon, r}$.
On the other hand, due to known upper bounds,
we can solve problems~\ref{lowdimLipschitzproblem} and~\ref{smoothstronglyproblem} at least at a linear rate, that is, scaling \emph{logarithmically} in $\epsilon$ as $\log(\epsilon^{-1})$.  
Therefore, it is reasonable to at least initially focus on the factors in front of the $\log(\epsilon^{-1})$: this is the approach we take.
Said differently, for problem~\ref{lowdimLipschitzproblem} we expect there exists $q > 0$ so that $f(x_k) - f^* \leq M r \cdot e^{-k/q}$ after $k$ queries.  
\emph{The quantity $q$ measures the number of queries needed to reduce the optimality gap by a constant factor.}
To suppress the dependence on $\epsilon$
for~\ref{lowdimLipschitzproblem} and~\ref{smoothstronglyproblem}, we define $q_{d, r} = \sup_{\epsilon \in (0,1)} \{\frac{T_{\epsilon, r, d}}{\log(\epsilon^{-1})}\}$ and define $q_{\kappa, r}$ similarly.
Table~\ref{table:maintheorem} reports bounds on $q_{d, r}$ and $q_{\kappa, r}$.

\section{Preliminaries: Hadamard manifolds, hyperbolic spaces and geodesic convexity} \label{prelims}

Throughout, $\calM$ denotes a Hadamard manifold which has tangent bundle $\T \calM$ and tangent spaces $\T_x \calM$.
A Hadamard manifold is a complete, simply connected Riemannian manifold with nonpositive curvature, see~\citep{bridsonmetric} and~\citep[Ch.~12]{lee2018riemannian}.  
The Riemannian metric on $\calM$ is denoted $\inner{\cdot}{\cdot}$ and $\norm{v} = \sqrt{\inner{v}{v}}$ for $v \in \T_x \calM$.
The Riemannian metric gives $\calM$ a notion of distance $\dist$, volume $\Vol$, geodesics, and intrinsic curvature.
The exponential map on $\calM$ at $x \in \calM$ is denoted by $\exp_x \colon \T_x\calM \to \calM$, and its inverse map by $\log_x \colon \calM \to \T_x\calM$.
As $\calM$ is Hadamard, the exponential map and its inverse are global diffeomorphisms.
For $x, y \in \calM$, $P_{x \rightarrow y} \colon \T_x \calM \rightarrow \T_y \calM$ denotes parallel transport along the geodesic connecting $x$ and $y$.
The boundary of a set $D \subseteq \calM$ is denoted $\partial D$, and its interior is $\interior D = D \setminus \partial D$.

We usually take $\calM$ to be a $d$-dimensional hyperbolic space of curvature $-1$, denoted by $\mathbb{H}^d$.
Certain submanifolds of $\mathbb{H}^d$ feature prominently in our lower bound constructions.
A connected and complete Riemannian submanifold $S \subseteq \calM$ is called \emph{totally geodesic} if a geodesic in $S$ is also a geodesic in $\calM$ (see Appendix~\ref{totgeosubmanifoldprelimaries}).
Totally geodesic submanifolds are abundant in $\mathbb{H}^d$ (but are rare in general Hadamard manifolds~\citep[Sec. 11.1]{CHEN2000187}).

A set $D \subseteq \calM$ is \emph{g-convex} if for all $x, y \in D$
the geodesic segment connecting $x, y$ is contained in $D$.
Balls $B(x, r)$, $r \geq 0$, and totally geodesic submanifolds are g-convex sets.
The following lemma can be proven using the hyperboloid or Beltrami-Klein models of $\mathbb{H}^d$ (Appendix~\ref{totgeosubmanifoldprelimaries}).
\begin{lemma} \label{totgeolemma}
Let $z \in \calM = \mathbb{H}^d$, $g \in \T_z \calM$, $L = \exp_z(\{v \in \T_z \calM : \langle g, v \rangle \geq 0\})$, and $S = \partial L = \exp_z(\{v \in \T_z \calM : \langle g, v \rangle = 0\}).$
The \emph{half-space} $L$ is g-convex.  Its boundary $S = \partial L$ is a $(d-1)$-dimensional totally geodesic submanifold of $\mathbb{H}^d$.
\end{lemma}

References on g-convex optimization include \citep{udriste1994convex},~\citep{rapcsak1997smoothnonlinear},~\citep{bacak2014hadamard}, and~\citep[Ch.~11]{boumal2020intromanifolds}.
For $f \colon \calM \to \reals$, we denote its Riemannian gradient and Hessian by $\nabla f$ and $\nabla^2 f$, respectively (see, for example, Chapters 3 and 5 of~\citep{boumal2020intromanifolds}).
\begin{definition}
Let $\calM$ be a Hadamard manifold, and let $D \subseteq \calM$ be g-convex.
A function $f \colon D \rightarrow \reals$ is $\mu$-strongly g-convex if
$f(\gamma(t)) \leq (1-t) f(x) + t f(y) - \frac{\mu}{2} t(1-t) \dist(x, y)^2$ for all $x, y \in D$,
where $\gamma \colon [0, 1] \rightarrow \reals$ is the geodesic with $\gamma(0) = x, \gamma(1) = y$.
If $\mu = 0$, we say that $f$ is g-convex or ``nonstrongly'' g-convex if we wish to emphasize $\mu=0$.
If $\mu > 0$, we say $f$ is strongly g-convex.
\end{definition}
\noindent If $f$ is differentiable, $f$ is $\mu$-strongly g-convex in $\calM$ if and only if
$$f(y) \geq f(x) + \inner{\nabla f(x)}{\log_x(y)} + \frac{\mu}{2} \dist(x, y)^2 \quad \quad \forall x, y \in \calM.$$
If $f$ is twice differentiable, $f$ is $\mu$-strongly g-convex in $\calM$ if and only if $\nabla^2 f(x) \succeq \mu I$ $\forall x \in \calM$.

\begin{definition}
A vector $g \in \T_x \calM$ is a subgradient of $f \colon \calM \to \reals$ at $x$ if $f(y) \geq f(x) + \langle g, \log_x(y)\rangle$ for all $y\in\calM$.
The subdifferential $\partial f(x)$ of $f$ at $x$ is the set of all subgradients of $f$ at $x$.
\end{definition}
\noindent For a g-convex function, subdifferentials are never empty.  
A g-convex function is differentiable if and only if all its subdifferentials contain exactly one vector (the gradient)~\cite[Sec.~3.4]{udriste1994convex}.

\begin{definition}
Let $D \subseteq \calM$ be connected.
A function $f \colon \calM \rightarrow \reals$ is $M$-Lipschitz in $D$ if $|f(x) - f(y)| \leq M \dist(x,y)$ for all $x, y \in D$.
A differentiable function $f \colon \calM \rightarrow \reals$ is $L$-smooth in $D$ if $\norm{\nabla f(x) - P_{y \rightarrow x} \nabla f(y)} \leq L \dist(x, y)$ for all $x, y \in D$.
\end{definition}

If $f$ is $C^1$, then $f$ is $M$-Lipschitz in $\calM$ if and only if $\|\nabla f(x)\| \leq M$ for all $x \in \calM$.
Let $D$ be g-convex.
If $f$ is $C^1$ and $L$-smooth in $D$ then
$\left|f(y) - f(x) -  \inner{\nabla f(x)}{\log_x(y)}\right| \leq \frac{L}{2} \dist(x, y)^2$ for all $x, y \in D$.
If $f$ is $C^2$, $f$ is $L$-smooth in $D$ if $\norm{\nabla^2 f(x)} \leq L$ for all $x \in D$ (operator norm).

The supremum of g-convex functions is always g-convex~\citep[Cor.~2.7]{udriste1994convex}.  
Distance functions and squared distance functions to g-convex sets are always g-convex.  
More precisely, we have the following (see \citep[Lem. 2 in App. B]{alimisis2019continuoustime} or~\citep[Thm. 5.6.1]{jostbook}).

\begin{lemma} \label{lemmaboundhess}
Let $\calM$ have sectional curvatures in the interval $[K, 0]$.  If $D$ is a closed g-convex set, then $x\mapsto \dist(x, D)$ is g-convex and 1-Lipschitz globally, and $C^\infty$ at points $x \not\in D$.

Fix $z \in \calM$, and let $f(x) = \frac{1}{2} \dist(x, z)^2$.  Then $f$ is $C^{\infty}$, $\nabla f(x) = - \log_x(z)$ and $\nabla^2 f(x) \succeq I$ for all $x \in \calM$, and $\norm{\nabla^2 f(x)} \leq \zeta_{r \sqrt{-K}} \leq 1 + r \sqrt{-K}$ for all $x \in B(z, r)$.
\end{lemma}


\section{Extending the $\tilde \Omega(\zeta)$ lower bound: a reduction argument} \label{nonstronglyconvexcaseextension}
Building on~\citep{hamilton2021nogo}, \citet{bumpfctpaper} prove the lower bound $\tilde \Omega(\zeta_r)$ for the strongly g-convex problem~\ref{smoothstronglyproblem}.  
We extend this result to the problems~\ref{lowdimLipschitzproblem},~\ref{highdimLipschitzproblem} and~\ref{smoothproblem}.
For simplicity, in this section we only state the results for hyperbolic space.  
The more general result for Hadamard manifolds of bounded curvature (Theorem~\ref{theoremtheorem}) can be found in Appendix~\ref{thisguyappthis}.

We have the following simple consequence of Theorem 24 of~\citep{bumpfctpaper}.
%
\begin{lemma} \label{lemmausefulguysec3}
Let $d \geq 2$, $L > 0$, $r \geq 64$, and $\calM = \mathbb{H}^d$.
Define $\hat\epsilon = \frac{1}{2^{10} r}$ and $\mu = 64 \hat{\epsilon} L = \frac{L}{2^{4} r}$.
Let $\calA$ be a deterministic first-order algorithm.
There is a $C^\infty$ function $f \colon \mathbb{H}^d \to \reals$ with minimizer $x^* \in B(\xorigin, \frac{3}{4} r)$ such that running $\calA$ on $f$ yields iterates $x_0, x_1, \ldots$ satisfying $f(x_k) - f^* \geq 2 \hat \epsilon L r^2$ for all $k=0, 1, \ldots T-1$, where $T = \lfloor \frac{\zeta_r}{50 \log(64 \zeta_r)} \rfloor$.

Moreover, $f$ is $\mu$-strongly g-convex in $\calM$, and $\mu(12 \mathscr{R} + 3)$-Lipschitz and $\mu(12 \mathscr{R} + 9)$-smooth in the ball $B(\xorigin, \mathscr{R})$, where $\mathscr{R} = 2^9 r \log^2(r)$.
For all $x \not \in B(\xorigin, \mathscr{R})$, $f(x) = 3 \mu \dist(x, \xorigin)^2$.
\end{lemma}
The second paragraph in Lemma~\ref{lemmausefulguysec3} is not stated explicitly by~\citet{bumpfctpaper} but is apparent from their proof.
The proof of Lemma~\ref{lemmausefulguysec3} can be found in Appendix~\ref{thisguyappthis}.

Using a reduction, we next prove lower bounds for the function class $\mathcal{F}_{r, L/2, 0, L}^d$.
The idea is that, given a hard function $f$ from Lemma~\ref{lemmausefulguysec3}, we modify $f$ so that it remains the same inside $B(\xorigin, \mathscr{R})$, and outside $B(\xorigin, \mathscr{R})$ it is not strongly g-convex but is strictly g-convex, $\frac{L}{2}$-Lipschitz and $L$-smooth.
We know that $f(x)$ is proportional to $\dist(x, \xorigin)^2$ outside $B(\xorigin, \mathscr{R})$.
We modify $f$ there so that it behaves similarly to $\dist(x, \xorigin)$ instead.  This works because the function $x \mapsto \dist(x, \xorigin)$ is g-convex and (outside a sufficiently large ball surrounding $\xorigin$) that same function is $1$-Lipschitz and $2$-smooth on $\calM = \mathbb{H}^d$~\citep[Thm.~11.7]{lee2018riemannian}.

Given any $C^\infty$ function $f$, define the $C^\infty$ function $\tilde{f}_{\mathscr{R}} \colon \calM \rightarrow \reals$ by
\begin{align} \label{eqdefineftildeR}
\tilde{f}_{\mathscr{R}}(x) = u_{\mathscr{R}}\Big(\frac{1}{2} \dist(x, \xorigin)^2\Big) f(x),
\end{align}
where $u_{\mathscr{R}} \colon \reals \rightarrow \reals$ is a $C^\infty$ function which is $1$ on $(-\infty, \frac{1}{2} \mathscr{R}^2]$ and scales as $\sqrt{\frac{1}{2} \mathscr{R}^2 \cdot t^{-1}}$ for $t > \frac{1}{2} \mathscr{R}^2$ sufficiently large.
See Appendix~\ref{nonstronglyconvexboundsapphighlevel} for the definition of $u_{\mathscr{R}}$.
Suppose $f$ is a hard function from Lemma~\ref{lemmausefulguysec3}.  
Since $\tilde{f}_{\mathscr{R}} = f$ in $B(\xorigin, \mathscr{R})$, we know $\tilde{f}_{\mathscr{R}}$ is $\mu$-strongly g-convex, and $\mu(12\mathscr{R} + 3)$-Lipschitz and $\mu (12 \mathscr{R} + 9)$-smooth in $B(\xorigin, \mathscr{R})$.
Lemma~\ref{lemmausefulguysec3} also guarantees $f(x) = 3\mu \dist(x, \xorigin)^2$ for all $x \not \in B(\xorigin, \mathscr{R})$.  
In Appendix~\ref{nonstronglyconvexboundsapp}, we use this fact to show that $\tilde{f}_{\mathscr{R}}$ is $12 \mu \mathscr{R}$-Lipschitz, $24 \mu \mathscr{R}$-smooth and strictly g-convex outside of $B(\xorigin, \mathscr{R})$.
Using the definitions of $\mu, \mathscr{R}$ in Lemma~\ref{lemmausefulguysec3} as well as the bound $r \geq 64$, we conclude $\tilde{f}_{\mathscr{R}}$ is in $\mathcal{F}_{r, \tilde{L}/2, 0, \tilde{L}}^d$ where $\tilde{L} = L \cdot 2^{10} \log^2(r)$.

Given the oracle $\calO_f$ of any function $f$, we can use $\calO_f$ to emulate the oracle $\calO_{\tilde{f}_{\mathscr{R}}}$ using equation~\eqref{eqdefineftildeR} and the formula~\eqref{eqfortildefRgrad} for $\nabla \tilde{f}_{\mathscr{R}}$ (given in Appendix~\ref{nonstronglyconvexboundsapp}).
To prove a lower bound for an algorithm $\tilde{\calA}$ designed to minimize g-convex functions, we make $\tilde{\calA}$ interact with the oracle $\calO_{\tilde{f}_{\mathscr{R}}}$ (which we simulate using $\calO_f$).
This implicitly defines an algorithm $\calA$ which interacts with $\calO_f$.
Explicitly, the algorithm $\calA$ runs $\tilde{\calA}$ as a subroutine as follows.  
If $\tilde{\calA}$ outputs $x_k$, then ${\calA}$ queries $\calO_f$ at $x_k$, receives $(f(x_k), \nabla f(x_k))$ from $\calO_f$, and then passes $(\tilde{f}_{\mathscr{R}}(x_k), \nabla \tilde{f}_{\mathscr{R}}(x_k))$ to $\tilde{\calA}$, which it computes using $(f(x_k), \nabla f(x_k))$ and equations~\eqref{eqdefineftildeR},~\eqref{eqfortildefRgrad}.

Applying Lemma~\ref{lemmausefulguysec3} to the algorithm ${\calA}$, we know there is a function $f$ with minimizer $x^*$ so that $f(x_k) - f(x^*) \geq 2 \hat\epsilon L r^2$ for all $k \leq T-1$.
Since $\tilde{f}_{\mathscr{R}} = f$ in $B(\xorigin, \mathscr{R})$ and $\tilde{f}_{\mathscr{R}}$ is strictly g-convex on $\calM$, we know the minimizer of $\tilde{f}_{\mathscr{R}}$ is also $x^*$ and that $\tilde{f}_{\mathscr{R}}(x_k) - \tilde{f}_{\mathscr{R}}(x^*) = f(x_k) - f(x^*)$ if $x_k \in B(\xorigin, \mathscr{R})$.
In Appendix~\ref{technicalfacttofinishproofoftheorem1p5}, we show that $\tilde{f}_{\mathscr{R}}(x_k) - \tilde{f}_{\mathscr{R}}(x^*) \geq 2 \hat\epsilon L r^2$ if $x_k \not \in B(\xorigin, \mathscr{R})$.
Lastly, observe that (by design) if we run $\tilde{\calA}$ on the function $\tilde{f}_{\mathscr{R}}$ then we get exactly the sequence $x_0, x_1, \ldots, x_{T-1}$.
We have proven the following theorem.

\begin{theorem} \label{nicetheorem}
Let $d \geq 2$, $\tilde L > 0$, $r \geq 64$, and $\calM = \mathbb{H}^d$.
Define $\hat\epsilon = \frac{1}{2^{10} r}, \epsilon = \frac{1}{2^{18} \zeta_r \log^2(\zeta_r)}, \epsilon' = \frac{1}{2^{18} \log^2(\zeta_r)}$ and $L = \frac{\tilde L}{2^{10} \log^2(r)}$.
Let $\calA$ be any deterministic first-order algorithm.

There is a $C^\infty$ function $\tilde f \in \calF_{r, \tilde L / 2, 0, \tilde L}^d$ with unique minimizer $x^*$ such that running $\calA$ on $\tilde f$ yields iterates $x_0, x_1, \ldots$ satisfying 
$$\tilde f(x_k) - \tilde f(x^*) \geq 2 \hat{\epsilon} L r^2 \geq \epsilon' \cdot \frac{\tilde L}{2} r \geq \epsilon \cdot \frac{1}{2} \tilde L r^2$$
for all $k = 0, 1, \ldots, T-1$ where $T = \lfloor \frac{\zeta_r}{50 \log(64 \zeta_r)} \rfloor$.
\end{theorem}

Theorem~\ref{nicetheorem} shows that the $\tilde \Omega(\zeta_r)$ bound holds for problems ~\ref{lowdimLipschitzproblem},~\ref{highdimLipschitzproblem} and~\ref{smoothproblem}, provided $\epsilon$ is not too big.
Note that $\epsilon' = \tilde \Theta(1)$ and $\epsilon = \tilde \Theta(\frac{1}{\zeta_r})$ in Theorem~\ref{nicetheorem} (recall that $\epsilon \leq \Theta(\frac{1}{\zeta_r})$ for problem~\ref{smoothproblem}).

In Table~\ref{table:maintheorem}, the best known upper bounds all depend on $\zeta_r$.  Theorem~\ref{nicetheorem} shows that this dependence is unavoidable.
We know that a variant of Riemannian gradient descent for the smooth g-convex problem~\ref{smoothproblem} has complexity $O(\frac{1}{\epsilon})$.\footnote{\citet{zhang2016complexitygeodesicallyconvex} prove the complexity $O(\frac{\zeta_r}{\epsilon})$ for RGD for problem~\ref{smoothproblem}.
This can be improved to $O(\frac{1}{\epsilon})$ as follows.
In Appendix D of their paper, \citet{davidmr} provide a variant on RGD for constrained optimization, which has complexity $O({\kappa})$ for problem~\ref{smoothstronglyproblem}.  Using the reduction provided in Proposition~\ref{reductionstronglyconvextoconvex}, we find that a regularized version of that method has complexity $O(\frac{1}{\epsilon})$.}
Theorem~\ref{nicetheorem} shows that in the regime $\epsilon = \tilde\Theta(\frac{1}{\zeta_r})$, \emph{this variant of RGD is optimal} (up to log factors).
Contrast this with the Euclidean case, where gradient descent is \emph{never} optimal and acceleration is used to achieve the optimal complexity of $O(\frac{1}{\sqrt{\epsilon}})$.

\section{Lower bounds for Lipschitz g-convex optimization: a resisting oracle argument} \label{secnonsmoothlowerbounds}

In this section we prove a lower bound for the Lipschitz g-convex problems~\ref{lowdimLipschitzproblem} and~\ref{highdimLipschitzproblem}.
For Euclidean space $\mathbb{R}^d$, such lower bounds are constructed using a maximum of affine functions $x \mapsto \max_{i=1, \ldots, d}\{\langle s_i e_i,  x \rangle\}$, where $(e_i)_{i=1}^d$ are the standard basis vectors and $s_i \in \{+1, -1\}$~\citep[Sec~7.4]{nemirovskibook}.
Before the algorithm makes any queries, there are $2^d$ possible minimizers, $\frac{r}{\sqrt{d}}(s_1, \ldots, s_d)$, corresponding to the vertices of a $d$-dimensional hypercube inscribed in a sphere of radius $r$.  After each query by the algorithm, the set of possible minimizers is halved.

We use a similar technique here.
However, we need a replacement for the affine functions used in Euclidean space.
Affine functions are both g-convex and g-concave.
On a Riemannian manifold, a function is affine if its Hessian vanishes identically.
\citet{Innami1982} shows that on most manifolds (including hyperbolic space), non-constant affine functions do not exist.\footnote{In fact, more can be said: on a hyperbolic space, if $f$ is a $C^3$ g-convex function whose Hessian vanishes at even just a single point $x$ then necessarily $\nabla f(x) = 0$, and so $f$ cannot be affine unless it is constant -- see Proposition~\ref{blahblah}.}
Moreover, the functions $x \mapsto \langle g, \log_{y}(x) \rangle$ do not work because they are not g-convex (see Appendix~\ref{appinterpolation}).

Our main idea is to replace the linear functions by \emph{distance functions to $(d-1)$-dimensional totally geodesic submanifolds, and suitably arrange those totally geodesic submanifolds}.

\begin{theorem} \label{nonsmoothlowerbound}
Let $r > 0$, $M > 0$ and $T$ a positive integer.  For every deterministic algorithm $\mathcal{A}$, there exists an $M$-Lipschitz g-convex function $f \in \calF_{r, M, 0, \infty}$ such that $\mathcal{A}$ requires at least $T$ queries to find a point $x \in \calM$ with $f(x) - f^* \leq M r \cdot  \frac{1}{2 \zeta_r \sqrt{T}}$.
\end{theorem}

\begin{proof}
Without loss of generality, we can assume $M=1$.  Let $d = T$ and $(e_i)_{i=1}^d$ be an orthonormal basis of $\T_{\xorigin} \calM$.  
Let $a > 0$ be such that $\tanh(a)/\tanh(r) = 1/\sqrt{d}$, and define $\delta = \frac{a}{2 T}$.
For $s \in \{+1, -1\}$ and $i = 1, \ldots, d$ define 
$$z_i^s = \exp_{\xorigin}(a s e_i), \quad H_i^s = \{v \in \T_{z_i^s} \calM : \langle \log_{z_i^s}(\xorigin), v\rangle = 0 \}, \quad S_i^s = \exp_{z_i^s}(H_i^s)$$
and define the $1$-Lipschitz g-convex functions
$$h_i^s(x) = \dist(x, S_i^s) - \dist(\xorigin, S_i^s) = \dist(x, S_i^s) - a.$$
These functions are g-convex by Lemmas~\ref{totgeolemma} and~\ref{lemmaboundhess}.

We build the hard function for $\mathcal{A}$ iteratively as a max of functions $h_i^s$ (shifted appropriately so that the subdifferential at $x_k$ always consists of exactly one vector).  Let $\mathcal{I}_0 = \{1, \ldots, d\}$ and $\calH_0 = \emptyset$.  For $k=0, \ldots, T-1$, inductively define 
\begin{equation}\label{inductivedefinitions}
\begin{split}
\calH_k = ((F_\ell, x_\ell, g_\ell))_{\ell=0}^{k-1},&\quad \quad x_k = \calA_k(\calH_k),
\\
(i_k, s_k) \in {\arg \max}_{i \in \mathcal{I}_k, s \in \{+1, -1\}} h_i^s(x_k), &\quad \quad \mathcal{I}_{k+1} = \mathcal{I}_k \setminus \{i_k\},
\\
f_k(x) = \max_{\ell \in \{0, \ldots, k\}}\{ h_{i_\ell}^{s_\ell}(x) - \ell \delta\}, &\quad \quad F_k = f_k(x_k), g_k \in \partial f_k(x_k).
\end{split}
\end{equation}

Let $f$ be the $1$-Lipschitz g-convex function $f_{T-1}.$
We claim that running $\mathcal{A}$ on $f$ produces exactly the sequence of iterates $x_0, \ldots, x_{T-1}$.
It suffices to show $f_{T-1}(x) = f_k(x)$ for all $k$ and all $x$ in $B(x_k, \frac{\delta}{2})$.
For all $\ell = k+1, \ldots, T-1$ we know
$f_k(x_k) \geq h_{i_k}^{s_k}(x_k) - k \delta \geq h_{i_\ell}^{s_\ell}(x_k) - \ell \delta + \delta$,
using the definition of
$(i_k, s_k)$.
Therefore, $f_k(x_k) \geq \delta + \max_{\ell \in \{k+1, \ldots, T-1\}} \{h_{i_\ell}^{s_\ell}(x_k) - \ell \delta\}$.
We conclude 
$f_k(x) \geq \max_{\ell \in \{k+1, \ldots, T-1\}} \{h_{i_\ell}^{s_\ell}(x) - \ell \delta\}$ for all $x \in B(x_k, \frac{\delta}{2})$, since $f_k$ is $1$-Lipschitz.
Hence,
$f_{T-1}(x) = \max\big\{f_k(x), \max_{\ell \in \{k+1, \ldots, T-1\}} \{h_{i_\ell}^{s_\ell}(x) - \ell \delta\}\big\} = f_k(x)$ for all $ x \in B(x_k, \delta / 2).$

It remains to lower bound $f(x_k) - f^*$ for all $k$.
Define $x^* = \exp_{\xorigin}(\frac{r}{\sqrt{d}} \sum_{k=0}^{T-1} \nabla h_{i_k}^{s_k}(\xorigin))$.
We know that $(\nabla h_{i_k}^{s_k}(\xorigin))_{k = 0}^{T-1}$ forms an orthonormal basis of $\T_{\xorigin} \calM$, because we defined $\mathcal{I}_{k+1} = \mathcal{I}_k \setminus \{i_k\}$.  Therefore, $\dist(\xorigin, x^*) = r$.
In Appendix~\ref{geoproofinlowerbound}, we show that $x^* \in \bigcap_{k=0}^{T-1} S_{i_k}^{s_k}$ using the hyperbolic law of cosines~\citep[Ch.~3.5]{ratcliffe2019hyperbolic}. 
From there, we know that $\dist(x^*, S_{i_k}^{s_k}) = 0$ for all $k$, so 
$f(x^*) = \max_{k \in \{0, \ldots, T-1\}} \{-a - k \delta\} = -a.$
Also, $f_{T-1}(x) \geq h_{i_0}^{s_0}(x) \geq -a$ for all $x \in \calM$.
We conclude that $x^* \in B(\xorigin, r)$ is a global minimizer of $f$ on $\calM$ with $f^* = f(x^*) = -a$.


Next, we lower bound $f(x_k)$.
For $i \in \{1, \ldots, d\}$ and $s \in \{+1, -1\}$, g-convexity of $h_i^s$ and $\nabla h_i^s(\xorigin) = -s e_i$ imply 
$$\{x \in \calM : h_i^s(x) < h_i^s(\xorigin) = 0\} \subseteq \exp_{\xorigin}(\{v \in \T_{\xorigin} \calM : \langle -s e_i , v \rangle < 0\}).$$
Since $\bigcup_{s \in \{+1,-1\}} (\exp_{\xorigin}(\{v \in \T_{\xorigin} \calM : \langle -s e_i , v \rangle \geq 0\})) = \calM$, we find that $\bigcup_{s \in \{+1, -1\}} \{x \in \calM : h_i^s(x) \geq 0\} = \calM.$
Therefore, 
$h_{i_k}^{s_k}(x_k) \geq \max_{s \in \{+1, -1\}} h_{i_k}^s(x_k) \geq 0$ for all $k = 0, \ldots, T-1.$

Therefore, $f(x_k) = f_{T-1}(x_k) \geq h_{i_k}^{s_k}(x_k) - k \delta \geq -(T-1) \delta \geq -\frac{a}{2}$ for all $k = 0, \ldots, T-1.$
Combining this upper bound on $f^*$ and lower bound on $f(x_k)$, we conclude
$f(x_k) - f^* \geq - \frac{a}{2} - (-a) = \frac{a}{2} = \frac{1}{2}r \cdot r^{-1} \arctanh(\tanh(r) / \sqrt{T}) \geq r \frac{1}{2 \zeta_r \sqrt{T}}.$
\end{proof}
\begin{remark}\label{remarkresult}
Theorem~\ref{nonsmoothlowerbound} says that if $\epsilon = \frac{1}{2 \zeta_r \sqrt{d}}$, then $T_{\epsilon, r, d} \geq d = \frac{d}{\log(2 \zeta_r \sqrt{d})} \log(\epsilon^{-1})$.
The lower bound in Theorem~\ref{nonsmoothlowerbound} implies the lower bound $q_{d, r} \geq \frac{d}{\log(2 \zeta_r \sqrt{d})}$ for the low-dimensional problem~\ref{lowdimLipschitzproblem}.
\end{remark}

\section{Lower bounds for smooth g-convex optimization: a Moreau smoothing argument} \label{secsmoothlowerbounds}
In this section we prove lower bounds for the smooth g-convex problems~\ref{smoothproblem} and~\ref{smoothstronglyproblem}.
In smooth convex optimization, the most well-known technique for proving a lower bound is building a so-called ``worst function in the world''~\citep[2.1.2]{nesterov2004introductory}.
However, there is perhaps the less well-known technique of smoothing convex Lipschitz functions, usually via \emph{Moreau envelopes}~\citep{guzmannemirovski}.  
We adopt the smoothing approach because this technique works far better than the other on manifolds, and the smoothing technique seems more general.
That is, given a lower bound construction  for Lipschitz g-convex optimization which improves over the one we present in Section~\ref{secnonsmoothlowerbounds}, smoothing is likely to produce an improved lower bound in the smooth setting.

The (Riemannian) Moreau envelope of a function $f \colon \calM \to \reals$ is
the function defined as follows:
\begin{align}\label{moreauenv}
f_\lambda \colon \calM \to \reals, \quad \quad f_\lambda(x) = \inf_{y \in \calM} \Big\{f(y) + \frac{1}{2 \lambda} \dist(x,y)^2\Big\}.
\end{align}
In convex analysis, the Moreau envelope is closely related to Fenchel duality.  There is no especially satisfying theory of Fenchel duality on Hadamard manifolds (but see~\citep{bergmannfenchel,hirai1}).
However, the Riemannian Moreau envelope can be studied directly without reference to duality.
It satisfies the following properties, proved in~\citep{azagra1,azagra2}.  For completeness, we provide a brief (simpler) proof in Appendix~\ref{moreauenvelopeapp}.

\begin{lemma} \label{moreauenvelope}
Let $\calM$ be a Hadamard manifold with curvature lower bounded by $-1$. Let $f \colon \calM \to \reals$ be $1$-Lipschitz and g-convex.
Then, $f_\lambda \colon \calM \to \reals$ with $\lambda > 0$ is g-convex, $1$-Lipschitz, and $\frac{1}{\tanh(\lambda)}$-smooth.  Moreover, $f(x) \geq f_\lambda(x) \geq f(x) - \lambda$ for all $x \in \calM$, and the value of $f_{\lambda}$ at $x \in \calM$ is determined by the values of $f$ in the ball $B(x, \lambda)$:
$f_\lambda(x) = \min_{y \in B(x, \lambda)} \Big\{f(y) + \frac{1}{2\lambda} \dist(x,y)^2\Big\}.$
\end{lemma}


We now smooth the construction from Section~\ref{secnonsmoothlowerbounds}, following~\citep{guzmannemirovski}. 

\begin{theorem} \label{thmsmoothlowerbound}
Let $r > 0$, $L > 0$ and $T$ a positive integer.  For every deterministic algorithm $\mathcal{A}$, there exists an $L$-smooth g-convex function $f \in \calF_{r, \infty, 0, L}$ such that $\mathcal{A}$ requires at least $T$ queries to find a point $x \in \calM$ with $f(x) - f^* \leq \frac{1}{2} L r^2 \cdot \frac{1}{8 \zeta_r^2 T^2}$.
\end{theorem}
\begin{proof}
Define $d, e_i, a, \delta, z_i^s, S_i^s, h_i^s$ as in the proof of Theorem~\ref{nonsmoothlowerbound}.
Define $\lambda = \frac{\delta}{4}$, $\mathcal{I}_0 = \{1, \ldots, d\}$ and $\calH_0 = \emptyset$.  For $k=0, \ldots, T-1$, inductively define $\calH_k, x_k, (i_k, s_k), \mathcal{I}_{k+1}, f_k$ as in equation~\eqref{inductivedefinitions}, and let $F_k = f_{\lambda, k}(x_k), g_k = \nabla f_{\lambda, k}(x_k)$ where $f_{\lambda, k}$ is the Moreau envelope of $f_k$ with parameter $\lambda$.

Let $f$ equal the $1$-Lipschitz g-convex function $f_{\lambda, T-1}.$
We claim that running $\mathcal{A}$ on $f$ produces exactly the sequence of iterates $x_0, \ldots, x_{T-1}$.
It suffices to show for each $k$ that $f_{\lambda, T-1}(x) = f_{\lambda, k}(x)$ for all $x \in B(x_k, \frac{\delta}{4})$.
Lemma~\ref{moreauenvelope} implies that the values of $f_{\lambda, T-1}$ and $f_{\lambda, k}$ in $B(x_k, \frac{\delta}{4})$ are determined, respectively, by the values of $f_{T-1}, f_k$ in $B(x_k, \lambda + \frac{\delta}{4}) = B(x_k, \delta / 2)$.
However, in the proof of Theorem~\ref{nonsmoothlowerbound}, we already showed for each $k$ that
$f_{T-1}(x) = f_{k}(x)$ for all $x \in B(x_k, \delta / 2)$.
We conclude $f_{\lambda, T-1}(x) = f_{\lambda, k}(x)$ for all $x \in B(x_k, \frac{\delta}{4})$, as desired.

Let us lower bound the suboptimality gap $f(x_k) - f^*$ for each $k$.  
Take $x^* \in B(\xorigin, r)$ as defined in Theorem~\ref{nonsmoothlowerbound}.
From the proof of Theorem~\ref{nonsmoothlowerbound}, we know $x^*$ is a global minimizer of $f_{T-1}$ with $f_{T-1}(x^*) = -a$.
It is immediate from equation~\eqref{moreauenv} that $x^*$ is therefore a global minimizer of $f = f_{\lambda, T-1}$ with $f^* = f_{\lambda, T-1}(x^*) = -a$.
From the proof of Theorem~\ref{nonsmoothlowerbound}, $f_{T-1}(x_k) \geq -(T-1) \delta$.  
So $f(x_k) \geq f_{T-1}(x_k) - \lambda \geq -T \delta \geq -\frac{a}{2}$ by Lemma~\ref{moreauenvelope}.
Hence for all $k$, 
\begin{align*}
f(x_k) - f^* &\geq \frac{a}{2} = \frac{1}{2} \frac{L r^2}{T^2} \frac{a}{Lr^2 / T^2} = \frac{1}{2} \frac{L r^2}{T^2} \frac{T^2}{r^2} a \tanh\Big(\frac{a}{8 T}\Big) \\
&= \frac{1}{2} \frac{L r^2}{T^2} \bigg[\frac{T^2}{r^2} \arctanh\Big(\frac{\tanh(r)}{\sqrt{T}}\Big) \tanh\Big(\frac{1}{8 T} \arctanh\Big(\frac{\tanh(r)}{\sqrt{T}}\Big)\Big)\bigg] \geq \frac{1}{2} \frac{L r^2}{T^2} \bigg[\frac{1}{8 \zeta_r^2}\bigg],
\end{align*}
where the smoothness constant of $f$ equals $L = \frac{1}{\tanh(\lambda)} = \frac{1}{\tanh(\frac{a}{8 T})}$, by Lemma~\ref{moreauenvelope}.
\end{proof}

Proposition~\ref{reductionstronglyconvextoconvex} in Appendix~\ref{reductionblah} provides a reduction between the smooth nonstrongly and strongly g-convex settings (see also~\citep{martinezrubio2021global}).
Using this proposition, the lower bound in Theorem~\ref{thmsmoothlowerbound} extends to an $\Omega(\frac{\sqrt{\kappa}}{\log(\zeta_r \kappa)})$ lower bound for the smooth strongly g-convex problem~\ref{smoothstronglyproblem}.
\begin{corollary}\label{thisremark}
Let $r > 0$ and $\kappa > \zeta_{r}$.  
For $\epsilon = \frac{2}{\kappa - \zeta_r}$, $T_{\epsilon, r, \kappa} \geq \frac{\sqrt{\kappa - \zeta_{r}}}{8 \log(\kappa [3 + \zeta_{r}^2])} \log(\epsilon^{-1})$.
In particular, $q_{\kappa, r} = \sup_{\epsilon \in (0,1)} \{\frac{T_{\epsilon, r, \kappa}}{\log(\epsilon^{-1})}\} \geq \frac{\sqrt{\kappa - \zeta_{r}}}{8 \log(\kappa [3 + \zeta_{r}^2])}$.
\end{corollary}

Lastly, we mention that if the optimization is carried out in a very small region, e.g., $r \leq O(\sqrt{\epsilon})$, then smooth g-convex optimization reduces to Euclidean convex optimization, and in particular upper and lower bounds from Euclidean space carry over.  See Appendix~\ref{carryovereuc} for details.

\subsection{Why assume $\epsilon \leq O(\zeta^{-1})$?} \label{lrvslrsquared}
In problem~\ref{smoothproblem} we assumed $\epsilon \leq \frac{8}{\zeta_r}$.  The following proposition explains why this is justified.
\begin{proposition}\label{usefulguy}
Let $f \colon \mathbb{H}^d \to \reals$ be differentiable, globally g-convex, and $L$-smooth in $B(\xorigin, r)$.
Suppose $\nabla f(x^*) = 0$ and $x^* \in B(\xorigin, r)$.  
Then $f(\xorigin) - f^* \leq \frac{1}{2} L r^2 \cdot \frac{8}{\zeta_r}.$
\end{proposition}
\clearpage 
\noindent Proposition~\ref{usefulguy} follows from the following fact expressed in Proposition~\ref{lrvslrsquaredprop}: if $f \colon \mathbb{H}^d \to \reals$ (not necessarily g-convex) is $L$-smooth (in a large enough region) then it is also $\frac{\pi\sqrt{5}}{2} L$-Lipschitz.\footnote{The proof of Proposition~\ref{lrvslrsquaredprop} is due to~\citet[pgs.~25, 33]{petrunin2023pigtikal}.  We fill in the details in Appendix~\ref{proofsfromgeometryinfluencescostfct}.
This fact closely mirrors the following result: if $f$ is $\rho$-Hessian-Lipschitz (in a large enough region) then it is $2\rho$-Lipschitz and $\sqrt{8}\rho$-smooth~\citep[Remark 3.2]{criscitiello2020accelerated}.}
We can also use this fact to prove that $\kappa \geq \Omega(\zeta_r)$ even holds for a g-convex domain of diameter $2 r$ which is \emph{not} a ball, as long as the domain is not too \emph{eccentric}: see Proposition~\ref{improvementonkappamorethanr}.

\section{Lower bound $\Omega(\frac{\zeta}{\epsilon^2})$ for subgradient descent: a worst-function-in-the-world argument} \label{polyaknotapp}
For problem~\ref{highdimLipschitzproblem}, we showed the lower bound $\tilde \Omega(\zeta_r + \frac{1}{\zeta_r^2 \epsilon^2})$ in Section~\ref{secnonsmoothlowerbounds}.
This lower bound does not match the upper bound $O(\frac{\zeta_r}{\epsilon^2})$.
We {conjecture} that this upper bound is optimal.
To support this conjecture, in Appendix~\ref{polyak} we prove an $\Omega(\frac{\zeta_r}{\epsilon^2})$ lower bound for subgradient descent for problem~\ref{highdimLipschitzproblem}.
For simplicity, we focus on subgradient descent \emph{with Polyak step size} (see Appendix~\ref{upperboundforsubgradientdescent}), although we expect the proof extends to other similar step size choices, like the one in~\citep{zhang2018estimatesequence}.  

\begin{corollary} \label{corthisguysubrgd}
Let $\epsilon \leq \frac{1}{4 \sqrt{2}}$ and $M > 0$.  Let $T = d = \lfloor \frac{1}{32} \frac{\zeta_r}{\epsilon^2} \rfloor$.
There is a globally g-convex and $M$-Lipschitz function $f \colon \calM = \mathbb{H}^d \to \reals$ so that the following holds:
There is an oracle for $f$ so that running subgradient descent with Polyak step size with that oracle produces iterates $x_0, x_1, \ldots$ satisfying $f(x_k) - f^* \geq \epsilon M r$ for all $k = 0, \ldots, T-1$.
\end{corollary}

We actually prove the $\Omega(\frac{\zeta_r}{\epsilon^2})$ lower bound for a \emph{large class of (intuitively reasonable) algorithms}, not just subgradient descent, see Theorem~\ref{lowerboundforsubrgd}.
Corollary~\ref{corthisguysubrgd} is a consequence of this result.

Corollary~\ref{corthisguysubrgd} and Theorem~\ref{lowerboundforsubrgd} are proven using a novel ``worst function in the world''
$$f(x) = \dist(x, x^*) + \max_{i=0, \ldots, d-2}\Big\{\frac{1}{4 \epsilon} \dist(x, L_i) \Big\},$$
for an appropriately chosen $x^* \in \partial B(\xorigin, r)$, and a particular choice of (g-convex) half-spaces $L_k = \exp_{y_k}(\{v \in \T_{y_k} \calM : \langle V_{k}, v \rangle \geq 0 \})$.
If an algorithm were simply minimizing the ``base function'' $x \mapsto \dist(x, x^*)$, then one subgradient query would reveal $x^*$.
The $\max$-term in $f$ is carefully constructed to perform a correction to the subgradients of this base function so that one query of $f$ reveals little information about $x^*$.

The key geometric fact leading to the $\Omega(\frac{\zeta_r}{\epsilon^2})$ lower bound is the following.  
Consider a right hyperbolic triangle with vertices $x^*, x_0, x_1$, where the right angle is at $x_1$, and the angle at $x_0$ is $\theta \in (0, \frac{\pi}{2})$.
Define the sidelengths $\dist(x_0, x^*) = r_0$ and $\dist(x_1, x^*) = r_1$.
Let $\epsilon = \cos(\theta)$.
Then, in hyperbolic space $r_1^2 \approx (1-\frac{\epsilon^2}{\zeta_{r_0}}) r_0^2$ for $\epsilon$ sufficiently small, whereas in Euclidean space $r_1^2 = (1-\epsilon^2) r_0^2$.  The side $r_1$ is significantly longer in hyperbolic space than in Euclidean space.

\section{Cutting-planes game: a width-bounded-separators argument} \label{feasibilitygame}
We conjecture that the optimal complexity for the low-dimensional Lipschitz g-convex problem~\ref{lowdimLipschitzproblem} is $\Theta(\zeta_r d)$.
In this section we consider a ``cutting-planes game'' which serves as a proxy for problem~\ref{lowdimLipschitzproblem}, and we prove an $\tilde \Omega(\zeta_r d)$ lower bound for this game.  
We take $\calM = \mathbb{H}^d$ throughout.
Let us motivate the game.
Consider an algorithm $\calA$ minimizing a g-convex $1$-Lipschitz function $f$ with a unique minimizer $x^*$ in $B(\xorigin, r)$, whose goal is to find a point $x$ with $f(x) -f^* \leq \epsilon r$, $\epsilon \in (0,1)$.
If $\calA$ queries $x_k \in \calM$, it learns $F_k = f(x_k)$ and a subgradient $g_k \in \partial f(x_k)$ satisfying $\|g_k\| \leq 1$.  Since $f$ is g-convex, we know 
$f(x) \geq F_k + \langle g_k, \log_{x_k}(x) \rangle$ for all $x$.  In particular, taking $x = x^*$ we find 
$\langle g_k, \log_{x_k}(x^*) \rangle \leq 0$. 
Moreover,
$0 \geq F_k - f^* + \langle g_k, \log_{x_k}(x^*) \rangle \geq F_k - f^* - \dist(x^*, S_k)$,
where $S_k = \exp_{x_k}(\{v \in \T_{x_k} \calM : \langle v, g_k \rangle = 0\}).$
The last inequality follows from the g-convexity of $x \mapsto \dist(x, S_k)$ and $-g_k$ is a subgradient of $x \mapsto \dist(x, S_k)$.
Therefore, the algorithm solves this g-convex problem if $\dist(x^*, S_k) \leq \epsilon r$ for some $k$.

In the cutting-planes game the algorithm is given a ball $B(\xorigin, r)$ which contains a point $x^*$.  The algorithm has access to the following (weaker) oracle: when the algorithm queries $x_k \in \calM$ it receives a unit vector $g_k \in \T_{x_k} \calM$ such that $\langle \log_{x_k}(x^*), g_k \rangle \leq 0$.
The algorithm's goal is to find $x_k$ so that $\dist(x^*, S_k) \leq \epsilon r$ where $S_k = \exp_{x_k}(\{v \in \T_{x_k} \calM : \langle v, g_k \rangle = 0\})$.
Equivalently, the algorithm's goal is to find $x_k$ such that the ball $B(x^*, \epsilon r)$ \emph{intersects} $S_k$.


Evidently, an algorithm which can solve the cutting-planes game can also solve the Lipschitz g-convex optimization problem.
\citet[\S3.2.6]{nesterov2004introductory} calls such methods ``cutting-plane schemes,'' and these include the center of gravity and ellipsoid methods.
A cutting-plane scheme ``forgets'' that it is minimizing a g-convex function, and simply plays the cutting-planes game.
In this sense, our lower bound below serves as a lower bound for cutting-plane schemes.

Before we proceed, we recall that the complexity of the cutting-planes game is $\Theta(d)$ for Euclidean space.
For hyperbolic space, the best known upper bound is $O(\zeta_r d^2)$~\citep{rusciano2019riemannian}, and the best known lower bound was $\tilde \Omega(\zeta_r)$ (from Section~\ref{nonstronglyconvexcaseextension}).  We improve this lower bound to $\tilde \Omega(\zeta_r d)$.

\begin{theorem} \label{lowerboundforfeasibilityproblem}
Let $d \geq 3, r \geq 1280 (d-1) \log(d)$ and $\epsilon = \frac{1}{320 (d-1)}$.  For every algorithm, there is an $x^* \in B(\xorigin, r)$ and a sequence of oracle responses $(g_k)_{k \geq 0}$ such that the algorithm needs at least $T = \frac{1}{32}(d-1) r$ queries to find a point $x_k$ so that $\dist(x^*, S_k) \leq \epsilon r$ (or equivalently $S_k \cap B(x^*, \epsilon r) \neq \emptyset$), where $S_k = \exp_{x_k}(\{v \in \T_{x_k} \calM : \langle v, g_k \rangle = 0\})$.
\end{theorem}


In the Euclidean case, the lower bound is constructed by tesselating space by hypercubes and choosing oracle hyperplanes which are parallel to the sides of the cube~\citep[\S3.2.5]{nesterov2004introductory}.  This construction does not generalize, so we propose a new strategy relying on ``width-bounded separators'' which first appeared in complexity theory literature~\citep{widthboundedseparator}, and whose properties were worked out by~\citet{hyperbolicintersectiongraphs} for hyperbolic space.  We use these to show Theorem~\ref{lowerboundforfeasibilityproblem}.

Let us sketch the idea behind the proof of Theorem~\ref{lowerboundforfeasibilityproblem}.
We start with a packing of $B(\xorigin, r)$ with $e^{\Theta(d r)}$ balls of radii $\epsilon r$.  The initial set of candidates for $x^*$, denoted $A_0$, consists of the centers of these balls.  At each iteration $k$, we maintain a set of possible remaining candidates for $x^*$, denoted $A_k$.
When the algorithm queries $x_k$, we show that there exists a unit vector $\tilde g_k \in \T_{x_k} \calM$ such that $S_k = \exp_{x_k}(\{v \in \T_{x_k} \calM : \inner{v}{\tilde g_k} = 0\})$ intersects at most $\frac{1}{2}|A_k|$ candidate balls (see Lemma~\ref{lemmafrombak}).
We then choose $g_k \in \{\tilde g_k, - \tilde g_k\}$ so that the next set of candidate minimizers $A_{k+1} \subseteq A_k$ satisfies $|A_{k+1}| \geq \frac{1}{2} (|A_k| - \frac{1}{2} |A_k|) = \frac{1}{4}|A_k|$.
Hence, after $T = \Theta(\log(|A_0|)) =  \Theta(\log(e^{d r})) = \Theta(d r)$ iterations $A_T$ will be nonempty.  We take $x^*$ to be an element of $A_T$.
The complete proof of Theorem~\ref{lowerboundforfeasibilityproblem} can be found in Appendix~\ref{feasprobleapplowerbound}.

\section{Geodesically convex interpolation: geometric obstructions} \label{interpolation}
To build lower bounds for g-convex optimization, we must choose a collection of function values and gradients $(F_i, x_i, g_i)_{i=1}^N$, and be able to build g-convex functions which interpolate those function values and gradients.
We say the data $(F_i, x_i, g_i)_{i=1}^N$ is interpolated by a g-convex function $f \colon \calM \rightarrow \reals$ if $f(x_i) = F_i$ and $g_i \in \partial f(x_i)$ for all $i$.
If the data $(F_i, x_i, g_i)_{i=1}^N$ is interpolated by the $\mu$-strongly g-convex function $f$, $\mu \geq 0$, they must satisfy the ``necessary conditions''
\begin{align} \label{necessaryguyss}
F_j \geq F_i + \langle g_i, \log_{x_i}(x_j)\rangle + \frac{\mu}{2} \dist(x_i, x_j)^2 \quad \forall i, j \in \{1, \ldots, N\}.
\end{align}
If $\calM$ is a Euclidean space these necessary conditions are sufficient for interpolation by a $\mu$-strongly convex function.  This forms the basis for new insights into the complexity of convex optimization~\citep{taylorinterpolation2016}.
Unfortunately, in the g-convex case these conditions are \emph{not sufficient} even for just three points.
\begin{proposition} \label{nointerpolation}
Let $\calM = \mathbb{H}^2$.
There exists data $(F_i, x_i, g_i)_{i=1}^3$ such that $F_j \geq F_i + \inner{g_i}{\log_{x_i}(x_j)}$ for all $i, j \in \{1,2,3\}$,
yet this data cannot be interpolated by any g-convex function.
\end{proposition}
The geometric fact underlying Proposition~\ref{nointerpolation} is that the altitude of an isoceles hyperbolic triangle is shorter than the altitude in a corresponding Euclidean triangle (see Proposition~\ref{thisguyonemore} in Appendix~\ref{appinterpolation}).

We do, however, have the following sufficient conditions, whose proof relies on the fact that the function $x\mapsto \langle g, \log_y(x) \rangle$ is $\|g\|$-smooth on $\mathbb{H}^d$ (see Lemma~\ref{boundedhessaffine} in Appendix~\ref{appinterpolation}).

\begin{proposition}\label{suffcondinteprolationguy}
Let $\calM = \mathbb{H}^d$.
Consider the data $(F_i, x_i, g_i)_{i=1}^N$.
Assume inequalities~\eqref{necessaryguyss} hold, and $\|g_i\| \leq \frac{\mu}{2}$ for all $i$.
Then $(F_i, x_i, g_i)_{i=1}^N$ is interpolated by a $\frac{\mu}{2}$-strongly g-convex function.
\end{proposition}

In Appendix~\ref{appinterpolation}, we take a more careful look at g-convex interpolation.  For example, in Proposition~\ref{necsuffconditions}, we give necessary and sufficient conditions for interpolation (but they seem to be of limited use), and link the difficulty of g-convex interpolation to the fact that g-convexity does not naturally fall into the framework of abstract convexity~\citep{abstractconvexity}.

\bibliography{../bibtex/boumal}

\appendix

\section{Proving an $\Omega(\frac{\zeta}{\epsilon^2})$ lower bound for a class of algorithms including subgradient descent} \label{polyak}

This section provides a proof of Corollary~\ref{corthisguysubrgd} from Section~\ref{polyaknotapp}.

\subsection{Totally geodesic submanifolds and the hyperboloid model} \label{totgeosubmanifoldprelimaries}
Let $\calM$ be a Hadamard manifold and $S \subset \calM$ a Riemannian submanifold.
If $x \in S$ and $v \in \T_x \calM$, we say $v$ is tangent to $S$ if $v \in \T_x S$, and $v$ is orthogonal to $S$ if $v$ is orthogonal to $\T_x S$, i.e., $\langle v, w \rangle = 0$ for all $w \in \T_x S$.
Below, parallel transport $P^\gamma$ along a curve $\gamma$, and the exponential map $\exp_x$ both denote those operations \emph{in the ambient space} $\calM$.
\begin{definition}
A Riemannian submanifold $S$ is totally geodesic if it is connected, complete, and all geodesics in $S$ are geodesics in the ambient manifold $\calM$.
\end{definition}
We have the following equivalent characterizations, see~\citep[Prop.~8.12]{lee2018riemannian},~\citep[Ch.~4]{oneill},~\citep[Ch.~XI,~XIV]{lang1999fundamentals}.
\begin{lemma}
Let $S$ be a complete, connected Riemannian submanifold of a Hadamard manifold $\calM$.
The following are equivalent.
\begin{itemize}
\item $S$ is totally geodesic.
\item The second fundamental form of $S$ vanishes.
\item If $(x, v) \in \T S$, then $\exp_x(t v) \in S$ for all $t \in \reals$.
\item If $x, y \in S$, then the geodesic through $x, y$ is contained in $S$.
\end{itemize}
\end{lemma}
\begin{lemma}\label{usefullemmaaboutparalleltransportintotgeo}
Let $S$ be a (complete, connected) totally geodesic submanifold of a Hadamard manifold $\calM$.  Then:
\begin{itemize}
\item If $x, y \in S$, then $\log_x(y) \in \T_x S$.
\item Let $\gamma \colon \reals \to S$ be a smooth curve in $S$ with $\gamma(0) = x$, and $v \in T_x S$.
Then the parallel transport of $v$ along $\gamma$ is the same in $S$ and $\calM$.
\item Let $\gamma \colon \reals \to S$ be a smooth curve in $S$ with $\gamma(0) = x$, and let $v \in \T_x S$ and $w \in \T_x \calM$ orthogonal to $\T_x S$.  Then $P_{0 \to t}^\gamma v \in \T_{\gamma(t)} S$, and $P_{0 \to t}^\gamma w$ is orthogonal to $\T_{\gamma(t)} S$.
\end{itemize}
\end{lemma}

Now let us consider $\calM = \mathbb{H}^d$.
To make our lower bound construction as concrete as possible, in the next section we shall give explicit formulae for certain geometric objects using the hyperboloid model of hyperbolic space.  
We introduce the relevant definitions and formulae in this section (see also~\citep[Ch.~3]{ratcliffe2019hyperbolic}).

Let $e_0, \ldots, e_d$ denote the standard basis vectors of $\reals^{d+1}$.
Let $J = \diag(-1, 1, \ldots, 1)$ be the $(d+1) \times (d+1)$ diagonal matrix whose diagonal consists of all ones, except its first entry is minus one.
For $x, y \in \mathbb{R}^{d+1}$, let $\langle x, y \rangle = x^\top J y$.
Define the submanifold $\calM = \{x \in \reals^{d+1} : \langle x, x\rangle=-1, \langle x, e_0 \rangle > 0\} \subset \reals^{d+1}$, which has tangent spaces $\T_x \calM = \{v \in \reals^{d+1} : \langle x, v\rangle = 0\}$.
Hyperbolic space $\mathbb{H}^d$ is \emph{identified} with $\calM$ endowed with the inner product $\langle u, v \rangle = u^\top J v$ on its tangent spaces.

We have $\dist(x, y) = \arccosh(- \langle x, y\rangle)$.
For $x, y\in \calM$ and $v \in \T_x \calM$,
$$\exp_x(v) = \cosh(\|v\|) x + \frac{\sinh(\|v\|)}{\|v\|} v, \quad \log_x(y) = \frac{\dist(x,y)}{\sinh(\dist(x,y))} (y - \cosh(\dist(x, y)) x).$$
If $v \in \T_x \calM$, $y = \exp_x(v)$, and $u, w \in \T_x \calM$ with $u$ orthogonal to $v$ and $w$ parallel to $v$, then
\begin{align} \label{usefulptrformula}
P_{x \to y} w = \sinh(\|v\|) x + \cosh(\|v\|) w, \quad P_{x \to y} u = u.
\end{align}

Let us consider totally geodesic submanifolds of $\calM = \mathbb{H}^d$.
We have the following classical characterization in the hyperboloid model.
\begin{lemma} \label{characterizationoftotgeoinhyperspace}
With $\calM = \mathbb{H}^d$:
\begin{itemize}
\item If $P$ is a subspace of $\mathbb{R}^{d+1}$ and $\calM \cap P$ is nonempty, then $\calM \cap P$ is a totally geodesic submanifold of $\calM$.  Moreover, $\T_x S = \T_x \calM \cap P$ for all $x \in S$

\item Conversely, if $S$ is a $k$-dimensional totally geodesic submanifold of $\calM$, then there is a unique $k+1$-dimensional subspace $P \subseteq \reals^{d+1}$ such that $S = \calM \cap P$.

\item If $x \in \calM$ and $\mathscr{S}$ is a $k$-dimensional linear subspace of $\T_x \calM$, then 
$S = \exp_x(\mathscr{S}) = \calM \cap \spann(x, \mathscr{S})$ is a $k$-dimensional totally geodesic submanifold whose tangent space at $x$ equals $\mathscr{S}$.

\item A $k$-dimensional totally geodesic submanifold of $\mathbb{H}^d$ is isometric to $\mathbb{H}^k$.
\end{itemize}
\end{lemma}

We define the following analogue of the Euclidean span.
This object plays an important role in assumption~\aref{A1} in the subsequent section.
\begin{definition}
For $(x_1, v_1), \ldots, (x_m, v_m) \in \T \calM$, define
$$\gspan(x_1, \ldots, x_m, v_1, \ldots, v_m) = \calM \cap \spann(x_1, \ldots, x_m, v_1, \ldots, v_m).$$
\end{definition}
If any of the points $x_i$ are repeated, then we usually do not repeat them in the $\gspan$.
If any of the $v_i$ equal zero, we usually omit them from the $\gspan$.
\begin{proposition} \label{prop}
Let $(x_1, v_1), \ldots, (x_m, v_m) \in \T \calM$.
Let $S = \gspan(x_1, \ldots, x_m, v_1, \ldots, v_m)$.  
Then $S$ is the minimal totally geodesic submanifold of $\calM$ which passes through each $x_i$ and is tangent to each $v_i$, i.e.,
\begin{align} \label{eq:tangentconstraint}
(x_i, v_i) \in \T S \quad \text{for all } i=1, \ldots, m.
\end{align}
\end{proposition}
By ``$S$ is minimal,'' we mean that for any totally geodesic submanifold $S'$ satisfying \eqref{eq:tangentconstraint}, then we must have $S \subset S'$.
This proposition tells us that the $\gspan$ is independent of the model we use to represent hyperbolic space.
\begin{proof}[Proof of Proposition~\ref{prop}]
$S$ is a totally geodesic submanifold.
Moreover, it is clear that $S$ satisfies \eqref{eq:tangentconstraint} because $x_i \in S$ and $v_i \in \T_{x_i} \calM \cap \spann(x_1, \ldots, x_m, v_1, \ldots, v_m) = \T_{x_i} S$.

Consider any totally geodesic submanifold $S'$ satisfying \eqref{eq:tangentconstraint}.  We know $S' = \calM \cap P'$ for some subspace $P'$ of $\reals^{d+1}$.
Moreover, for each $i=1, \ldots, m$, there is a curve $\gamma_i \subset S' \subset P'$ with $\gamma_i(0) = x_i, \gamma_i'(0) = v_i$.  Therefore $x_i \in P'$ and $v_i \in \T_{x_i} S' = P' \cap \T_{x_i} \calM \subset P'$.  So $\spann(x_1, \ldots, x_m, v_1, \ldots, v_m) \subset P'$.  Therefore $S \subset S'$.
\end{proof}

\subsection{An $\Omega(\frac{\zeta}{\epsilon^2})$ lower bound for a class of algorithms}
In this section, we prove an $\Omega(\frac{\zeta}{\epsilon^2})$ lower bound for a class of algorithms (satisfying assumptions~\aref{A1} and~\aref{A2} below) solving problem~\ref{highdimLipschitzproblem}.

Let $\epsilon \leq \frac{1}{4 \sqrt{2}}$.
Let $T = d = \lfloor \frac{1}{32} \frac{\zeta_r}{\epsilon^2} \rfloor$.
Fix $\theta \in (0, \frac{\pi}{2})$ so that $\epsilon = \frac{1}{4} \cos(\theta)$ (and so $\cos^2(\theta) \leq \frac{1}{2}$).  Without loss of generality, we consider functions with Lipschitz constant $M = \frac{2}{\cos(\theta)}$.
We divide the proof into two parts: (I) the \textbf{geometric setup}, and (II) the actual \textbf{lower bound proof}.

\bigskip

\noindent \textbf{(I) Geometric setup:}
We inductively define a series of geometric objects.

Let us first initialize ($k=0$).  Let $y_0 = \xorigin$ and $I_0 = \{y_0\}$.  Let $S_0 = \partial B(y_0, r)$ ($S_0$ is a hyperbolic sphere of dimension $d-1$).  $S_0$ is the boundary of the ball $B_0 = B(y_0, r)$.
Let $r_0$ be the radius of $S_0$, i.e., $r_0 = r$.
$S_0$ and $B_0$ are contained in the $d$-dimensional totally geodesic submanifold $H_0 = \calM$.
Choose an orthonormal basis $e_1, \ldots, e_d$ for $\T_{y_0} \calM$.
Define $e_i^{(0)} = e_i$ for $i=1, \ldots d$.
\emph{Working in the hyperboloid model}, we take $\xorigin = e_0$, and let the orthonormal basis for $\T_{y_0} \calM$ consist of the coordinate basis vectors, also denoted $e_1, \ldots, e_d$.

Next, for $k=1, \ldots, d-1$:
\begin{itemize}
\item Define $\Delta_{k-1} > 0$ and $r_{k} > 0$ by
\begin{align} \label{yetanotherusefulguy25}
\cos(\theta) = \frac{\tanh(\Delta_{k-1})}{\tanh(r_{k-1})}, \text{    }\sin(\theta) = \frac{\sinh(r_k)}{\sinh(r_{k-1})}, \text{    } \cosh(r_{k-1}) = \cosh(r_k) \cosh(\Delta_{k-1}).
\end{align}
(The third equation is a consequence of the first two equations.)
\begin{lemma} \label{superusefulguy2}
The sequence $r_k > 0$ given by $r_0 = r, \sin(\theta) = \frac{\sinh(r_k)}{\sinh(r_{k-1})}$ satisfies $r_k \geq \frac{r}{2}$ for all $k \leq T = d = \lfloor \frac{1}{32}\frac{\zeta_r}{\epsilon^2} \rfloor = \lfloor \frac{\zeta_r}{2\cos^2(\theta)} \rfloor$.
\end{lemma}
\begin{proof}
Indeed, $\sinh(r_k) = \sin(\theta)^k \sinh(r) \geq \sinh(r/2)$ provided
$$k \leq \frac{\log\Big(\frac{\sinh(r/2)}{\sinh(r)}\Big)}{\log(\sin(\theta))} = 2 \frac{\log\Big(\frac{\sinh(r/2)}{\sinh(r)}\Big)}{\log(1-\cos^2(\theta))}.$$
We know $2 \frac{\log\big(\frac{\sinh(r/2)}{\sinh(r)}\big)}{\log(1-x)} \geq \frac{\zeta_r}{2 x}$ for all $x \in (0,1/2]$ and $r > 0$.
Therefore, since $k \leq \frac{\zeta_r}{2\cos^2(\theta)}$ and $\cos^2(\theta) \leq 1/2$, we conclude $r_k \geq r/2$.
\end{proof}

\item Define $\tilde g_{k-1} = - \frac{1}{\cos(\theta)} e_{k}^{(k-1)}$.
Define $y_k = \exp_{y_{k-1}}(\Delta_{k-1} e_{k}^{(k-1)}) = \exp_{y_{k-1}}(-\Delta_{k-1} \frac{\tilde g_{k-1}}{\|\tilde g_{k-1}\|})$.
Define $e_i^{(k)} = P_{y_{k-1} \to y_k} e_i^{(k-1)}$ for all $i = 1, \ldots, d$.

In the hyperboloid model, observe that $e_i^{(k)} = e_i$ for all $i \geq k+1$, and $e_i^{(k)} \in \spann(e_0, \ldots, e_k)$ for all $1 \leq i \leq k$ (both of which can be verified by induction and formula~\eqref{usefulptrformula}).  Moreover,
\begin{align} \label{somanyusefuleqns}
y_{k} = \cosh(\Delta_{k-1}) y_{k-1} + \sinh(\Delta_{k-1}) e_{k}, \quad e_k^{(k)} = \sinh(\Delta_{k-1}) y_{k-1} + \cosh(\Delta_{k-1}) e_{k}.
\end{align}
In particular, $y_k \in \calM \cap \spann(e_0, e_1, \ldots, e_k)$.

\item Define the $k$-dimensional totally geodesic submanifold $I_k = \gspan(y_0, \ldots, y_k)$.  
Note that
$$I_k = \calM \cap \spann(e_0, e_1, \ldots, e_k) = \gspan(y_0, e_1, \ldots, e_k) = \gspan(y_0, \ldots, y_{k-1}, e_{k}^{(k-1)}).$$
Using Proposition~\ref{prop} and that $e_i^{(k)} \in \spann(e_0, \ldots, e_k)$ for all $1 \leq i \leq k$, we determine that $\T_{y_k} I_k = \spann(e_1^{(k)}, \ldots, e_k^{(k)})$, and $I_k = \gspan(y_k, e_1^{(k)}, \ldots, e_k^{(k)})$.
Lastly, note that $I_0 \subset I_1 \subset, \ldots$.

\item Define $H_k = \{x \in \calM : \langle x, e_{i}^{(i)} \rangle = 0, \text{ } \forall 1 \leq i\leq k\}$.  As an intersection of $\calM$ and a subspace, $H_k$ is a totally geodesic submanifold.
Note that $H_0 \supset H_1 \supset \ldots$.

\begin{lemma} \label{propertiesofHk}
$H_k$ is a $d-k$-dimensional totally geodesic submanifold which (a) passes through $y_k$,
(b) is orthogonal to $e_i^{(k)}$ for $i=1, \ldots, k$, and (c) is tangent to $e_i^{(k)}$ for $i=k+1, \ldots, d$.
In particular, $H_k = \gspan(y_k, e_{k+1}^{(k)}, \ldots, e_d^{(k)}) = \calM \cap \spann(y_k, e_{k+1}, \ldots, e_d)$.
\end{lemma}
\begin{proof}
We can prove (a) and (b) by induction on $k$.  The base case $k=0$ is clear.
By the inductive hypothesis $y_{k-1} \in H_{k-1}$ and $e_k^{(k-1)} \in \T_{y_{k-1}} H_{k-1}$, so we have that $y_k \in H_{k-1}$ ($H_{k-1}$ is totally geodesic).
Equation~\eqref{somanyusefuleqns} implies $y_k \in \{x \in \calM : \langle x, e_k^{(k)} \rangle = 0\}$.
Therefore $y_k \in H_k$.

Again by the inductive hypothesis, $H_{k-1}$ is orthogonal to $e_i^{(k-1)}$ for all $i \leq k-1$.
Since $y_k \in H_{k-1}$, $H_{k-1}$ is orthogonal to $e_i^{(k)} = P_{y_{k-1} \to y_k} e_i^{(k-1)}$ for all $i \leq k-1$ (Lemma~\ref{usefullemmaaboutparalleltransportintotgeo}).
As $H_k \subset H_{k-1}$, $H_k$ is orthogonal to $e_i^{(k)}$ for all $i \leq k-1$.
By definition, $H_k$ is orthogonal to $e_k^{(k)}$, and so we conclude $H_k$ is orthogonal to $e_i^{(k)}$ for all $i \leq k$.

For (c): by Lemma~\ref{characterizationoftotgeoinhyperspace}, we know that
$$\T_{y_k} H_k = \{u \in \mathbb{R}^{d+1} : \langle y_k, u\rangle = 0, \langle e_i^{(i)}, u \rangle = 0 \text{ } \forall 1 \leq i \leq k\}.$$
Since $y_k \in \spann(e_0, \ldots, e_k)$ and $e_i^{(i)} \in \spann(e_0, \ldots, e_i)$, we conclude that $e_i = e_i^{(k)} \in \T_{y_k} H_k$ for $i \geq k+1$.
Using (b), $ \T_{y_k} H_k = \spann(e_{k+1}^{(k)}, \ldots, e_d^{(k)})$, and so $H_k$ has dimension $d-k$ and $H_k = \exp_{y_k}(\spann(e_{k+1}^{(k)}, \ldots, e_d^{(k)})) = \calM \cap \spann(y_k, e_{k+1}, \ldots, e_d)$ (Lemma~\ref{characterizationoftotgeoinhyperspace}).
\end{proof}

\item Let $S_k$ be the set of all points $x^* \in S_{k-1}$ such that the angle between $e_{k}^{(k-1)}$ and $\log_{y_{k-1}}(x^*)$ equals $\theta$.
Note that $S_0 \supset S_1 \supset \ldots$.

\begin{lemma} \label{propertiesofSk}
(a) $S_k$ is the hyperbolic sphere of dimension $d-k-1$ contained in $H_k$ which has center $y_k$ and radius $r_k$.  

(b) If $x^* \in S_k$, then $\log_{y_k}(x^*)$ is orthogonal to $I_k$.  In particular, $y_k$ is the closest point in $I_k$ to $x^*$.
\end{lemma}
\begin{proof}
We can verify (a) by induction using the hyperboloid model.  It is clearly true for $k=0$.
By the inductive hypothesis, $\dist(x^*, y_{k-1}) = r_{k-1}$ for all $x^* \in S_{k-1}$.  Therefore,
\begin{equation} \label{usefuldescofSk}
\begin{split}
S_k &= \{x^* \in S_{k-1} : r_{k-1} \cos(\theta) = \langle \log_{y_{k-1}}(x^*), e_k^{(k-1)} \rangle = \langle \log_{y_{k-1}}(x^*), e_k \rangle\} \\
&= \{x^* \in S_{k-1} : r_{k-1} \cos(\theta) = \frac{r_{k-1}}{\sinh(r_{k-1})} \langle x^*, e_k \rangle\}.
\end{split}
\end{equation}
So using equation~\eqref{somanyusefuleqns} and that $\cosh(r_{k-1}) = - \langle x^*, y_{k-1} \rangle$, if $x^* \in S_k$ then
$$\langle x^*, e_{k}^{(k)} \rangle = -\cosh(r_{k-1}) \sinh(\Delta_{k-1}) + \sinh(r_{k-1}) \cos(\theta) \cosh(\Delta_{k-1}) = 0,$$
where the last equality follows from equation~\eqref{yetanotherusefulguy25}.
We conclude that $S_k$ is contained in $\{x \in \calM : \langle x, e_k^{(k)} \rangle = 0\}$.
Because $S_k \subset S_{k-1} \subset H_{k-1}$, we conclude
$$S_k \subset H_{k-1} \cap \{x \in \calM : \langle x, e_k^{(k)} \rangle = 0\} = \{x \in \calM : \langle x, e_i^{(i)} \rangle = 0 \text{ } \forall i\leq k\} = H_k.$$
A calculation using equations~\eqref{yetanotherusefulguy25},~\eqref{somanyusefuleqns} and~\eqref{usefuldescofSk} shows that $\dist(y_k, x^*) = r_k$ for all $x^* \in S_k$.
So we have shown $S_k$ is a subset of the sphere of radius $r_k$ centered at $y_k$ in $H_k$.

By the inductive hypothesis, $S_{k-1}$ is a hyperbolic sphere in $H_{k-1}$ with center $y_{k-1}$ and radius $r_{k-1}$.
Now let $x^*$ be a point in $H_k$ with $\dist(x^*, y_k) = r_k$.
By Lemma~\ref{propertiesofHk}, $H_k$ contains $y_k$ and is orthogonal to $e_{k}^{(k)}$.  Since $H_k$ is totally geodesic, $\log_{y_k}(x^*)$ is orthogonal to $e_{k}^{(k)}$.  Therefore, the hyperbolic triangle $x^*, y_k, y_{k-1}$ has a right angle at $y_k$.
Using $\dist(x^*, y_k) = r_k$, $\dist(y_k, y_{k-1}) = \Delta_{k-1}$, equations~\eqref{yetanotherusefulguy25} and hyperbolic trigonometry, we find that $\dist(y_{k-1}, x^*) = r_{k-1}$ and the angle between $e_k^{(k-1)}$ and $\log_{y_{k-1}}(x^*)$ equals $\theta$.
Therefore, $x^* \in S_k$.
We conclude that $S_k$ contains (and so equals) the sphere of radius $r_k$ centered at $y_k$ in $H_k$.
We have shown (a).

For (b): we know that $y_k \in H_k$ by Lemma~\ref{propertiesofHk} and $x^* \in S_k \subset H_k$.
Since $H_k$ is totally geodesic, $\log_{y_k}(x^*) \in \T_{y_k} H_k$.
However, Lemma~\ref{propertiesofHk} implies $\log_{y_k}(x^*) \in \T_{y_k} H_k$ is orthogonal to $\spann(e_1^{(k)}, \ldots, e_k^{(k)}) = \T_{y_k} I_k$.
\end{proof}

\item Let $B_k \subset H_k$ be the $d-k$-dimensional hyperbolic ball whose boundary is $S_k$.
Since $B_0 \supset B_1 \supset \ldots$,  
we find that $y_{\ell} \in B_{k}$ for all $\ell \geq k$.
\end{itemize}


\noindent For $k = 0, \ldots, d-2$, consider the following additional geometric objects:
\begin{itemize}
\item For each $x^* \in S_{k+1}$, define the tangent vector $V_{k, x^*}$ and corresponding half-space $L_{k, x^*}$: 
$$V_{k, x^*} = -\tilde g_k - \frac{\log_{y_k}(x^*)}{\dist(y_k, x^*)}, \quad \quad L_{k, x^*} = \exp_{y_k}(\{ v \in \T_{y_k} \calM : \langle V_{k, x^*}, v \rangle \geq 0\}).$$

\item Define $C_k = \bigcap_{x^* \in S_{k+1}} L_{k, x^*}$.  As an intersection of (g-convex) half-spaces, $C_k$ is also g-convex.
\end{itemize}

\begin{lemma}\label{lemma0}
We have the following:
\begin{itemize}
\item For $x^* \in S_{k+1}$, $\partial L_{k, x^*}$ is a $(d-1)$-dimensional totally geodesic submanifold containing $x^*$ and $I_k$.

\item For $x^* \in S_{k+1}$, $y_{k+1} \in B_{k+1} \subset C_k \subset L_{k, x^*}$ for all $k \in \{0, \ldots, d-2\}$.

\item Let $x^* \in S_{d-1}$.  Then $y_0, \ldots, y_{d-1} \in C_k \subset L_{k, x^*}$ for all $k \leq d-2$.
\end{itemize}
\end{lemma}
\begin{proof}
\emph{For the first bullet}: if $x^* \in S_{k+1}$, we know $\langle \log_{y_k}(x^*), e_{k+1}^{(k)}\rangle = \dist(y_k, x^*) \cos(\theta)$, and so
$$\langle \frac{1}{\cos(\theta)} e_{k+1}^{(k)} - \frac{\log_{y_k}(x^*)}{\dist(y_k, x^*)}, \log_{y_k}(x^*) \rangle = \dist(y_k, x^*) - \frac{\dist(y_k, x^*)^2}{\dist(y_k, x^*)} = 0.$$
That is, $\partial L_{k, x^*}$ contains $x^*$.
Using that $e_{k+1}^{(k)} = e_{k+1}$ and $y_k \in I_k = \calM \cap \spann(e_0, \ldots, e_k)$ in the hyperboloid model, we see that $\langle \frac{1}{\cos(\theta)} e_{k+1}^{(k)} - \frac{\log_{y_k}(x^*)}{\dist(y_k, x^*)}, \log_{y_k}(x) \rangle = 0$ for all $x \in I_k$.
Therefore, $\partial L_{k, x^*}$ also contains $I_k$.

\emph{For the second bullet}: we know angle between $e_{k+1}^{(k)}$ and $\log_{y_k}(x^*)$ equals $\theta$, and so the angle between $e_{k+1}^{(k)}$ and $V_{k, x^*}$ equals $\frac{\pi}{2} - \theta$ ($V_{k, x^*}$ is a linear combination of $e_{k+1}^{(k)}$ and $\log_{y_k}(x^*)$).
Therefore, $y_{k+1} \in L_{k, x^*}$, and since $x^* \in S_{k+1}$ was arbitrary we have $y_{k+1} \in C_k$.

We know $V_{k, x^*}$ is orthogonal to $\partial L_{k, x^*}$.
Since $y_k, x^* \in \partial L_{k, x^*}$ and $\partial L_{k, x^*}$ is totally geodesic (by the first bullet), we know that $P_{y_k \to x^*} V_{k, x^*}$ is orthogonal to $\partial L_{k, x^*}$ (Lemma~\ref{usefullemmaaboutparalleltransportintotgeo}).
Therefore, $L_{k, x^*} = \exp_{x^*}(\{u \in \T_{x^*}\calM : \langle P_{y_k \to x^*} V_{k, x^*}, u \rangle \geq 0\})$.
Using that $H_{k+1} = \exp_{x^*}(\T_{x^*} H_{k+1})$ since $H_{k+1}$ is totally geodesic, we find
\begin{align} \label{onemoregreateqn}
H_{k+1} \cap L_{k, x^*} = \exp_{x^*}(\{u \in \T_{x^*} H_{k+1} : \langle \calP_{k+1} P_{y_k \to x^*} V_{k, x^*}, u \rangle \geq 0\}),
\end{align}
where $\calP_{k+1}$ denotes orthogonal projection onto $\T_{x^*} H_{k+1}$.
Therefore, $H_{k+1} \cap L_{k, x^*}$ is a half-space in $H_{k+1}$.
We seek to compute $\calP_{k+1} P_{y_k \to x^*} V_{k, x^*}$.
We can do this using two-dimensional hyperbolic geometry.

Define the two-dimensional totally geodesic submanifold
$$\tilde H = \gspan(y_k, y_{k+1}, x^*) = \gspan(y_k, e_{k+1}^{(k)}, \log_{y_k}(x^*)) = \gspan(y_{k+1}, e_{k+1}^{(k)}, \log_{y_{k+1}}(x^*)).$$
By Lemma~\ref{propertiesofHk}, $e_{k+1}^{(k+1)}$ is orthogonal to $H_k$, and so in particular to $\log_{y_{k+1}}(x^*)$ as well.
Therefore, $e_{k+1}^{(k+1)}$ and $\log_{y_{k+1}}(x^*)$ form an orthogonal basis for $\T_{y_{k+1}} \tilde H$, and so $P_{y_{k+1} \to x^*} e_{k+1}^{(k+1)}$ and $\log_{x^*}(y_{k+1})$ form an orthogonal basis for $\T_{x^*} \tilde H$ (Lemma~\ref{usefullemmaaboutparalleltransportintotgeo}).
By its definition, $V_{k, x^*}$ is in
$$\spann(e_{k+1}^{(k)}, \log_{y_k}(x^*)) = \T_{y_k} \tilde H$$
and is orthogonal to $\log_{y_k}(x^*)$.
Therefore, $P_{y_k \to x^*} V_{k, x^*}$ is in
$$\T_{x^*} \tilde H = \spann(P_{y_{k+1} \to x^*} e_{k+1}^{(k+1)}, \log_{x^*}(y_{k+1}))$$
and is orthogonal to $\log_{x^*}(y_k)$.
Let $\alpha \in (0, \frac{\pi}{2})$ be the angle between $\log_{x^*}(y_k)$ and $\log_{x^*}(y_{k+1})$ ($y_k, y_{k+1}, x^*$ is a right triangle).
We conclude the angle between $P_{y_k \to x^*} V_{k, x^*}$ and $\log_{x^*}(y_k)$ is $\frac{\pi}{2} - \alpha$, and so we can decompose $P_{y_k \to x^*} V_{k, x^*}$ in the orthogonal basis for $\T_{x^*} \tilde H$:
$$P_{y_k \to x^*} V_{k, x^*} = \|P_{y_k \to x^*} V_{k, x^*}\| \bigg[ \sin\Big(\frac{\pi}{2} - \alpha\Big) \frac{P_{y_{k+1} \to x^*} e_{k+1}^{(k+1)}}{\|P_{y_{k+1} \to x^*} e_{k+1}^{(k+1)}\|} + \cos\Big(\frac{\pi}{2} - \alpha\Big) \frac{\log_{x^*}(y_{k+1})}{\|\log_{x^*}(y_{k+1})\|}\bigg].$$
On the other hand, $P_{y_{k+1} \to x^*} e_{k+1}^{(k+1)}$ is orthogonal to $H_{k+1}$ and $\log_{x^*}(y_{k+1})$ is tangent to $H_{k+1}$, since $H_{k+1}$ is totally geodesic, is orthogonal to $e_{k+1}^{(k+1)}$, and contains $x^*$ and $y_{k+1}$.
We conclude that 
$$\calP_{k+1} P_{y_k \to x^*} V_{k, x^*} = \|P_{y_k \to x^*} V_{k, x^*}\| \bigg[\cos\Big(\frac{\pi}{2} - \alpha\Big) \frac{\log_{x^*}(y_{k+1})}{\|\log_{x^*}(y_{k+1})\|}\bigg].$$
Using equation~\eqref{onemoregreateqn}, we get 
$$H_{k+1} \cap L_{k, x^*} = \exp_{x^*}(\{u \in \T_{x^*} H_{k+1} : \langle \log_{x^*}(y_{k+1}), u \rangle \geq 0\}),$$
Since $B_{k+1} \subset H_{k+1}$ is g-convex, has center $y_{k+1}$ and contains $x^*$ on its boundary ($S_{k+1}$), we conclude that $B_{k+1} \subset H_{k+1} \cap L_{k, x^*} \subset L_{k, x^*}$.
Since $x^* \in S_{k+1}$ was arbitrary, $B_{k+1} \subset C_k$.

\emph{For the third bullet}: for all $k = 0, \ldots, d-2$ we know that $x^* \in C_k$ by the second bullet, and $I_k \subset C_k$ by the first bullet.  
In particular, $x^*, y_0, \ldots, y_k \in C_k$ for all $k = 0, \ldots, d-2$.
Since $y_{\ell} \in B_\ell \subset B_{k+1}$ for all $\ell \geq k+1$, we conclude (using the second bullet) that $y_0, \ldots, y_{d-1} \in C_k$ for all $k \leq d-2$.
\end{proof}

\begin{lemma} \label{lemma1}
Let $k \in \{0, \ldots, d-2\}$, $x^* \in S_{k+1}$ and $y \in I_k$.
The g-convex function $x \mapsto \frac{1}{\cos(\theta)} \dist(x, L_{k, x^*})$ has a subgradient $\hat g$ at $y$ such that $-\hat g + \frac{\log_{y}(x^*)}{\dist(y, x^*)}$ is tangent to $I_{k+1}$.
If $y = y_k$, then we can take $\hat g$ such that $-\hat g + \frac{\log_{y_k}(x^*)}{\dist(y_k, x^*)} = \frac{1}{\cos(\theta)} e_{k+1}^{(k)} = -\tilde g_k$.
%
\end{lemma}
\begin{proof}
First, let us consider the particular case $y = y_k$.  The subdifferential of $x \mapsto \frac{1}{\cos(\theta)} \dist(x, L_{k, x^*})$ at $y_k$ contains all vectors of the form $t \frac{V_{k, x^*}}{\norm{V_{k, x^*}}}$ with $t \in [-\frac{1}{\cos(\theta)},0]$.  Take $\hat g = -\tan(\theta) \frac{V_{k, x^*}}{\|V_{k, x^*}\|} = -\tan(\theta) \frac{V_{k, x^*}}{\sqrt{\frac{1}{\cos^2(\theta)}-1}} = - V_{k, x^*} = \tilde g_k + \frac{\log_{y_k}(x^*)}{\dist(y_k, x^*)}$.

Now consider the general case $y \in I_k$.
By Lemma~\ref{propertiesofSk}(b), $\partial L_{k, x^*}$ contains $y$ and $y_k$.
Therefore, $P_{y_k \to y} V_{k, x^*}$ is orthogonal to $\partial L_{k, x^*}$ (Lemma~\ref{usefullemmaaboutparalleltransportintotgeo}), and the subdifferential of $x \mapsto \frac{1}{\cos(\theta)} \dist(x, L_{k, x^*})$ at $y$ contains all vectors of the form $t P_{y_k \to y} \frac{V_{k, x^*}}{\|V_{k, x^*}\|}$ with $t \in [-\frac{1}{\cos(\theta)},0]$.

Let $\theta_y \in (0, \frac{\pi}{2})$ be the angle between $\log_y(y_{k+1})$ and $\log_y(x^*)$.
By Lemma~\ref{lemma0}, $y_{k}$ is the closest point in $I_k$ to $x^*$, so $\dist(y, x^*) \geq \dist(y_k, x^*)$.
Therefore, comparing the two right triangles $y_k, y_{k+1}, x^*$ and $y, y_{k+1}, x^*$ (the triangles are right by part (b) of Lemma~\ref{propertiesofSk}), we determine $\theta_y \leq \theta$.
In particular, $\hat g = -\frac{\sin(\theta_y)}{\cos(\theta)} P_{y_k \to y} \frac{V_{k, x^*}}{\|V_{k, x^*}\|}$ is a subgradient at $y$.


Consider the two-dimensional totally geodesic submanifold $\tilde H = \gspan(y, y_{k+1}, x^*).$
We know that $\log_{y_{k+1}}(x^*)$ is tangent to $\tilde H$ and orthogonal to $\log_{y_{k+1}}(y)$.
Therefore, $P_{y_{k+1} \to y} \log_{y_{k+1}}(x^*)$ is tangent to $\tilde H$ and orthogonal to $\log_{y}(y_{k+1})$.
That is, $P_{y_{k+1} \to y} \log_{y_{k+1}}(x^*)$ and $\log_{y}(y_{k+1})$ form an orthogonal basis for $\T_y \tilde H$.
However, $\log_y(x^*)$ is contained in $\T_y \tilde H$, and so we can decompose in this orthogonal basis:
\begin{align} \label{thisisnice}
\frac{\log_y(x^*)}{\dist(y, x^*)} = \cos(\theta_y) \frac{\log_{y}(y_{k+1})}{\dist(y, y_{k+1})} + \sin(\theta_y) \frac{P_{y_{k+1} \to y} \log_{y_{k+1}}(x^*)}{\|P_{y_{k+1} \to y} \log_{y_{k+1}}(x^*)\|}.
\end{align}
Taking $y = y_k$ and using $\frac{\log_{y_k}(y_{k+1})}{\dist(y_k, y_{k+1})} = e_{k+1}^{(k)}$,
\begin{align} \label{thisisanothereqnagain2}
\frac{\log_{y_k}(x^*)}{\dist(y_k, x^*)} = \cos(\theta) e_{k+1}^{(k)} + \sin(\theta) \frac{P_{y_{k+1} \to y_k} \log_{y_{k+1}}(x^*)}{\|P_{y_{k+1} \to y_k} \log_{y_{k+1}}(x^*)\|}.
\end{align}
Recalling that ${\log_{y_{k+1}}(x^*)}$ is orthogonal to $I_{k+1}$, and
using formula~\eqref{usefulptrformula} and Lemma~\ref{usefullemmaaboutparalleltransportintotgeo}, 
\begin{align} \label{thisisanothereqnagain1}
P_{y_{k+1} \to y} {\log_{y_{k+1}}(x^*)}= P_{y_k \to y} P_{y_{k+1} \to y_k} {\log_{y_{k+1}}(x^*)}.
\end{align}
Combining the definition of $V_{k, x^*}$, equations~\eqref{thisisanothereqnagain2} and~\eqref{thisisanothereqnagain1}, and $\|V_{k, x^*}\| = \tan(\theta)$, we find
\begin{align*}
\hat{g} &= -\frac{\sin(\theta_y)}{\cos(\theta)} \frac{1}{\|V_{k, x^*}\|} P_{y_k \to y} V_{k, x^*} 
= -\frac{\sin(\theta_y)}{\cos(\theta)} \frac{1}{\|V_{k, x^*}\|} P_{y_k \to y} \bigg[\frac{1}{\cos(\theta)} e_{k+1}^{(k)} - \frac{\log_{y_k}(x^*)}{\dist(y_k, x^*)}\bigg] \\
&= c_1 P_{y_k \to y} e_{k+1}^{(k)} + \frac{\sin(\theta_y)}{\cos(\theta)} \frac{1}{\|V_{k, x^*}\|}\sin(\theta) P_{y_k \to y} \frac{P_{y_{k+1} \to y_k} \log_{y_{k+1}}(x^*)}{\|P_{y_{k+1} \to y_k} \log_{y_{k+1}}(x^*)\|} \\
&= c_1 P_{y_k \to y} e_{k+1}^{(k)} + \sin(\theta_y) P_{y_{k+1} \to y} \frac{\log_{y_{k+1}}(x^*)}{\|P_{y_{k+1} \to y_k} \log_{y_{k+1}}(x^*)\|}
\end{align*}
for some constant $c_1$ which we are omitting.
Using equation~\eqref{thisisnice}, we find
$$-\hat{g} + \frac{\log_y(x^*)}{\dist(y, x^*)} = - c_1 P_{y_k \to y} e_{k+1}^{(k)} + \cos(\theta_y) \frac{\log_{y}(y_{k+1})}{\dist(y, y_{k+1})}$$
which is indeed tangent to $I_{k+1}$
\end{proof}

\noindent \textbf{(II) The lower bound:}
We know that $S_{d-1}$ is a $0$-dimensional hyperbolic sphere of radius $r_{d-1} \geq r/2$ (i.e., a set consisting of two points which are at least $2 \cdot r/2 = r$ apart).  From now on we \emph{fix} an $x^* \in S_{d-1}$.
Consider the following $g$-convex, $M$-Lipschitz function:
\begin{align} \label{formulaforfwfinw}
f(x) = \dist(x, x^*) + \max_{i=0, \ldots, d-2} \bigg\{ \frac{1}{\cos(\theta)} \dist(x, L_{i, x^*})\bigg\}.
\end{align}
We know $x^* \in S_{d-1} \subset S_k$ for all $k \in \{0, \ldots, d-1\}$ and $S_{k+1} \subset L_{k, x^*}$ for all $k \in \{0, \ldots, d-2\}$ (Lemma~\ref{lemma0}).
Therefore, $x^* \in L_{k, x^*}$ for all $k \in \{0, \ldots, d-2\}$, and we conclude $f^* = f(x^*) = \dist(x^*, x^*) + 0 = 0$.

Let $\calA$ be a deterministic first-order algorithm. 
We make the following mild assumption on $\calA$.
\begin{assumption}\label{A1}
The algorithm $\calA$ chooses its queries $x_k$ in the $\gspan$ of past queries and subgradients.  That is, $x_0 = \xorigin$, and
$x_k \in \gspan(x_0, \ldots, x_{k-1}, g_0, \ldots, g_{k-1})$ for all $k \geq 1$.
\end{assumption}
The assumption~\aref{A1} is mild, and can probably be removed using similar techniques as in the Euclidean case~\citep[Ch.~7]{nemirovskiandyudin1983}.  
All algorithms we are familiar with and which use the exponential map $\exp$ as the retraction (and parallel transport or the differential of $\exp$ as the transporter) satisfy this assumption.

\begin{assumption}\label{A2}
When running on the oracle for $f$ defined below, for each $k \leq d-1$ the algorithm $\calA$ queries $x_k \in \cap_{\ell=0}^{k-1} L_{\ell, x^*}$.
\end{assumption}
The assumption~\aref{A2} is \emph{more restrictive}, and it is an open question how it can be removed.
We verify it for a reasonable algorithm in the next section.

We know that $\calA$ initially queries $x_0 = \xorigin = y_0 \in I_0$ by assumption~\aref{A1}, that $x_0 = y_0 \in L_{\ell, x^*}$ for all $\ell \in \{0, \ldots, d-2\}$ (Lemma~\ref{lemma0}), and $f(x_0) - f^* = r_0 = r \geq r/2$.
We show by induction that for each $k = 1, \ldots, d-1$, (a) there is a subgradient $g_{k-1} \in \partial f(x_{k-1})$ which is tangent to $I_{k}$, (b) the algorithm $\calA$ produces the next query $x_{k}$ in $I_{k}$, (c) $x_k \in L_{\ell, x^*}$ for all $\ell \in \{0, \ldots, d-2\}$, and (d) $f(x_{k}) - f^* = r_k \geq \frac{r}{2}$.

Indeed, given $x_{k-1} \in I_{k-1}$, Lemma~\ref{lemma1} implies $x \mapsto \frac{1}{\cos(\theta)} \dist(x, L_{k-1, x^*})$ has a subgradient $\hat{g}_{k-1}$ such that $g_{k-1} = \hat{g}_{k-1} - \frac{\log_{x_{k-1}}(x^*)}{\dist(x_{k-1}, x^*)}$ is tangent to $I_k$.
By the inductive hypothesis we know $x_{k-1} \in L_{\ell, x^*}$ for all $\ell \in \{0, \ldots, d-2\}$, and so $g_{k-1}$ is a subgradient of $f$ at $x_{k-1}$ which is tangent to $I_k$.
Giving this subgradient to the algorithm, the next query $x_k$ lies in $\gspan(x_0, \ldots, x_{k-1}, g_0, \ldots, g_{k-1}) = I_k$ by assumption~\aref{A1}.  This shows (a) and (b).

By assumption~\aref{A2}, we know that $x_k \in L_{\ell, x^*}$ for $\ell < k$.
By Lemma~\ref{lemma0}, $I_\ell \subset \partial L_{\ell, x^*}$ for all $\ell$, so in particular $x_k \in I_k \subset I_\ell \subset L_{\ell, x^*}$ for all $\ell \geq k$.  From this, we conclude (c).
Since $x_k \in L_{\ell, x^*}$ for all $\ell \in \{0, \ldots, d-2\}$, $f(x_k) = \dist(x_k, x^*) \geq \dist(y_k, x^*) = r_k \geq r/2$ by Lemmas~\ref{superusefulguy2} and~\ref{propertiesofSk}.  This shows (d), and concludes the induction.
%
%
We showed that for all $k = 0, \ldots, d-1$, $f(x_k) - f^* \geq \frac{r}{2} = \frac{1}{4} \cos(\theta) \frac{2 r}{\cos(\theta)} = \epsilon M r.$  That is, we have proven:
\begin{theorem} \label{lowerboundforsubrgd}
Let $\epsilon \leq \frac{1}{4 \sqrt{2}}$.  Let $T = d = \lfloor \frac{1}{32} \frac{\zeta_r}{\epsilon^2} \rfloor$.
Let $f \colon \mathbb{H}^d \to \reals$ be the globally g-convex and $M$-Lipschitz function given by equation~\eqref{formulaforfwfinw}.
For any algorithm $\mathcal{A}$ satisfying assumptions~\aref{A1} and~\aref{A2}, there is an oracle for $f$ so that running $\mathcal{A}$ with that oracle produces iterates $x_0, x_1, \ldots$ satisfying $f(x_k) - f^* \geq \epsilon M r$ for all $k = 0, \ldots, T-1$.
\end{theorem}

\begin{remark}
Let us consider the role of assumption~\aref{A2}, and how strong that assumption is.

Observe that if $x_k \in I_k$ then $f(x_k) - f^* \geq r_k \geq \frac{r}{2}$ by Lemma~\ref{propertiesofSk}.
So assumption~\aref{A2} is not needed to guarantee that the function gap is large.
However, if $x_k \in I_k$ is not in $\cap_{\ell=0}^{k-1} L_{\ell, x^*}$, then it is possible that there is no subgradient in $\partial f(x_k)$ which is tangent to $I_{k+1}$.
Hence, a subgradient at $x_k \not\in \cap_{\ell=0}^{k-1} L_{\ell, x^*}$ might reveal too much information about $x^*$.
Assumption~\aref{A2} helps to guarantee that $\partial f(x_k)$ contains a subgradient which is tangent to $I_{k+1}$.

Consider an algorithm satisfying assumptions~\aref{A1} and~\aref{A2}.
Then the algorithm queries $x_1$ in the geodesic $I_1 = \{\exp_{x_0}(-t g_0) : t \in \reals\}$.  Assumption~\aref{A2} further restricts that $x_1$ is in $\{\exp_{x_0}(-t g_0) : t \geq 0\}$, which seems to be a reasonable restriction (it is intuitively unclear what advantage in general an algorithm could have by querying outside this region).
In general, we know that the algorithm queries $x_k$ in $Q_k := I_k \cap \bigcap_{\ell=0}^{k-1} L_{k, x^*}$ which is a (g-convex) subset of
$$\tilde Q_k := I_k \cap \bigcap_{\ell=0}^{k-1} \exp_{x_\ell}(\{v \in \T_{x_\ell}\calM : \langle -g_\ell, u \rangle \geq 0\}).$$
Intuitively, $Q_k$ seems to occupy most of $\tilde Q_k$.  So the assumption $x_k \in Q_k$ also seems relatively mild.
\end{remark}
\begin{remark}
Under the assumption that $x_k \in \cap_{\ell=0}^{k-1} C_k$ (in place of~\aref{A2}), it is not hard to see that the following alternative function also works in place of the function $f$ defined in equation~\eqref{formulaforfwfinw}:
\begin{align*}
f(x) &= \dist(x, x^*) + \max_{i=0, \ldots, d-2} \bigg\{ \frac{1}{\cos(\theta)} \dist(x, C_k)\bigg\}
\end{align*}
This function has the added benefit that it is more symmetric than $f$ from equation~\eqref{formulaforfwfinw}.
\end{remark}

\subsection{An $\Omega(\frac{\zeta}{\epsilon^2})$ lower bound for subgradient descent with Polyak step size} \label{upperboundforsubgradientdescent}
Our goal is to prove Corollary~\ref{corthisguysubrgd}.
To this end, let us first describe and analyze subgradient descent with Polyak step size for arbitrary $f$ in $\mathcal{F}_{r, M, 0, \infty}$.
It takes the form $x_{k+1} = \exp_{x_k}(-\eta_k g_k)$ with a particular choice of $\eta_k$ depending on $r, f(x_k) - f^*$ and $\|g_k\|$.  
To motivate the choice of step size, we take a geometric approach, which is also useful in the lower bound proof.  
We initially know that $x^* \in B(\xorigin, r)$, and define $s_0 = r$ and $x_0 = \xorigin$.  
At iteration $k \geq 0$, suppose we know that $x^* \in B(x_{k}, s_{k})$.  After querying $x_k$, we receive a subgradient $g_k \in \partial f(x_k)$ with $\|g_k\| \leq M$.  Since $g_k$ is a subgradient, we know that
$x^* \in U_k = B(x_k, s_k) \cap \{x \in \calM : \langle g_k, \log_{x_k}(x) \rangle \leq f^* - f(x_k)\}.$

We define $x_{k+1}$ to be the center of the \emph{minimal ball} containing $U_k$, and let its radius be $s_{k+1}$.  
A computation using hyperbolic trigonometry (see Appendix~\ref{appdetailsforpolyakstepsize}) shows that $x_{k+1} = \exp_{x_k}(-\eta_k g_k)$, where $\eta_k > 0$ and $s_{k+1}$ satisfy
\begin{equation} \label{yetanotherusefuleqn}
\begin{split}
\cos(\theta_k) &= \frac{f(x_k) - f^*}{s_k \|g_k\|} = \tanh(\eta_k \|g_k\|) / \tanh(s_k), \quad \sin(\theta_k) = \sinh(s_{k+1}) / \sinh(s_k), \\ \cosh(s_{k+1}) &= \cosh(s_k) \sqrt{1 - \frac{1}{\zeta_{s_k}^2} \cdot \frac{(f(x_k) - f^*)^2}{\|g_k\|^2}},
\end{split}
\end{equation}
with $\theta_k \in (0, \frac{\pi}{2})$ (see the next subsection for the definition of $\theta_k$).

Defining $\overline\Delta_T = \min_{k = 0, \ldots, T-1}\{f(x_k) - f^*\}$, we have
\begin{align*}
\cosh(s_{k+1}) &\leq \cosh(s_k) \sqrt{1 - \frac{1}{\zeta_r^2} \cdot \frac{(f(x_k) - f^*)^2}{M^2}} \leq \cosh(s_k) \sqrt{1 - \frac{1}{\zeta_r^2} \cdot \frac{\overline\Delta_T^2}{M^2}}.
\end{align*}
Therefore, $1 \leq \cosh(s_T) \leq \cosh(r) \bigg(\sqrt{1 - \frac{1}{\zeta_r^2} \cdot \frac{\overline\Delta_T^2}{M^2}}\bigg)^T$.
After taking logs and rearranging, 
$$\log(\cosh(r)) \geq -T \log\bigg(\sqrt{1 - \frac{1}{\zeta_r^2} \cdot \frac{\overline\Delta_T^2}{M^2}}\bigg) \geq \tanh(r)^2 \cdot \frac{T}{2}  \cdot \frac{\overline\Delta_T^2}{r^2 M^2},$$
and we conclude that $\overline\Delta_T^2 \leq \frac{2 \log(\cosh(r))}{\tanh(r)^2} \cdot \frac{r^2 M^2}{T} \leq 2 \zeta_r \frac{r^2 M^2}{T}.$
So subgradient descent with Polyak step size solves problem~\ref{highdimLipschitzproblem} in $O(\frac{\zeta_r}{\epsilon^2})$ queries.
We argue below that this is tight.
\begin{remark}
In Euclidean space, the Polyak step size can be derived in an analogous way, and leads to $\eta_k = \frac{f(x_k) - f^*}{\|g_k\|^2}$ and $s_{k+1}^2 = s_k^2 - \frac{(f(x_k) - f^*)^2}{\|g_k\|^2}$ (see~\citep[Sec.~5.3.2]{polyak1987book}).
\end{remark}

%

\begin{proof}[Proof of Corollary~\ref{corthisguysubrgd}]
Consider applying subgradient descent to the function $f$ defined in the proof of Theorem~\ref{lowerboundforsubrgd}.  
We show by induction on $k$ that the iterates produced by subgradient descent with Polyak step size are exactly $x_k = y_k$, and $s_k = r_k$.  This is clearly true for $k=0$.
Suppose $x_k = y_k$ and $s_k = r_k$.  
Then $f(x_k) - f^* = f(y_k) - f^* = \dist(x_k, x^*) + 0 = r_k$ (Lemma~\ref{lemma0}), and so
$$\cos(\theta_k) = \frac{f(x_k) - f^*}{s_k \|g_k\|} = \frac{r_k}{r_k \|g_k\|} = \frac{1}{\|g_k\|} = \cos(\theta)$$
using equation~\eqref{yetanotherusefuleqn}, and that the subgradient at $y_k$ returned by the oracle from Theorem~\ref{lowerboundforsubrgd} equals $g_k = \tilde g_k = - \frac{1}{\cos(\theta)} e_{k+1}^{(k)}$ (Lemma~\ref{lemma1}).
We conclude that $\theta_k = \theta$, and so $r_{k+1} = s_{k+1}$ and $\eta_k \|g_k\| = \Delta_k$ by equations~\eqref{yetanotherusefuleqn} and~\eqref{yetanotherusefulguy25}.
Therefore, $x_{k+1} = \exp_{x_k}(-\Delta_k \frac{g_k}{\|g_k\|}) = \exp_{y_k}(\Delta_k e_{k+1}^{(k)}) = y_{k+1}$, which concludes the induction.

By Lemma~\ref{lemma0}, $y_0, \ldots, y_{d-1} \in L_{k, x^*}$ for all $k$, so subgradient descent with Polyak step size satisfies assumption~\aref{A2}.
Also, clearly subgradient descent satisfies~\aref{A1}.  
We can therefore, apply Theorem~\ref{lowerboundforsubrgd}.
\end{proof}

\subsubsection{Minimal ball containing $B(x_k, r_k) \cap \{x \in \calM : \langle g_k, \log_{x_k}(x) \rangle \leq f^* - f(x_k)\}$} \label{appdetailsforpolyakstepsize}
In this section we derive explicit formulae for the center $x_{k+1}$ of the minimal ball and its radius $s_{k+1}$.
By symmetry $x_{k+1} = \exp_{x_k}(-\eta_k g_k)$ for some $\eta_k \in \reals$.

Take $v \in \T_{x_k} \calM$ so that $\|v\| = s_k$ and $\langle g_k, v \rangle = f^* - f(x_k)$.  Let $\theta_k \in [0, \frac{\pi}{2}]$ denote the angle between $-g_k$ and $v$.
It satisfies $\cos(\theta_k) = \frac{f(x_k) - f^*}{s_k \|g_k\|}$.
Let $z$ be the point on the geodesic $t \mapsto \exp_{x_k}(-t g_k)$
such that the hyperbolic triangle $x_k, \exp_{x_k}(v), z$ has right angle at $z$.
Define $s_{k+1}' = \dist(\exp_{x_k}(v), z)$.
A computation (details omitted) shows that in fact $B(z, s_{k+1}')$ contains $B(x_k, s_k) \cap \{x \in \calM : \langle g_k, \log_{x_k}(x) \rangle \leq f^* - f(x_k)\}$ and is the unique minimal ball, and so we find that $x_{k+1} = z$ and $s_{k+1} = s_{k+1}'$.

To determine $s_{k+1}$, hyperbolic trigonometry on the right triangle $x_k, \exp_{x_k}(v), z = x_{k+1}$ implies $\frac{f(x_k) - f^*}{s_k \|g_k\|} = \cos(\theta_k) = \tanh(\eta_k \|g_k\|) / \tanh(s_k)$, $\sin(\theta_k) = \sinh(s_{k+1}) / \sinh(s_k)$,
and 
\begin{align*}
\cosh(s_{k+1}) &= \cosh(s_k) / \cosh(\eta_k \|g_k\|) = \cosh(s_k) \sqrt{1 - \tanh(\eta_k \|g_k\|)^2} \\
&= \cosh(s_k) \sqrt{1 - \frac{1}{\zeta_{s_k}^2} \cdot \frac{(f(x_k) - f^*)^2}{\|g_k\|^2}},
\end{align*}
where the last equality follows from $\frac{f(x_k) - f^*}{s_k \|g_k\|} = \tanh(\eta_k \|g_k\|) / \tanh(s_k)$.

\section{Lower bound for the cutting-planes game: Proof of Theorem~\ref{lowerboundforfeasibilityproblem}} \label{feasprobleapplowerbound}
We first need the following bounds on the volume of a hyperbolic ball.  Let $V_d(r)$ be the hyperbolic volume of a ball of radius $r$ in $\mathbb{H}^d$.  Let $\omega_d, \sigma_{d}$ be the volume and surface area of the $d$-dimensional ball, sphere, respectively (as a subset of Euclidean space).

\begin{lemma} \label{Lemmavolumes}
We have $V_d(r) = \omega_d \int_0^r \sinh^{d-1}(t) dt \leq \frac{\omega_d}{(d-1) 2^{d-1}} e^{r (d-1)}$ for all $r > 0$.  If $r \geq 4 \log(d)$ then $V_d(r) \geq \frac{\omega_d}{4 (d-1) 2^{d-1}} e^{r (d-1)}$.
\end{lemma}
\begin{proof}
The formula $V_d(r) = \omega_d \int_0^r \sinh^{d-1}(t) dt$ is given by~\citet[Ex.~3.4.6]{ratcliffe2019hyperbolic}.
The upper bound follows from $\sinh(t) \leq e^t/2$ for all $t \geq 0$.  The lower bound follows from the inequality
$$\sinh(t) \geq \frac{1}{2 \cdot 2^{1/(d-1)}} e^{t} \quad \forall t \geq 2 \log(d).$$
Taking $a = 2\log(d)$ and using that $r \geq 2 a$, we have
\begin{align*}
V_d(r) \geq \omega_d \int_{a}^r \sinh(t)^{d-1} dt &\geq \frac{\omega_d}{2 \cdot 2^{d-1}} \int_{a}^r e^{t(d-1)} dt = \frac{\omega_d}{2 (d-1) 2^{d-1}} [e^{r(d-1)} - e^{a (d-1)}] \\
&\geq \frac{\omega_d}{4 (d-1) 2^{d-1}} e^{r (d-1)}.
\end{align*}
\end{proof}

The following Lemma~\ref{lemmafrombak} is due to~\citet{hyperbolicintersectiongraphs}.  Following the definitions from~\citep[Sec. 2]{hyperbolicintersectiongraphs}, a tiling of $\calM$ is a collection of interior disjoint compact subsets of $\calM$ which cover $\calM$.  A tiling is called $(\rho_1, \rho_2)$-nice if each tile $T \in \mathcal{T}$ contains a ball of radius $\rho_1$ with center denoted $c_T$, and is contained in a ball of radius $\rho_2$.  For $p \in \calM$, we define a ``random totally geodesic submanifold'' through $p$ as follows:  $S = \exp_{p}(\{v \in \T_p \calM : \langle g, v\rangle = 0\})$ is the totally geodesic submanifold whose unit normal vector $g \in \T_p \calM$ is chosen uniformly at random from the $(d-1)$-dimensional unit sphere. \citet{hyperbolicintersectiongraphs} does not work out the dependence on the dimension $d$ in Lemma~\ref{lemmafrombak}, so we do this in Appendix~\ref{applemmafrombak}.

\begin{lemma} \label{lemmafrombak}\citep[Lem. 9]{hyperbolicintersectiongraphs}
Let $d \geq 3$, $c_0 \in (0, 1]$ and $\tau \geq 8 c_0^{-1} \log(d)$.  Let $\mathcal{T}$ be a $(c_0 \tau/2, \tau/2)$-nice tiling, and let $\mathcal{I}$ be a finite subset of $\mathcal{T}$.  Let $p \in M$ and $S$ be a random totally geodesic submanifold through $p$.  Then
$\mathbb{E}[|\mathcal{S}|] \leq \frac{1}{2} e^{5 (d-1) \tau} |\mathcal{I}|^{\frac{d-2}{d-1}},$
where $\mathcal{S} = \{T \in \mathcal{I} : T \cap S \neq \emptyset\}$.
\end{lemma}

\begin{proof}[Proof of Theorem~\ref{lowerboundforfeasibilityproblem}]
Define $c_0 = 1/4$, $\tau = \frac{r}{40 (d-1)}$.  Then $\tau \geq 8 c_0^{-1} \log(d)$, and there exists a $(\epsilon r = c_0 \tau / 2, \tau/2)$-nice tiling $\mathcal{T}$ of $\calM$, as explained in Appendix~\ref{applemmafrombak}.
Let $\mathcal{I}_0$ be the set of tiles in $\mathcal{T}$ contained in $B(\xorigin, r)$, and define $A_0 = \{c_T : T \in \mathcal{I}_0\}$.  $A_0$ is the set of initial possible candidates for $x^*$.  We know that the tiles $\mathcal{I}_0$ cover $B(\xorigin, r/2)$.  
Using the bounds on $V_d(r)$ given by Lemma~\ref{Lemmavolumes}, we have $|A_0| = |\mathcal{I}_0| \geq \frac{V_d(r/2)}{V_d(\tau / 2)} \geq \frac{1}{4} e^{\frac{1}{4} (d-1) r}.$

For every $k = 0, 1, \ldots, T-1$, the algorithm makes the query $x_k$, and we show (by induction) that there is a unit vector $g_k \in \T_{x_k} \calM$ and a set $\mathcal{I}_{k+1} \subseteq \mathcal{I}_{k}$ such that $|\mathcal{I}_{k+1}| \geq |\mathcal{I}_k| / 4$, and each $x^* \in A_{k+1} = \{c_T : T \in \mathcal{I}_{k+1}\}$ is consistent with the past queries and such that the algorithm has not won the game, i.e., 
$$\langle g_{\ell}, \log_{x_{\ell}}(x^*) \rangle \leq 0, \quad \text{and} \quad S_\ell \cap B(x^*, \epsilon r) = \emptyset, \quad \forall \text{ } \ell \in \{0, \ldots, k\} \text{ and } \forall \text{ } x^* \in A_{k+1}.$$
Indeed, fix $k \in \{0, 1, \ldots, T-1\}$.  By Lemma~\ref{lemmafrombak} there exists $\tilde g_k \in \T_{x_k} \calM$ such that 
$$|\{T \in \mathcal{I}_k : B(c_T, \epsilon r) \cap S_k \neq \emptyset\}| \leq |\{T \in \mathcal{I}_k : T \cap S_k \neq \emptyset\}| \leq \frac{1}{2} e^{5 (d-1) \tau} |\mathcal{I}|^{\frac{d-2}{d-1}}$$
where as usual $S_k = \exp_{x_k}(\{v \in \T_{x_k} \calM : \langle g_k, v \rangle = 0\})$.
For $s \in \{+1, -1\} = \{+, -\}$, define
\begin{align*}
\mathcal{I}_{k+1}^s = \{T \in \mathcal{I}_k : s \langle g_k, \log_{x_k}(y) \rangle \geq 0, \text{  } \forall y \in B(c_T, \epsilon r)\}.
\end{align*}
If $|\mathcal{I}_{k+1}^+| \geq |\mathcal{I}_{k+1}^-|$, define $g_k = \tilde g_k, \mathcal{I}_{k+1} = \mathcal{I}_{k+1}^+$.  Otherwise, define $g_k = - \tilde g_k, \mathcal{I}_{k+1} = \mathcal{I}_{k+1}^-$.
So, 
$$|\mathcal{I}_{k+1}| \geq \frac{1}{2}\bigg(|\mathcal{I}_{k}| - |\{T \in \mathcal{I}_k : B(c_T, \epsilon r) \cap S_k \neq \emptyset\}|\bigg) \geq \frac{1}{2} |\mathcal{I}_{k}| \bigg(1 - \frac{1}{2} e^{5 (d-1) \tau} |\mathcal{I}_k|^{-1/(d-1)}\bigg).$$
By induction, we know that
$$|\mathcal{I}_{k}| \geq \frac{1}{4^k} |\mathcal{I}_{0}| \geq \frac{1}{4^{T-1}} |\mathcal{I}_{0}| \geq \frac{1}{4^T} e^{\frac{1}{4} (d-1) r} \geq \frac{1}{e^{2T}} e^{\frac{1}{4} (d-1) r} \geq e^{\frac{1}{8} (d-1) r} \geq e^{5 (d-1)^2 \tau}$$
where the penultimate inequality follows from $T = \frac{1}{16} (d-1) r$, and the last inequality follows from $\tau = \frac{r}{40 (d-1)}$.  Therefore $|\mathcal{I}_{k}|^{-1/(d-1)} \leq e^{-5 (d-1) \tau}$, and so $|\mathcal{I}_{k+1}| \geq \frac{1}{4} |\mathcal{I}_{k}|$, finishing the induction.

We know $A_T$ is nonempty because $|A_0| \geq \frac{1}{4} e^{\frac{1}{4} (d-1) r}$ and $|A_{T}| \geq \frac{1}{4^T} |A_0| \geq \frac{1}{4^{T+1}} e^{\frac{1}{4} (d-1) r} \geq 1$.  So there exists $x^* \in A_T$ such that $\langle g_{\ell}, \log_{x_{\ell}}(x^*) \leq 0$ and $\dist(S_\ell, x^*) > \epsilon r$ for all $\ell = 0, \ldots, T$.
\end{proof}

\section{Geodesically convex interpolation: a closer look than Section~\ref{interpolation}}\label{appinterpolation}
In this section, we initiate a study of interpolation by g-convex functions.
We leave a more complete study of interpolation to future work.
Let $\calM$ be a Hadamard manifold.
We call a collection $(F_i, x_i, g_i)_{i=1}^N$ the ``data'' of an interpolation problem if $x_i \in \calM, F_i \in \reals$ and $g_i \in \T_{x_i}\calM$ for all $i=1, \ldots, N$.
We say the data is \emph{interpolated} by a g-convex function $f \colon \calM \rightarrow \reals$ if $f(x_i) = F_i$ and $g_i \in \partial f(x_i)$ for all $i$.
We say a function $f \colon \calM \to \reals$ is the \emph{minimal} (\emph{maximal}) g-convex function interpolating the data $(F_i, x_i, g_i)_{i=1}^N$ if for all other g-convex functions $\tilde f \colon \calM \to \reals$ we have $\tilde f(x) \geq f(x)$ ($\tilde f(x) \leq f(x)$) for all $x \in \calM$.

Consider an algorithm running on a g-convex function $f \colon \calM \to \reals$ querying points $x_0, x_1, \ldots$.  
Given $F_k = f(x_k)$ and $g_k \in \partial f(x_k)$, the algorithm learns that $f$ is lower bounded as
\begin{align} \label{lowerboundd}
f(x) \geq F_k + \langle g_k , \log_{x_k}(x)\rangle \quad \forall x\in \calM.
\end{align}
Since the data $(F_\ell, x_\ell, g_\ell)_{\ell=0}^k$ is interpolated by the g-convex function $f$, it must satisfy
\begin{align}\label{necessaryconditions}
F_j \geq F_i + \langle g_i, \log_{x_i}(x_j) \rangle \quad \forall i, j \in \{0, \ldots, k\}.
\end{align}
Aggregating all the lower bounds~\eqref{lowerboundd} from past queries, the algorithm knows that $f$ satisfies
$$f(x) \geq f_k(x) \quad \forall x\in \calM, \quad \text{ where } \quad f_k(x) = \max_{\ell \in \{0, \ldots, k\}}\{F_\ell  + \langle g_\ell, \log_{x_\ell}(x)\rangle\}.$$
Using the conditions~\eqref{necessaryconditions}, one sees that this lower bound $f_k$ does interpolate the data $(F_\ell, x_\ell, g_\ell)_{\ell=0}^k$.

If $\calM$ is a Euclidean space, the functions $x \mapsto F_i + \langle g_i, \log_{x_i}(x) \rangle$ are affine, and so the lower bound $f_k$ is \emph{convex}: $f_k$ is the minimal convex function interpolating the data learned, and in particular the algorithm knows $f^* \geq \min_{x \in B(\xorigin, r)} \{f_k(x)\}$.
However for general $\calM$ (e.g., if $\calM$ is a hyperbolic space), the functions $x \mapsto F_i + \langle g_i, \log_{x_i}(x) \rangle$ are not g-convex, and the lower bound $f_k$ is \emph{not g-convex}.
This points to a primary difficulty inherent in proving complexity lower bounds for g-convex optimization.
Perhaps when the algorithm learns that $F = f(y)$ and $g \in \partial f(y)$, actually lower bound $f(x) \geq F + \langle g, \log_y(x) \rangle$ is loose and can be improved?  This is not the case.
\begin{proposition} \label{mingconvex}
Let $\calM$ be a Hadamard manifold, $F \in \reals$, and $(y,g) \in \T\calM$.
Let $\calF_{F, y,g}$ denote the set of g-convex functions $f \colon \calM \to \reals$ with $f(y) = F$ and $g \in \partial f(y)$.
Then for all $x \in \calM$
$$\min_{f \in \calF_{F,y,g}} \{f(x)\} = F + \langle g, \log_{y}(x)\rangle.$$
\end{proposition}
\begin{proof}
Fix $x \in \calM \setminus \{y\}$.
Define $g_{||} = \frac{\langle g, \log_y(x) \rangle}{\dist(x,y)^2} \log_y(x)$ and $g_{\perp} = g - g_{||}$.
Let $S$ be the geodesic line $\{\exp_y(t \log_y(x)) : t \in \reals\}$.  Note that $S$ is a g-convex set, and so $z \mapsto \dist(z, S)$ is g-convex.

Define the point $x'$ as $x' = x$ if $\langle g, \log_y(x) \rangle \leq 0$, and $x' = \exp_y(-\log_{y}(x))$ if $\langle g, \log_y(x) \rangle > 0$.
Define the functions
$$f_{||}(z) = F - \|g_{||}\| \dist(x', y) + \|g_{||}\| \dist(x', z), \quad f_{\perp}(z) = \|g_{\perp}\| \dist(z, S), \quad f(z) = f_{||}(z) + f_{\perp}(z).$$
We know that $g_{\perp}$ is a subgradient of $f_{\perp}$ at $y$, since $g_{\perp}$ is orthogonal to $S$, and
$$\nabla f_{||}(y) = -\frac{\|g_{||}\|}{\dist(x,y)} \log_y(x') = - \frac{|\langle g, \log_y(x) \rangle|}{\dist(x,y)^2} \log_y(x') = \frac{\langle g, \log_y(x) \rangle}{\dist(x,y)^2} \log_y(x) = g_{||}.$$
Therefore, $\nabla f_{||}(y) + g_{\perp} = g_{||} + g_{\perp} = g$ is a subgradient of $f$ at $y$.
Also, $f(y) = F + \|g_{\perp}\| \dist(y, S) = F$, and $f$ is g-convex (as the sum of two g-convex functions).
So $f \in \calF_{F, y, g}$.

Lastly, if $\langle g, \log_y(x) \rangle \leq 0$, then $f(x) = F - \|g_{||}\| \dist(x, y) = F + \langle \log_y(x), g \rangle$, using $\dist(x, S) = 0$ and $x = x'$.
If $\langle g, \log_y(x) \rangle > 0$, then $f(x) = F + \|g_{||}\| \dist(x, y) = F + \langle \log_y(x), g \rangle$, using $\dist(x, S) = 0$ and $\dist(x, x') = 2 \dist(x,y)$.
\end{proof}
\begin{remark}
Observe that the worst function in the world given by equation~\eqref{formulaforfwfinw} is essentially a maximum of functions from the proof of Proposition~\ref{mingconvex}.  This is not a coincidence.
\end{remark}

We know that $f_k$ is not g-convex.  However, perhaps the necessary conditions~\eqref{necessaryconditions} for interpolation by a g-convex function are also sufficient?
The necessary conditions are sufficient if $\calM$ is a Euclidean space; they are \emph{not sufficient} for negatively curved spaces even for just three points.
\begin{proposition}[Restatement of Proposition~\ref{nointerpolation}] \label{thisguyonemore}
Let $\calM = \mathbb{H}^2$.
There exists data $(F_i, x_i, g_i)_{i=1}^3$ such that
\begin{align} \label{thisguynecinterpconditions}
F_j \geq F_i + \inner{g_i}{\log_{x_i}(x_j)} \quad \quad \forall i, j \in \{1,2,3\},
\end{align}
yet this data cannot be interpolated by any g-convex function.
\end{proposition}
\begin{proof}
We denote the geodesic segment $\{\exp_x(t \log_x(y)) : t \in [0,1]\}$ between $x, y \in \calM$ by $[x,y]$.

Consider three points $x_1, x_2, x_3 \in \calM$ such that $\dist(x_1, x_2) = \dist(x_1, x_3) = 1$ and the angle (at $x_1$) between the geodesic segments $[x_1,x_2]$ and $[x_1,x_3]$ equals $2 \theta$.
Let $h$ be the length of the altitude from $x_1$ to the geodesic segment $[x_2, x_3]$.  
Assume the altitude intersects the segment $[x_2, x_3]$ at point $p$.  
Of course $h = \dist(x_1, p)$, $p$ is the midpoint of $[x_2, x_3]$, the altitude bisects the angle at $x_1$, and the altitude intersects $[x_2, x_3]$ at a right angle.

Consider the data $F_1 = 1, F_2=F_3 = 0$, and let $g_1 = -\frac{1}{\cos(\theta)} \cdot \frac{\log_{x_1}(p)}{\dist(x_1, p)}$.
Assume, for the sake of contradiction, that the data can be interpolated by a g-convex function $f$.  
Then, $f(p) \leq \max\{f_1, f_2\} = 0$.
Additionally, the function $t \mapsto f(\exp_{x_1}(t \exp_{x_1}^{-1}(p)))$ is convex, so $f(p) \geq f(x_1) + \inner{g_1}{\log_{x_1}(p)} = 1 - \frac{h}{\cos(\theta)}$.

We have not used the non-Euclidean geometry of $\calM$.  We do so now.  By hyperbolic trigonometry, $\cos(\theta) = \frac{\tanh(h)}{\tanh(1)} < h$ (since $h < 1$).  Hence, $f(p) \geq 1 - \frac{h}{\cos(\theta)} > 0$.  This is a contradiction.

Lastly, it is easy to choose $g_2$ and $g_3$ so that the data $(F_i, x_i, g_i)_{i=1}^3$ satisfies~\eqref{thisguynecinterpconditions}.
For example, one may choose $g_2 = g_3 = 0$.  Alternatively, if we let $\alpha$ be the angle (at $x_2$) between the segments $[x_2, x_1]$ and $[x_2, x_3]$, then we can take $g_2, g_3$ to be tangent vectors of length $\frac{1}{\sin(\alpha)}$ which are perpendicular to the geodesic segment $[x_2, x_3]$.
\end{proof}

Let us take another perspective.
The set of convex functions on $\mathbb{R}^d$ equals the set of functions which are each a supremum of affine functions.  More generally, abstract convexity~\citep{abstractconvexity} considers classes of functions (the ``convex'' functions) which can be written as supremums of a base class of functions (the ``base'' functions).
Geodesic convexity does not naturally fit into this framework.
Indeed, it is easy to check that every g-convex function can be written as a supremum of functions of the form $x \mapsto F + \langle g, \log_y(x) \rangle$.
However, it is not true that every supremum of functions of this form are g-convex (e.g., consider finite supremums like $f_k$ above). 


Despite the negative result Proposition~\ref{nointerpolation}, we do, however, have the following necessary and sufficient conditions for interpolation.
\begin{proposition} \label{necsuffconditions}
Consider the data $(F_i, x_i, g_i)_{i=1}^N$, let $\mathcal{D} = (F_i, x_i)_{i=1}^N$, and let $\mathcal{F}_{\mathcal{D}}$ denote the set of g-convex functions $f$ satisfying $f(x_i) \leq F_i$ for $i=1, \ldots, N$.
Define the function
$$f_{\mathcal{D}} \colon \calM \to \reals \cup \{+\infty\}, \quad f_{\mathcal{D}}(x) = \sup_{f \in \mathcal{F}_{\mathcal{D}}} \{f(x)\}.$$
Then the data $(F_i, x_i, g_i)_{i=1}^N$ can be interpolated by a g-convex function if and only if
\begin{align}\label{necsuffconditionsappthm}
f_{\mathcal{D}}(x) \geq F_i + \langle g_i, \log_{x_i}(x)\rangle, \quad \forall x \in \calM \text{ and } \forall i = 1, \ldots, N.
\end{align}
If $(F_i, x_i, g_i)_{i=1}^N$ can be interpolated by a g-convex function, $f_{\mathcal{D}}$ is the maximal such function.
%
\end{proposition}
\begin{proof}
(Condition~\eqref{necsuffconditionsappthm} is sufficient) Assume~\eqref{necsuffconditionsappthm} holds.  
Taking $x = x_i$, $f_{\mathcal{D}}(x_i) \geq F_i$.  But by definition, $f_{\mathcal{D}} \leq F_i$.
Hence, $f_{\mathcal{D}}(x_i) = F_i$, and so condition~\eqref{necsuffconditionsappthm} implies  $g_i$ is a subgradient of $f_{\mathcal{D}}$ at $x_i$.
So $f_{\mathcal{D}}$ interpolates the data.

(Condition~\eqref{necsuffconditionsappthm} is necessary) Assume the data is interpolated by a g-convex function $f$.  
Then $f(x) \leq f_{\mathscr{D}}(x)$ for all $x \in \calM$.
Therefore for each $i$ and $x$, $f_{\mathscr{D}}(x) \geq f(x) \geq F_i + \langle g_i, \log_{x_i}(x) \rangle$.
So condition~\eqref{necsuffconditionsappthm} holds, and the previous paragraph implies $f_{\mathcal{D}}$ interpolates the data.  Therefore, $f_{\mathcal{D}}$ is the maximal g-convex function interpolating the data.
\end{proof}

\begin{remark}
One can relate $f_{\mathcal{D}}$ to a g-convex hull \citep{LytchakPetrunin} in the Hadamard manifold $\reals \times \calM$.  For simplicity, we omit the details.
One can verify that if $\calM$ is a hyperbolic space and $x$ is not in the g-convex hull of $\{x_1, \ldots, x_N\}$, then $f_{\mathcal{D}}(x) = +\infty$.
\end{remark}
If $\calM$ is a Euclidean space, then $f_{\mathcal{D}}$ has the explicit formula 
$$f_{\mathcal{D}}(x) = \min_{\lambda_i \geq 0, \sum_{i=1}^N \lambda_i = 1, \sum_{i=1}^N \lambda_i x_i = x} \bigg\{\sum_{i=1}^N \lambda_i F_i\bigg\},$$ 
and the set of necessary and sufficient conditions in Proposition~\ref{necsuffconditions} can shown to be equivalent to the conditions~\eqref{necessaryconditions}~\citep[Ch. 3]{taylorinterpolation}.
However, it is unclear that there is an analogous formula even for hyperbolic space.

Let us conclude this section with a positive result about g-convex interpolation.  For simplicity we focus on hyperbolic spaces.  
We have the following lemma which may be useful in other contexts.
\begin{lemma} \label{boundedhessaffine}
Let $\calM = \mathbb{H}^d$.  The function $x \mapsto \langle g, \log_y(x) \rangle$ is globally $\|g\|$-Lipschitz and $\|g\|$-smooth.
\end{lemma}
\begin{proof}
For $s, \dot{s} \in \T_y \calM$, let $P_{y, s} = P_{y \to \exp_y(s)}$, $T_{y, s} = d \exp_y(s)$ and $c_{y, s, \dot{s}}(t) = \exp_{y}(s + t \dot{s})$.
Given $s, \dot{s} \in \T_y \calM$, let $\dot{s}_{||} = \frac{\langle s, \dot{s} \rangle}{\|s\|^2} s$ and $\dot{s}_{\perp} = \dot{s} - \dot{s}_{||}$.
Moreover, given $\ddot{s} \in \T_{\exp_y(s)} \calM$, let $\ddot{s}_{||} = \frac{\langle P_{y, s} s, \ddot{s} \rangle}{\|s\|^2} s$ and $\ddot{s}_{\perp} = \ddot{s} - \ddot{s}_{||}$.
It is well-known (see for example \citep[App.~B]{criscitiello2020accelerated}) that 
\begin{align} \label{eq10}
T_{y, s}\dot{s} = P_{y, s}\dot{s}_{||} + \frac{\sinh(\|s\|)}{\|s\|} P_{y, s}\dot{s}_{\perp}, \quad \quad T_{y, s}^{-1}\ddot{s} = P_{y, s}^{-1}\ddot{s}_{||} + \frac{\|s\|}{\sinh(\|s\|)} P_{y, s}^{-1}\ddot{s}_{\perp}.
\end{align}

Define $f(x) = \langle g, \log_y(x)\rangle$, and $\hat f = f \circ \exp_y \colon \T_y \calM \to \reals$.
Note that $\hat{f}(s) = \inner{g}{s}$.
Let $x \in \calM$, $s = \log_y(x)$, and $\ddot{s} \in \T_x \calM$ with $\|\ddot{s}\|=1$.
By Lemma 5 of~\citep{agarwal2018arcfirst},
\begin{align} \label{eq11}
\nabla f(x) = (T_{y, s}^{-1})^*\nabla \hat f(s) = (T_{y, s}^{-1})^* g
\end{align}
and, defining $\dot{s} = T_{y, s}^{-1} \ddot{s}$,
\begin{align}\label{eq105}
\langle \ddot{s}, \nabla^2 f(x) \ddot{s} \rangle = \langle T_{y, s}^{-1} \ddot{s}, \nabla^2 \hat{f}(s) T_{y, s}^{-1} \ddot{s} \rangle - \langle \nabla f(x), c_{y, s, \dot{s}}''(0)\rangle = - \langle g, T_{y,s}^{-1} c_{y, s, \dot{s}}''(0)\rangle.
\end{align}
Equations~\eqref{eq10} and~\eqref{eq11} imply $\|\nabla f(x)\| \leq \|T_{y, s}^{-1}\| \cdot \|g\| \leq \|g\|$, so $f$ is $\|g\|$-Lipschitz.

By equation (73) from \citep[App.~B]{criscitiello2020accelerated}, we know that
$$c_{y, s, \dot{s}}''(0) = -\frac{\sinh(\|s\|)\cosh(\|s\|)-\|s\|}{\|s\|^3} \|\dot{s}_{\perp}\|^2 \cdot P_{y, s} s + 2 \frac{\cosh(\|s\|)\|s\| - \sinh(\|s\|)}{\|s\|^3} \inner{s}{\dot{s}} \cdot P_{y, s} \dot{s}_{\perp}.$$
Therefore, using equation~\eqref{eq10} and $\dot{s} = T_{y, s}^{-1} \ddot{s}$,
\begin{align*}
T_{y, s}^{-1} c_{y, s, \dot{s}}''(0) 
=& -\frac{\sinh(\|s\|)\cosh(\|s\|)-\|s\|}{\|s\|^3} \|\dot{s}_{\perp}\|^2 \cdot s \\
&+ 2 \frac{\|s\|}{\sinh(\|s\|)} \cdot \frac{\cosh(\|s\|)\|s\| - \sinh(\|s\|)}{\|s\|^3} \inner{s}{\dot{s}} \cdot  \dot{s}_{\perp} \\
=& -\frac{\sinh(\|s\|)\cosh(\|s\|)-\|s\|}{\|s\|^3} \frac{\|s\|^2}{\sinh^2(\|s\|)} \|\ddot{s}_{\perp}\|^2 \cdot s \\
&+ 2 \frac{\|s\|^2}{\sinh^2(\|s\|)} \cdot \frac{\cosh(\|s\|)\|s\| - \sinh(\|s\|)}{\|s\|^3} \inner{P_{y, s} s}{\ddot{s}} \cdot P_{y, s}^{-1} \ddot{s}_{\perp}.
\end{align*}
Using $\|\ddot{s}\| = 1$, we find
\begin{align*}
\|T_{y, s}^{-1} c_{y, s, \dot{s}}''(0) \|
\leq& \frac{\sinh(\|s\|)\cosh(\|s\|)-\|s\|}{\|s\|^3} \frac{\|s\|^2}{\sinh^2(\|s\|)} \|s\| \leq 1.
\end{align*}
By~\eqref{eq105}, we conclude 
$|\langle \ddot{s}, \nabla^2 f(x) \ddot{s} \rangle| \leq \|g\| \cdot \|T_{y, s}^{-1} c_{y, s, \dot{s}}''(0)\| \leq \|g\|.$
\end{proof}

Using this lemma, we can derive sufficient conditions for g-convex interpolation (Proposition~\ref{interpolationsuffcinditions}) and extension (Proposition~\ref{gconvexextension}).
\begin{proposition}[Restatement of Proposition~\ref{suffcondinteprolationguy}]\label{interpolationsuffcinditions}
Consider the data $(F_i, x_i, g_i)_{i=1}^N$.
Assume that $\|g_i\| \leq \frac{\mu}{2}$ for all $i = 1, \ldots, N$, and
\begin{align}\label{thisinequex}
F_j \geq F_i + \langle g_i, \log_{x_i}(x_j)\rangle + \frac{\mu}{2} \dist(x_i, x_j)^2 \quad \forall i, j \in \{1, \ldots, N\}.
\end{align}
Then $(F_i, x_i, g_i)_{i=1}^N$ is interpolated by a $\frac{\mu}{2}$-strongly g-convex function.
\end{proposition}
\begin{proof}
Let $f(x) = \max_{i \in \{1, \ldots, N\}}\big\{ F_i + \langle g_i, \log_{x_i}(x)\rangle + \frac{\mu}{2} \dist(x_i, x)^2 \big\}$.
Taking $i=j$, we know $f(x_j) \geq F_j$.  Inequality~\eqref{thisinequex} implies $f(x_j) \leq F_j$, so we conclude $f(x_j) = F_j$, for all $j$.
Moreover, 
$$f(x) \geq F_i + \langle g_i, \log_{x_i}(x_j)\rangle + \frac{\mu}{2} \dist(x_i, x)^2 \geq f(x) + \langle g_i, \log_{x_i}(x_j)\rangle, \quad \forall x\in\calM, \forall i \in \{1, \ldots, N\},$$
so $g_i \in \partial f(x_i)$.  So $f$ interpolates the data.

By Lemma~\ref{boundedhessaffine}, each function $x \mapsto F_i + \langle g_i, \log_{x_i}(x)\rangle + \frac{\mu}{2} \dist(x_i, x)^2$ is $\mu - \|g_i\| \geq \mu - \frac{\mu}{2} = \frac{\mu}{2}$-strongly g-convex.
Therefore, $f$ is $\frac{\mu}{2}$-strongly g-convex.
\end{proof}

We have the following result (compare with the convex case~\citep{cvxextension}):
\begin{proposition}\label{gconvexextension}
Let $D$ be an open g-convex subset of $\calM = \mathbb{H}^d$ and let $f \colon D \to \reals$ be differentiable and $\mu$-strongly g-convex in $D$.
If $\|\nabla f(x)\| \leq \mu$ for all $x \in D$, then there is a globally g-convex function $\tilde f \colon \calM \to \reals$ such that $\tilde f(x) = f(x)$ for all $x \in D$.
\end{proposition}
\begin{proof}
Define 
$$\tilde f(x) = \sup_{y \in D}\Big\{ f(y) + \langle \nabla f(y), \log_y(x) \rangle + \frac{\mu}{2} \dist(x,y)^2\Big\}.$$
If $x \in D$, then $\tilde f(x) \geq f(x)$ (taking $y = x$).  On the other hand, $\mu$-strong g-convexity of $f$ in $D$ implies $f(x) \geq f(y) + \langle \nabla f(y), \log_y(x) \rangle + \frac{\mu}{2} \dist(x,y)^2$ for all $y \in D$.  Therefore, $f(x) \geq \tilde f(x)$.  We conclude $f(x) = \tilde f(x)$ if $x \in D$.

By Lemma~\ref{boundedhessaffine}, the function $x \mapsto f(y) + \langle \nabla f(y), \log_y(x) \rangle$ is $\|\nabla f(y)\|$-smooth.
Hence, the function $x \mapsto f(y) + \langle \nabla f(y), \log_y(x) \rangle + \frac{\mu}{2} \dist(x,y)^2$ is $\mu - \|\nabla f(y)\| \geq \mu - \mu = 0$-strongly g-convex, if $y \in D$.
Therefore, $\tilde f$ is g-convex.
\end{proof}

\section{Carrying over the convex upper and lower bounds when $r$ is very small} \label{carryovereuc}
%
It is important to mention that if the optimization is carried out in a very small region then smooth g-convex optimization reduces to Euclidean convex optimization, and in particular upper and lower bounds from Euclidean space carry over.
However, we shall see that the resulting lower bounds are worse than those given in Section~\ref{secsmoothlowerbounds}.
This technique does, however, allow us to easily construct an ``eventually accelerated'' algorithm for smooth strongly g-convex optimization, like the algorithm of \citet{ahn2020nesterovs}.  See below.

To state the following proposition, we need properties of the Riemannian curvature tensor $R$ and its covariant derivative $\nabla R$ which are worked out in~\citep[Sec.~2]{criscitiello2020accelerated}.  
The assumption $\norm{\nabla R} \leq F$ holds with $F = 0$ for symmetric spaces, including hyperbolic spaces.
\begin{proposition} \label{gconvextoconvexwhenrsmall}
Let $\calM$ be a complete Riemannian manifold which has sectional curvatures in the interval $[-K, K]$ with $K \geq 0$ and also satisfies $\norm{\nabla R(x)} \leq F$ for all $x \in \calM$.

Let $\xorigin \in \calM$.
Let $f \colon \calM \rightarrow \reals$ be a twice continuously differentiable function which has a minimizer $x^*$.
Assume $f$ is $\mu$-strongly g-convex and $L$-smooth in $B(\xorigin, r)$, $\kappa = \frac{L}{\mu}$, and $x^*$ is contained in $B(\xorigin, r)$, where $r\in (0, {\frac{1}{\sqrt{\kappa}}} \cdot \min\{\frac{1}{4 \sqrt{K}}, \frac{K}{4 F}\}]$.
Define the pullback $\hat{f} = f \circ \exp_{\xorigin}$.
Then, $\hat{f}$ is $\frac{\mu}{2}$-strongly convex and $\frac{3}{2} L$-smooth in the Euclidean sense in $B_{\xorigin}(0, r) \subseteq \T_{\xorigin}\calM$.
\end{proposition}
\begin{proof}
Following~\citep[Sec.~2]{criscitiello2020accelerated}, if $\norm{s} \leq \min\{\frac{1}{4 \sqrt{K}}, \frac{K}{4 F}\}$ then
$$\norm{\nabla^2 \hat{f}(s) - P_{\xorigin, s}^* \nabla^2 f(y) P_{\xorigin, s}} \leq  \frac{7}{9} L K \norm{s}^2
+ \frac{3}{2} K \norm{s} \norm{\nabla f(y)}$$
where $y = \exp_{\xorigin}(s)$ and $P_{x, s}$ denotes parallel transport along the curve $c(t) = \exp_x(ts)$ from $t=0$ to $t=1$.
Using $L$-smoothness of $f$ and $x^* \in B(\xorigin, r)$, 
$$\norm{\nabla f(y)} = \norm{\nabla f(y) - P_{x^* \rightarrow y} \nabla f(x^*)} \leq L \dist(y, x^*) \leq 2 L r.$$
We determine
$\norm{\nabla^2 \hat{f}(s) - P_{\xorigin, s}^* \nabla^2 f(y) P_{\xorigin, s}} \leq 4 L K r^2 \leq \frac{\mu}{2}$
if $\norm{s} \leq r$, by our choice of $r$.

By $\mu$-strong g-convexity and $L$-smoothness of $f$, we know all eigenvalues of the symmetric operator $P_{\xorigin, s}^* \nabla^2 f(y) P_{\xorigin, s}$ are in $[\mu, L]$.
Hence, all eigenvalues of $\nabla^2 \hat{f}(s)$ are in the interval $[\frac{\mu}{2}, L + \frac{\mu}{2}]$ for all $\norm{s} \leq r$.
\end{proof}

We can use this proposition to prove an $\Omega(\sqrt{\kappa})$ lower bound when minimizing a $\kappa$-conditioned smooth strongly g-convex function in a ball of radius $r \leq O(\frac{1}{\sqrt{\kappa}})$.
Indeed, any algorithm $\calA$ querying points $x$ on $\calM$ can be converted into an algorithm $\tilde \calA$ querying points in a fixed tangent space $\T_{\xorigin} \calM$ using the exponential map in the obvious way.
The lower bound for Euclidean spaces provides a hard function $\tilde{f} \colon \T_{\xorigin}\calM \rightarrow \reals$ for $\tilde{\calA}$. 
In turn, the function $f = \tilde{f} \circ \log_{\xorigin}$ is a hard function for $\calA$.
Proposition~\ref{gconvextoconvexwhenrsmall} guarantees $f$ is g-convex on $\calM$ in a small enough ball of radius $r$.
Likewise, we can prove an $\Omega(\frac{1}{\sqrt{\epsilon}})$ lower bound for minimizing smooth g-convex functions in a ball of radius $r \leq O(\sqrt{\epsilon})$.

However, \emph{these lower bounds have two major downsides}.
First, the functions constructed are not globally g-convex (or $L$-smooth) and so do not actually belong to the function classes we defined in Section~\ref{problemclasses}.
Second, these lower bounds only hold when $r$ is very small, e.g., $r \leq O(\sqrt{\epsilon})$.
Ignoring the first downside, one can still try extend these bounds to large $r$; however, this results in unsatisfactory bounds.
For example, for smooth g-convex optimization, roughly we have that $\Omega(\frac{1}{\sqrt{\epsilon}})$ queries are required to find a point $x$ such that $f(x) - f^* \leq \epsilon \cdot \frac{1}{2} L (\sqrt{\epsilon})^2 = \frac{\epsilon^2}{r^2} \cdot \frac{1}{2} L r^2$.  Letting $\epsilon' = \frac{\epsilon^2}{r^2}$, we find that $\Omega(\frac{1}{\sqrt{r} (\epsilon')^{1/4}})$ queries are needed to find a point $x$ such that $f(x) - f^* \leq \epsilon' \cdot \frac{1}{2} L r^2$.
This bound does not even have the correct scaling in $\epsilon'$!

Lastly, we mention that we can use Proposition~\ref{gconvextoconvexwhenrsmall} to construct an ``eventually accelerated'' algorithm for smooth strongly g-convex optimization, like the algorithm of \citet{ahn2020nesterovs}.  
Let $\calM$ be a complete Riemannian manifold with constants $K$ and $F$ as in Proposition~\ref{gconvextoconvexwhenrsmall}.
Consider a function $f \colon \calM \rightarrow \reals$ with a minimizer $x^*$.  Assume $f$ is globally $\mu$-strongly g-convex and $L$-smooth in some g-convex set $D$ containing the ball $B(x^*, r)$ where $r = {\frac{1}{\sqrt{\kappa}}} \cdot \min\{\frac{1}{4 \sqrt{K}}, \frac{K}{4 F}\}$.
Suppose $x_0 \in B(x^*, r)$ and $B(x_0, r) \subseteq D$.
Since $\hat{f}$ has condition number at most $3 \kappa$ in $B_{x_0}(0, r) \subseteq \T_{x_0}\calM$, we can minimize $f$ by simply running a Euclidean algorithm on $\hat{f} = f \circ \exp_{x_0}$ in the ball $B_{x_0}(0, r)$.
For example, running an accelerated algorithm for constrained optimization (like from~\citep[Sec.~5.1]{nesterovacceleratedgradientinconvexset2007}) produces a \emph{locally} accelerated algorithm for g-convex optimization.  

This gives a simple proof that when $\dist(x_0, x^*) \leq O({\frac{1}{\sqrt{\kappa}}})$, there is an accelerated algorithm.
We can combine this algorithm with Riemannian gradient descent (RGD) to get a globally convergent algorithm which performs at least as well as RGD and which is eventually accelerated.
Consider running $T = \tilde \Theta(\kappa)$ steps of the algorithm we just described (with $r = {\frac{1}{\sqrt{\kappa}}} \cdot \min\{\frac{1}{4 \sqrt{K}}, \frac{K}{4 F}\}$) in parallel with $T$ steps of projected RGD on $f$ (in $D$), both initialized at the same point $x_0 \in D$.  
After these $T$ steps, we set $x_1 \in D$ to be the last iterate produced by one of the algorithms based on function value and we repeat starting from $x_1$, etc.
By design, this algorithm performs at least as well as RGD in terms of function value.
By our choice of $T = \tilde \Theta(\kappa)$, once an iterate $x_k$ of the algorithm enters $B(x^*, r)$, then the distance to the minimizer of the subsequent outer iterates $x_{k+1}, x_{k+2}, \ldots$ is decreasing.  So once this method is restarted in $B(x^*, r)$, it stays in this ball, and the NAG sequence in a tangent space provides cost function value decrease at an accelerated rate.  
Therefore this algorithm is eventually accelerated.  
We note that the algorithms of \citet{zhang2018estimatesequence} and \citet{ahn2020nesterovs} do not require a bound on $\norm{\nabla R}$ while the algorithm we presented does.

\section{Proof of Theorem~\ref{nonsmoothlowerbound}: showing $x^* \in \bigcap_{k=0}^{T-1} S_{i_k}^{s_k}$} \label{geoproofinlowerbound}
Consider the hyperbolic triangle formed by $x^*, \xorigin, z_{i_k}^{s_k}$.
Let $\beta \in [0, \pi]$ be the angle between $\log_{z_{i_k}^{s_k}}(x^*)$ and $\log_{z_{i_k}^{s_k}}(\xorigin)$.
Let $b = \dist(x^*, z_{i_k}^{s_k})$.
The angle between $\log_{\xorigin}(x^*)$ and $\log_{\xorigin}(z_{i_k}^{s_k}) = a s_k e_{i_k}$ equals $\theta = \arccos(1/\sqrt{d})$ for all $k=0, \ldots, T-1$.  
The hyperbolic law of cosines~\citep[Ch. 3.5]{ratcliffe2019hyperbolic} and $\cos(\theta) = \frac{1}{\sqrt{d}} = \frac{\tanh(a)}{\tanh(r)}$ imply
\begin{align*}
\cosh(b) &= \cosh(a) \cosh(r) - \sinh(a) \sinh(r) \cos(\theta)\\
&= \cosh(a) \cosh(r) - \sinh(a) \cosh(r) \tanh(a) \\
 &= \frac{\cosh(r)}{\cosh(a)}[\cosh(a)^2 - \sinh^2(a)] = \frac{\cosh(r)}{\cosh(a)}.
\end{align*}
One more application of the hyperbolic law of cosines implies 
$$\cosh(r) = \cosh(a) \cosh(b) - \sinh(a)\sinh(b) \cos(\beta) = \cosh(r) - \sinh(a)\sinh(b) \cos(\beta).$$  
Therefore, $\cos(\beta) = 0$, i.e., $\langle\log_{z_{i_k}^{s_k}}(\xorigin), \log_{z_{i_k}^{s_k}}(x^*) \rangle = 0$.
So, $x^* \in S_{i_k}^{s_k}$ for each $k$, as claimed.

\section{Extending the $\tilde \Omega(\zeta)$ lower bound: details from Section~\ref{nonstronglyconvexcaseextension}} \label{extendingthelbgeneral}

\subsection{Hadamard manifolds of bounded curvature}\label{thisguyappthis}
In this section, we prove the following generalization of Theorem~\ref{nicetheorem}.
\begin{theorem}[Generalization of Theorem~\ref{nicetheorem}] \label{theoremtheorem}
Let $\calM$ be a Hadamard manifold of dimension $d \geq 2$ whose sectional curvatures are in the interval $[\Klo, \Kup]$ with $\Kup < 0$.
Let $\tilde{L} > 0$ and $r \geq \frac{64}{\sqrt{-\Kup}}$.
Define
$$\hat{\epsilon} = \frac{1}{2^{10} r \sqrt{-\Klo}}, \quad \epsilon = \frac{1}{2^{18} \zeta_{r \sqrt{-\Klo}} \log^2(\zeta_{r \sqrt{-\Klo}})}, \quad \epsilon' = \frac{1}{2^{18} \log^2(\zeta_{r \sqrt{-\Klo}})},$$
and $L = \frac{\tilde L}{2^{10} \log^2(r \sqrt{-\Klo})}, \tilde M = \frac{\tilde{L}}{2 \sqrt{-\Klo}}$.
Let $\calA$ be any deterministic algorithm.

There is a $C^\infty$ function $\tilde{f} \in \mathcal{F}_{r, \tilde M, 0, \tilde{L}}^{\xorigin}$ with unique minimizer $x^*$ such that running $\calA$ on $\tilde{f}$ yields iterates $x_0, x_1, x_2, \ldots$ satisfying 
$$\tilde{f}(x_k) - \tilde{f}(x^*) \geq 2 \hat \epsilon L r^2 \geq \epsilon' \cdot \tilde M r \geq \epsilon \cdot \frac{1}{2} \tilde{L} r^2$$
for all $k = 0, 1, \ldots, T-1$, where $T = \Bigg\lfloor \frac{\zeta_{r \sqrt{-\Kup}}}{50 \log\big(64 \zeta_{r \sqrt{-\Klo}}\big)}\Bigg\rfloor$.
\end{theorem}
Towards this end, we have the following generalization of Lemma~\ref{lemmausefulguysec3}.

\begin{lemma}[Generalization of Lemma~\ref{lemmausefulguysec3}] \label{lemmaaboutintermediaryfunctionclass}
Let $\calM$ be a Hadamard manifold of dimension $d \geq 2$ which satisfies the ball-packing property (see A1 in~\citep{bumpfctpaper}) with constants $\tilde{r}, \tilde{c}$ and point $\xorigin \in \calM$.  
Also assume $\calM$ has sectional curvatures in the interval $[\Klo, 0]$ with $\Klo < 0$.
Let $L > 0$ and $r \geq \max\big\{\tilde{r}, \frac{8}{\sqrt{-\Klo}}, \frac{4(d+2)}{\tilde{c}}\big\}$.
Define $\hat{\epsilon} = \frac{1}{2^{10} r \sqrt{-\Klo}}$ and $\mu = 64 \hat\epsilon L = \frac{L}{2^{4} r \sqrt{-\Klo}}$.
Let $\calA$ be any deterministic algorithm.


There is a $C^\infty$ function $f$ with minimizer $x^* \in B(\xorigin, \frac{3}{4} r)$ such that running $\calA$ on $f$ yields iterates $x_0, x_1, x_2, \ldots$ satisfying $f(x_k) - f(x^*) \geq 2 \hat\epsilon  L r^2$ for all $k = 0, 1, \ldots, T-1$, where
\begin{align} \label{thisusefulguyonT}
{T} = 
\Bigg\lfloor \frac{\tilde{c} (d+2)^{-1} r}{\log\big(2\cdot 10^6 \cdot \tilde{c} (d+2)^{-1} r (r \sqrt{-\Klo})^2\big)}\Bigg\rfloor.
\end{align}

Moreover, $f$ is $\mu$-strongly g-convex in $\calM$, and $\mu(12 \mathscr{R} + \frac{3}{\sqrt{-\Klo}})$-Lipschitz and $\mu (12 \mathscr{R} \sqrt{-\Klo} + 9)$-smooth in the ball $B(\xorigin, \mathscr{R})$, where $\mathscr{R} = 2^9 r \log(r \sqrt{-\Klo})^2$.  
Outside of the ball $B(\xorigin, \mathscr{R})$, $f(x) = 3\mu \dist(x, \xorigin)^2$ for all $x \not\in B(\xorigin, \mathscr{R})$.
\end{lemma}

\begin{proof}
Due to the assumed lower bound on $r$, we can apply Theorem 24 of~\citep{bumpfctpaper} to $\calA$.
After scaling by $6 \mu$,\footnote{Scaling by $6 \mu$ is legitimate by a simple reduction argument: given an algorithm $\mathcal{A}$, we can create an algorithm $\hat{\mathcal{A}}$ which internally runs $\mathcal{A}$ to query an oracle $\calO_{\hat f}$ and multiplies the outputs of $\calO_{\hat f}$ by $6 \mu$ before forwarding them to $\mathcal{A}$.  We call upon Theorem 24 of~\citep{bumpfctpaper} to claim there exists a hard function $\hat f$ for $\hat{\mathcal{A}}$.  Since $\mathcal{A}$ is effectively interacting with $\calO_{6 \mu \hat f}$, we find $f = 6 \mu \hat f$ is a hard function for $\calA$.} that theorem provides a function $f$ which is $\mu$-strongly g-convex in $\calM$, and $\mu(12 \mathscr{R} + \frac{3}{\sqrt{-\Klo}})$-Lipschitz and $\mu (12 \mathscr{R} \sqrt{-\Klo} + 9)$-smooth in $B(\xorigin, \mathscr{R})$ with minimizer $x^* \in B(\xorigin, \frac{3}{4} r)$ such that 
$$\dist(x_k, x^*) \geq \frac{r}{4} \quad \quad \forall k \leq T-1,$$
where $x_0, x_1, \ldots$ denote the queries produced by running $\calA$ on $f$.
Using $\dist(x_k, x^*) \geq \frac{r}{4}$,
$$2 \hat\epsilon \cdot  L r^2 = \frac{1}{32} \mu r^2 \leq \frac{\mu}{2} \dist(x_k, x^*)^2 \leq f(x_k) - f(x^*), \quad \quad \forall k \leq T-1.$$
For the last inequality we used the $\mu$-strong g-convexity of $f$.
\end{proof}

Applying the exact same proof given in Section~\ref{nonstronglyconvexcaseextension} for Theorem~\ref{nicetheorem} to Lemma~\ref{lemmaaboutintermediaryfunctionclass}, we conclude the following theorem (which is also a generalization of Theorem~\ref{nicetheorem}).

\begin{theorem} \label{thmonemore}
Let $\calM$ be a Hadamard manifold of dimension $d \geq 2$ which satisfies the ball-packing property with constants $\tilde{r}, \tilde{c}$ and point $\xorigin \in \calM$.  
Also assume $\calM$ has sectional curvatures in the interval $[\Klo, 0]$ with $\Klo < 0$.
Let $\tilde{L} > 0$ and $r \geq \max\big\{\tilde{r}, \frac{8}{\sqrt{-\Klo}}, \frac{4(d+2)}{\tilde{c}}\big\}$.
Define
$$\hat{\epsilon} = \frac{1}{2^{10} r \sqrt{-\Klo}}, \quad \epsilon = \frac{1}{2^{18} \zeta_{r \sqrt{-\Klo}} \log^2(\zeta_{r \sqrt{-\Klo}})}, \quad \epsilon' = \frac{1}{2^{18} \log^2(\zeta_{r \sqrt{-\Klo}})},$$
and $L = \frac{\tilde L}{2^{10} \log^2(r \sqrt{-\Klo})}, \tilde M = \frac{\tilde{L}}{2 \sqrt{-\Klo}}$.
Let $\calA$ be any deterministic algorithm.

There is a $C^\infty$ function $\tilde{f} \in \mathcal{F}_{r, \tilde M, 0, \tilde{L}}$ with unique minimizer $x^*$ such that running $\calA$ on $\tilde{f}$ yields iterates $x_0, x_1, x_2, \ldots$ satisfying 
$$\tilde{f}(x_k) - \tilde{f}(x^*) \geq 2 \hat \epsilon L r^2 \geq \epsilon' \cdot \tilde M r \geq \epsilon \cdot \frac{1}{2} \tilde{L} r^2$$
for all $k = 0, 1, \ldots, T-1$, where $T$ is given by equation~\eqref{thisusefulguyonT}.
\end{theorem}

Theorem~\ref{theoremtheorem} now follows directly from Theorem~\ref{thmonemore} by using Lemma 7 of~\citep{bumpfctpaper} for the values of $\tilde{r}$ and $\tilde{c}$ in the ball-packing property.

\subsection{Defining the function $u_{\mathscr{R}} \colon \reals \to \reals$} \label{nonstronglyconvexboundsapphighlevel}
Define the $C^\infty$ function $u_{\mathscr{R}} \colon \reals \rightarrow \reals$ by $u_{\mathscr{R}}(\mathscr{D}) = 1$ for all $\mathscr{D} \leq \frac{1}{2} \mathscr{R}^2$ and $u_{\mathscr{R}}(\mathscr{D}) = 1 - e^{-4 / \sqrt{2 \mathscr{D} / \mathscr{R}^2-1}}$ for all $\mathscr{D} > \frac{1}{2} \mathscr{R}^2$.
That $u_{\mathscr{R}}$ is $C^\infty$ can be verified using the same method used in the proof of Lemma 2.20 from~\citep[Ch.~2]{lee2012smoothmanifolds}.

\subsection{Verifying $\tilde{f}_{\mathscr{R}}$ is Lipschitz, smooth and strictly g-convex in $\calM$} \label{nonstronglyconvexboundsapp}
Let $\calM$ be Hadamard manifold with sectional curvatures in $[\Klo, 0]$.
To complete the proof of Theorem~\ref{nicetheorem} from Section~\ref{nonstronglyconvexcaseextension} (or its generalization Theorem~\ref{thmonemore}), we show that given $f$ with $f(x) = 3\mu \dist(x, \xorigin)^2$ for $x \not \in B(\xorigin, \mathscr{R})$, the function $\tilde{f}_{\mathscr{R}}$ defined by~\eqref{eqdefineftildeR} is strictly g-convex, $12 \mu \mathscr{R}$-Lipschitz, and {$24 \mu \mathscr{R} \sqrt{-\Klo}$}-smooth outside of the ball $B(\xorigin, \mathscr{R})$.

Let $\gamma(t)$ be a geodesic with $\gamma(0) = x, \gamma'(0) = v$ and $\norm{v} = 1$.  
For the moment, define $\mathscr{D}(x) = \frac{1}{2} \dist(x, \xorigin)^2$.
For the gradient, we have
\begin{align*}
\inner{\nabla \tilde{f}_{\mathscr{R}}(x)}{v} =& \frac{d}{dt}[\tilde{f}(\gamma(t))]_{t=0} = \frac{d}{dt}[u_{\mathscr{R}}(\mathscr{D}(\gamma(t))) f(\gamma(t))]_{t=0} \\
=&u_{\mathscr{R}}(\mathscr{D}(x)) \frac{d}{dt}[f(\gamma(t))]_{t=0} + \frac{d}{dt}[u_{\mathscr{R}}(\mathscr{D}(\gamma(t)))]_{t=0} f(x) \\
=& u_{\mathscr{R}}(\mathscr{D}(x)) \inner{\nabla f(x)}{v} + f(x) u_{\mathscr{R}}'(\mathscr{D}(x)) \frac{d}{dt}[\mathscr{D}(\gamma(t))]_{t=0} \\
=&u_{\mathscr{R}}(\mathscr{D}(x)) \inner{\nabla f(x)}{v} - f(x) u_{\mathscr{R}}'(\mathscr{D}(x)) \inner{\log_x(\xorigin)}{v}
\end{align*}
which gives us the formula
\begin{equation} \label{eqfortildefRgrad}
\begin{split}
\nabla \tilde{f}_{\mathscr{R}}(x) = 
\begin{cases}
	\nabla f(x) & \text{if } \dist(x, \xorigin) \leq \mathscr{R}; \\
	u_{\mathscr{R}}\big(\mathscr{D}(x)\big) \nabla f(x) \\
	\quad - f(x) u_{\mathscr{R}}'\big(\mathscr{D}(x)\big) \exp_{x}^{-1}(\xorigin) & \text{otherwise}.
\end{cases}
\end{split}
\end{equation}
By Lemma~\ref{technicallemmaaboutfctu4}, we also see that $\|{\nabla \tilde{f}_{\mathscr{R}}(x)}\| \leq 12 \mu \mathscr{R}$ if $\mathscr{D}(x) > \frac{1}{2} \mathscr{R}^2$.

For the Hessian we have:
\begin{align*}
&\inner{v}{\nabla^2 \tilde{f}_{\mathscr{R}}(x) v} = \frac{d^2}{dt^2}[\tilde{f}_{\mathscr{R}}(\gamma(t))]_{t=0} = \frac{d^2}{dt^2}[u_{\mathscr{R}}(\mathscr{D}(\gamma(t))) f(\gamma(t))]_{t=0} \\
=& \frac{d^2}{dt^2}[u_{\mathscr{R}}(\mathscr{D}(\gamma(t)))]_{t=0} f(x) + u_{\mathscr{R}}(\mathscr{D}(x)) \frac{d^2}{dt^2}[f(\gamma(t))]_{t=0} \\
   &+ 2 \frac{d}{dt}[u_{\mathscr{R}}(\mathscr{D}(\gamma(t)))]_{t=0} \frac{d}{dt}[f(\gamma(t))]_{t=0} \\
=& f(x) u_{\mathscr{R}}''(\mathscr{D}(x)) \Big(\frac{d}{dt}[\mathscr{D}(\gamma(t))]_{t=0}\Big)^2 + f(x) u_{\mathscr{R}}'(\mathscr{D}(x)) \frac{d^2}{dt^2}[\mathscr{D}(\gamma(t)))]_{t=0} \\
&+ u_{\mathscr{R}}(\mathscr{D}(x)) \frac{d^2}{dt^2}[f(\gamma(t))]_{t=0} + 2 \frac{d}{dt}[u_{\mathscr{R}}(\mathscr{D}(\gamma(t)))]_{t=0} \frac{d}{dt}[f(\gamma(t))]_{t=0}.
\end{align*}
Therefore,
\begin{align*}
\inner{v}{\nabla^2 \tilde{f}_{\mathscr{R}}(x) v}
=& f(x) u_{\mathscr{R}}''(\mathscr{D}(x)) \inner{\log_x(\xorigin)}{v}^2 + f(x) u_{\mathscr{R}}'(\mathscr{D}(x)) \inner{v}{\nabla^2 \mathscr{D}(x) v} \\
&+ u_{\mathscr{R}}(\mathscr{D}(x)) \frac{d^2}{dt^2}[f(\gamma(t))]_{t=0} - 2 \inner{\log_x(\xorigin)}{v} \inner{\nabla f(x)}{v} \\
=& 6 \mu\Bigg( [u_{\mathscr{R}}(\mathscr{D}(x)) + \mathscr{D}(x) u_{\mathscr{R}}'(\mathscr{D}(x))]\inner{v}{\nabla^2 \mathscr{D}(x) v} \\
&+ [2 u_{\mathscr{R}}'(\mathscr{D}(x)) + \mathscr{D}(x) u_{\mathscr{R}}''(\mathscr{D}(x)) ]\inner{\log_x(\xorigin)}{v}^2 \Bigg).
\end{align*}

Lemmas~\ref{technicallemmaaboutfctu1} and~\ref{technicallemmaaboutfctu2} from Appendix~\ref{technicalfactsaboutthefunctionuapp} shows that $u_{\mathscr{R}}(\mathscr{D}) + \mathscr{D} u_{\mathscr{R}}'(\mathscr{D}) \geq 0$ and $2 u_{\mathscr{R}}'(\mathscr{D}) + \mathscr{D} u_{\mathscr{R}}''(\mathscr{D}) \leq 0$ for all $\mathscr{D} > \frac{1}{2} \mathscr{R}^2$.
Therefore, $\inner{v}{\nabla^2 \tilde{f}_{\mathscr{R}}(x) v}$ is at least
$$6 \mu\Bigg( [u_{\mathscr{R}}(\mathscr{D}(x)) + \mathscr{D}(x) u_{\mathscr{R}}'(\mathscr{D}(x))] + [2 u_{\mathscr{R}}'(\mathscr{D}(x)) + \mathscr{D}(x) u_{\mathscr{R}}''(\mathscr{D}(x)) ]2 \mathscr{D} \Bigg).$$
Lemma~\ref{technicallemmaaboutfctu3} shows that this is strictly greater than zero, verifying strict g-convexity.

Likewise, Lemma~\ref{technicallemmaaboutfctu4} along with Lemma~\ref{lemmaboundhess} imply $\inner{v}{\nabla^2 \tilde{f}_{\mathscr{R}}(x) v}$ is at most
$$6 \mu\Bigg( [u_{\mathscr{R}}(\mathscr{D}(x)) + \mathscr{D}(x) u_{\mathscr{R}}'(\mathscr{D}(x))]2 \sqrt{-\Klo}\sqrt{2 \mathscr{D}}\Bigg) \leq {24 \mu \sqrt{-\Klo} \mathscr{R}.}$$

\subsection{Technical facts about the function $u_{\mathscr{R}} \colon \reals \rightarrow \reals$} \label{technicalfactsaboutthefunctionuapp}
In the following lemmas, we use the change of variables 
$$\tau = 1 / \sqrt{\frac{2 \mathscr{D}}{\mathscr{R}^2}-1} \iff \mathscr{D} = \mathscr{R}^2 \frac{1+\tau^2}{2\tau^2}.$$
for $\mathscr{D} \in (\frac{1}{2} \mathscr{R}^2, \infty)$ and $\tau \in (0, \infty).$

\begin{lemma} \label{technicallemmaaboutfctu1}
$u_{\mathscr{R}}(\mathscr{D}) + \mathscr{D} u_{\mathscr{R}}'(\mathscr{D}) \geq 0, \quad \forall \mathscr{D} \in (\frac{1}{2} \mathscr{R}^2, \infty)$.
\end{lemma}
\begin{proof}
Computing we find:
$$u_{\mathscr{R}}(\mathscr{D}) + \mathscr{D} u_{\mathscr{R}}'(\mathscr{D}) = 1-e^{-4\tau}(1+2\tau+2\tau^3).$$
On the other hand, we know $1 + 4 \tau + 8 \tau^2 + \frac{32}{3} \tau^3 \geq 1+2\tau+2\tau^3$ for all $\tau \geq 0$ (which can be verified with Mathematica), so
$$e^{4\tau} = \sum_{j = 0}^{\infty} \frac{1}{j!} (4\tau)^j \geq \sum_{j = 0}^{3} \frac{1}{j!} (4\tau)^j = 1 + 4 \tau + 8 \tau^2 + \frac{32}{3} \tau^3 \geq 1+2\tau+2\tau^3, \quad \forall \tau \geq 0,$$
which allows us to conclude.
\end{proof}

\begin{lemma} \label{technicallemmaaboutfctu2}
$2 u_{\mathscr{R}}'(\mathscr{D}) + \mathscr{D} u_{\mathscr{R}}''(\mathscr{D}) \leq 0, \quad \forall \mathscr{D} \in (\frac{1}{2} \mathscr{R}^2, \infty)$.
\end{lemma}
\begin{proof}
Computing we find:
$$2 u_{\mathscr{R}}'(\mathscr{D}) + \mathscr{D} u_{\mathscr{R}}''(\mathscr{D}) = 2 \mathscr{R}^{-2} e^{-4\tau} \tau^3 (-1-4\tau + 3\tau^2-4\tau^3)$$
and one can check that $-1-4\tau + 3\tau^2-4\tau^3 \leq 0$ for all $\tau \geq 0$.
\end{proof}

\begin{lemma} \label{technicallemmaaboutfctu3}
$[u_{\mathscr{R}}(\mathscr{D}) + \mathscr{D} u_{\mathscr{R}}'(\mathscr{D})] + [2 u_{\mathscr{R}}'(\mathscr{D}) + \mathscr{D} u_{\mathscr{R}}''(\mathscr{D}) ]2 \mathscr{D} > 0, \quad \forall \mathscr{D} \in (\frac{1}{2} \mathscr{R}^2, \infty)$.
\end{lemma}
\begin{proof}
Computing we find:
$$[u_{\mathscr{R}}(\mathscr{D}) + \mathscr{D} u_{\mathscr{R}}'(\mathscr{D})] + [2 u_{\mathscr{R}}'(\mathscr{D}) + \mathscr{D} u_{\mathscr{R}}''(\mathscr{D}) ]2 \mathscr{D} = 1 - e^{-4\tau}(1+4\tau + 8\tau^2 - 2 \tau^3 + 16 \tau^4 - 6 \tau^5 + 8 \tau^6).$$
On the other hand, we know $\sum_{j = 0}^{7} \frac{1}{j!} (4\tau)^j > 1+4\tau + 8\tau^2 - 2 \tau^3 + 16 \tau^4 - 6 \tau^5 + 8 \tau^6$ for all $\tau > 0$ (which can be verified with Mathematica), which allows us to conclude.
\end{proof}

\begin{lemma} \label{technicallemmaaboutfctu4}
$(u_{\mathscr{R}}(\mathscr{D}) + \mathscr{D} u_{\mathscr{R}}'(\mathscr{D})) 2 \sqrt{2 \mathscr{D}} \leq 4 \mathscr{R}, \quad \forall \mathscr{D} \in (\frac{1}{2} \mathscr{R}^2, \infty)$.
\end{lemma}
\begin{proof}
This follows from a calculation similar to the proof of Lemma~\ref{technicallemmaaboutfctu1}.
\end{proof}

\subsection{Technical fact from Section~\ref{nonstronglyconvexcaseextension}: $x_k \not \in B(\xorigin, \mathscr{R})$ implies $\tilde{f}_{\mathscr{R}}(x_k) - \tilde{f}_{\mathscr{R}}(x^*) \geq 2 \hat \epsilon L r^2$} \label{technicalfacttofinishproofoftheorem1p5}
Recall Lemma~\ref{lemmausefulguysec3} provides a $\mu = 64 \hat{\epsilon} L$-strongly g-convex function $f$, with minimizer $x^*$, satisfying $f(x_k) - f(x^*) \geq 2 \hat\epsilon L r^2$ for all $k \leq T-1$.

If $x_k \not \in B(\xorigin, \mathscr{R})$, then consider the geodesic segment $\gamma_{x^* \rightarrow x_k} \colon [0,1] \rightarrow \calM$ given by $\gamma_{x^* \rightarrow x_k}(t) = \exp_{x^*}(t \log_{x^*}(x_k))$.
Continuity of $t \mapsto \dist(\gamma_{x^* \rightarrow x_k}(t), \xorigin)$ implies there exists a $t \in (0,1)$ so that $\mathscr{R} = \dist(\gamma_{x^* \rightarrow x_k}(t), \xorigin)$.  Let $y = \gamma(t)$.
Geodesic convexity of $\tilde{f}_{\mathscr{R}}$ implies
$$\tilde{f}_{\mathscr{R}}(y) \leq (1-t)\tilde{f}_{\mathscr{R}}(x^*) + t \tilde{f}_{\mathscr{R}}(x_k) \leq (1-t)\tilde{f}_{\mathscr{R}}(y) + t \tilde{f}_{\mathscr{R}}(x_k),$$
which implies $\tilde{f}_{\mathscr{R}}(y) \leq \tilde{f}_{\mathscr{R}}(x_k)$.
On the other hand, we know $\tilde{f}_{\mathscr{R}} = f$ in $B(\xorigin, \mathscr{R})$. Therefore $\tilde{f}_{\mathscr{R}}$ is $\mu$-strongly g-convex in $B(\xorigin, \mathscr{R})$, so
$$\tilde{f}_{\mathscr{R}}(x_k) - \tilde{f}_{\mathscr{R}}(x^*) \geq \tilde{f}_{\mathscr{R}}(y) - \tilde{f}_{\mathscr{R}}(x^*) \geq \frac{\mu}{2} \dist(y, x^*)^2 \geq \frac{\mu}{2} (\mathscr{R} - r)^2 \geq 32 \hat{\epsilon} L (2^{11} r)^2 \geq 2 \hat \epsilon L r^2.$$

\section{Properties of the Moreau envelope: Proof of Lemma~\ref{moreauenvelope}} \label{moreauenvelopeapp}
\citet[Cor. 4.5]{azagra1} shows that $f_\lambda$ is g-convex and $C^1$.
The function $y \mapsto f(y) + \frac{1}{2 \lambda} \dist(x,y)^2$ is strongly g-convex and so has a unique minimizer which we denote
$$y_{\lambda}(x) = \arg\min_{y \in \calM} \{f(y) + \frac{1}{2 \lambda} \dist(x,y)^2\}.$$
First-order optimality conditions imply $\nabla f(y_\lambda(x)) = \frac{1}{\lambda} \log_{y_\lambda(x)}(x)$.
Since, $f$ is $1$-Lipschitz, this implies $\dist(x, y_\lambda(x)) = \|\log_{y_\lambda(x)}(x)\| \leq \lambda$.  
This proves $f_\lambda(x) = \min_{y \in B(x, \lambda)} \{f(y) + \frac{1}{2 \lambda} \dist(x,y)^2\}$. 
Next, using the g-convexity and $1$-Lipschitzness of $f$, we have
$$f(x) \geq f_\lambda(x) \geq f(y_\lambda(x)) \geq f(x) + \langle \nabla f(x), \log_x(y_\lambda(x))\rangle \geq f(x) - \dist(x, y_\lambda(x)) \geq f(x) - \lambda.$$

To compute the Lipschitz constant of $\nabla f_\lambda$, we follow (but slightly simplify) the proof of \citep[Prop. 7.1]{azagra2}.
Let $\delta > 0$.
By Theorem 1.5 of \citep{azagra2} (which provides a useful characterization of Lipschitz gradient), it suffices to show that for each $x_0 \in \calM$ the function $x \mapsto f_\lambda(x) - \frac{\zeta_{\lambda + 2 \delta}}{2 \lambda} \dist(x, x_0)^2$ is g-concave when restricted to $B(x_0, \delta)$.
For $x \in B(x_0, \delta)$,
\begin{align*}
f_\lambda(x) - \frac{\zeta_{\lambda + 2 \delta}}{2 \lambda} \dist(x, x_0)^2 
&= \min_{y \in B(x, \lambda)} \{f(y) + \frac{1}{2 \lambda} \dist(x,y)^2\} - \frac{\zeta_{\lambda + 2 \delta}}{2 \lambda} \dist(x, x_0)^2
\\&= \min_{y \in B(x_0, \delta + \lambda)} \{f(y) + \frac{1}{2 \lambda} \dist(x,y)^2 - \frac{\zeta_{\lambda + 2 \delta}}{2 \lambda} \dist(x, x_0)^2\}.
\end{align*}
As the minimum of g-concave functions is g-concave, it suffices to show that for each $y \in B(x_0, \delta + \lambda)$, the function $\tilde{d}(x) = \frac{1}{2} \dist(x,y)^2 - \frac{\zeta_{\lambda + 2 \delta}}{2} \dist(x, x_0)^2$ is g-concave when restricted to $B(x_0, \delta)$.  This is easy to see by looking at the maximum eigenvalue of its Hessian.  For $x \in B(x_0, \delta)$,
$$\lambda_{max}(\nabla^2 \tilde \dist(x)) \leq \zeta_{\dist(x, y)} - \zeta_{\lambda + 2 \delta} \leq 0$$
using that $\dist(x,y) \leq \delta + \delta + \lambda = 2 \delta + \lambda$.  This concludes the proof of Lipschitz gradient.

Finally, to see that $f_\lambda$ is $1$-Lipschitz, note that $\nabla f_\lambda(x) = - \frac{1}{\lambda} \log_x(y_\lambda(x))$ for all $x \in \calM$ \citep[Prop. 3.7]{azagra1}, so $\|\nabla f_\lambda(x)\| = \frac{1}{\lambda} \dist(x, y_{\lambda}(x)) \leq 1$.

\section{The relation between $\zeta$, $\epsilon$ and $\kappa$: proofs from Section~\ref{lrvslrsquared}}\label{proofsfromgeometryinfluencescostfct}
The following proposition states if $f$ is gradient-Lipschitz then it is Lipschitz.  
The proof is due to~\citep{petrunin2023pigtikal}.  We just fill in the details.
\begin{proposition} \label{lrvslrsquaredprop}
Let $p \in \calM = \mathbb{H}^d$, and assume $f \colon \calM \to \reals$ is differentiable and $L$-smooth in $B(p, 2)$.  Then $\|\nabla f(p)\| \leq \frac{\pi \sqrt{5} L}{2} \leq 4 L$. 
\end{proposition}
\begin{remark}
It is necessary that we assume $f$ is $L$-smooth in a region which is not too small, because we can construct functions with small Hessian and arbitrarily large gradient in a sufficiently small region by pushing forward functions defined in a tangent space via the exponential map, as in Appendix~\ref{carryovereuc}.
\end{remark}

\begin{proof}[Proof of Proposition~\ref{lrvslrsquaredprop}]
Take a 2-dimensional totally geodesic submanifold $S$ which passes through $p$ and is tangent to $\nabla f(p)$, i.e., $\nabla f(p) \in \T_p S$.  
Let $\hat{f}$ be the restriction of $f$ to $S$.  
Note that $\hat{f}$ is $L$-smooth in $S \cap B(p, 2)$.\footnote{This follows from: if $\calP_x$ denotes orthogonal projection from $\T_x\calM$ to $\T_x S$, one can check that $\calP_y P_{x \to y} = P_{x \to y} \calP_x$.}
Consider a circle $C \subseteq S \cap B(p, 2)$ of radius $R = \arccosh(3/2) < 1$
containing $p$.
Let $c \colon [0, T] \to C$, $T = 2 \pi \sinh(R)$, be an arc-length parameterization of $C$ with $c(0) = c(T) = p$, and let $P_{t \to s}^c$ denote parallel transport (in $S$) along $c$ from $c(t)$ to $c(s)$.  

For the moment, assume that $f$ is twice differentiable.
Observe that  
$$\frac{d}{dt}[P_{t \to 0}^c \nabla \hat{f}(c(t))] = P_{t \to 0}^c D_t \nabla \hat{f}(c(t)) = P_{t \to 0}^c \nabla^2 \hat{f}(c(t)) c'(t),$$
where $D_t$ is the covariant derivative along $\gamma$.  
Therefore, 
\begin{align}\label{thisineqisnice}
\|P_{T \to 0}^c \nabla \hat{f}(p) - \nabla \hat{f}(p)\| = \|\int_{0}^T P_{t \to 0}^c \nabla^2 \hat{f}(c(t)) c'(t) dt\| \leq L T.
\end{align}
  
On the other hand,  $P_{T \to 0}^c \nabla \hat{f}(p)$ is a rotation of $\nabla \hat{f}(p)$ by some angle $\theta$, and the Gauss-Bonnet Formula~\citep[Ch. 9]{lee2018riemannian} states that $\theta = \int K d A = 2 \pi - \int_C k_g = -2 \pi (\cosh(R) - 1)$, where $\int K d A$ is the integral of the curvature $K = -1$ over the interior of $C$ (a disk), and $k_g = \coth(R)$ is the geodesic curvature of $C$.  
Since $R = \arccosh(3/2)$, we know $\theta = -\pi$, and $\|P_{T \to 0}^c \nabla \hat{f}(p) - \nabla \hat{f}(p)\| = 2 \|\nabla \hat{f}(p)\|$.  We conclude that $\|\nabla f(p)\| = \|\nabla \hat{f}(p)\| \leq L T / 2 = L \pi \sinh(R) = L \pi \sqrt{5}/2 \leq 4 L$.

If $f$ is not twice differentiable, we can still conclude the result in a similar way. 
Place $N$ equally spaced points around $C$, $p_n = c(t_n), t_n = \frac{n T}{N}$ for $n = 0, \ldots, N$.
Let $\alpha = \frac{\pi}{N}$.
Let $\gamma_n$ denote the geodesic segment between $p_n, p_{n+1}$.
Let $\Gamma_N$ denote the geodesic polygon formed by the segments $\gamma_n$.
Let $P_{n \to m}^{\Gamma}, m \leq n,$ denote parallel transport (in $S$) from $p_n$ to $p_m$ along the geodesic segments $\gamma_n, \gamma_{n-1}, \ldots, \gamma_m$.  Note that $P_{n \to 0}^{\Gamma} = P_{n-1 \to 0}^{\Gamma} P_{n \to n-1}^{\Gamma}.$
The $L$-smoothness of $\hat{f}$ implies, for $n = 0, \ldots, N$,
\begin{align*}
&\|P_{n \to 0}^{\Gamma} \nabla \hat{f}(p_n) - \nabla \hat{f}(p)\| \\
&= \|P_{n-1 \to 0}^{\Gamma} P_{n \to n-1}^{\Gamma} \nabla \hat{f}(p_n) - P_{n-1 \to 0}^{\Gamma} \nabla \hat{f}(p_{n-1}) + P_{n-1 \to 0}^{\Gamma} \nabla \hat{f}(p_{n-1}) - \nabla \hat{f}(p)\| \\
&\leq \|P_{n \to n-1}^{\Gamma} \nabla \hat{f}(p_n) - \nabla \hat{f}(p_{n-1})\| 
+ \|P_{n-1 \to 0}^{\Gamma} \nabla \hat{f}(p_{n-1}) - \nabla \hat{f}(p)\| \\
&\leq L \dist(p_n, p_{n-1}) 
+ \|P_{n-1 \to 0}^{\Gamma} \nabla \hat{f}(p_{n-1}) - \nabla \hat{f}(p)\|.
\end{align*}
Therefore, $\|P_{n \to 0}^{\Gamma} \nabla \hat{f}(p_n) - \nabla \hat{f}(p)\| \leq \sum_{n=0}^{n-1} L \dist(p_n, p_{n-1}) \leq L T$, and so for all $N$,
$$\|P_{N \to 0}^{\Gamma} \nabla \hat{f}(p) - \nabla \hat{f}(p)\| \leq L T.$$

On the other hand,  $P_{N \to 0}^\Gamma \nabla \hat{f}(p)$ is a rotation of $\nabla \hat{f}(p)$ by some angle $\theta_N$, and the Gauss-Bonnet Formula~\citep[Ch. 9]{lee2018riemannian} states that $\theta_N = \int K d A$, where $\int K d A$ is the integral of the curvature $K = -1$ over the interior of the polygon $\Gamma$.
The area of that polygon equals $2N = \frac{2}{\pi \alpha}$ times the area of a right hyperbolic triangle with angle $\alpha$ and hypotenuse $R$.
Let $\beta \in (0, \frac{\pi}{2})$ be the other angle in this triangle which is not a right angle.
Using hyperbolic trigonometry, we find
$$\beta = \arcsin\bigg(\frac{1}{\sinh(R)} \sinh\Big(\arctanh(\tanh(R) \cos(\alpha))\Big)\bigg).$$
The area of this triangle equals $\pi - \frac{\pi}{2} - \alpha - \beta$, and so we find that
$$|-\pi - \theta_N| = |-\pi - K \cdot \frac{2\pi}{\alpha} \cdot (\pi - \frac{\pi}{2} - \alpha - \beta)| \leq \frac{5\pi}{4} \alpha^2, \quad \forall \alpha \in (0, \frac{\pi}{2}].$$
Hence, taking $N$ large enough (i.e., $\alpha$ small enough),
$\|P_{N \to 0}^{\Gamma} \nabla \hat{f}(p) - \nabla \hat{f}(p)\|$ is arbitrarily close to $2\|\nabla \hat{f}(p)\|$.
Yet, $\|P_{N \to 0}^{\Gamma} \nabla \hat{f}(p) - \nabla \hat{f}(p)\| \leq L T$, and we conclude as in the case where $f$ is twice differentiable.
\end{proof}

\begin{proof}[Proof of Proposition~\ref{usefulguy}]
Assume $r > 2$.
By Proposition~\ref{lrvslrsquaredprop}, $\|\nabla f(p)\| \leq \frac{\pi \sqrt{5} L}{2}$ for all $p \in B(\xorigin, r - 2)$.
If $\dist(x^*, \xorigin) \leq r-2$, then we know that $f(\xorigin) - f^* \leq \frac{\pi \sqrt{5} L}{2} (r-2)$.
Otherwise, define $x' = \exp_{\xorigin}\big( \frac{r-2}{\dist(x^*, \xorigin)} \log_{\xorigin}(x^*)\big)$.
Then $f(\xorigin) - f(x') \leq \frac{\pi \sqrt{5} L}{2} (r-2)$ and $\|\nabla f(x')\| \leq \frac{\pi \sqrt{5} L}{2}$.
So by $L$-smoothness of $f$,
$$f^* = f(x^*) \geq f(x') + \langle \nabla f(x'), \log_{x'}(x^*) \rangle - \frac{L}{2} \dist(x', x^*)^2 \geq f(x') -  \frac{\pi \sqrt{5} L}{2} \cdot 2 - \frac{L}{2} \cdot 4$$
and so, if $r > 2$,
$$f(\xorigin) - f^* \leq f(\xorigin) - f(x') + f(x') - f^* \leq \frac{\pi \sqrt{5} L}{2}(r-2) + \pi \sqrt{5} L + 2 L \leq \frac{\pi \sqrt{5} L}{2} r + 2 L.$$
We conclude, for all $r > 0$,
$f(\xorigin) - f^* \leq \min \Big\{\frac{\pi \sqrt{5} L}{2} r + 2 L, \frac{1}{2} L r^2\Big\} \leq \frac{8}{\zeta_r} \cdot \frac{1}{2} L r^2$.
\end{proof}
To prove the bound $\kappa \geq \zeta_r$~\citet{martinezrubio2021global} and~\citet{hamilton2021nogo} use that the domain is a ball.  
We can use Proposition~\ref{lrvslrsquaredprop} to prove $\kappa \geq \Omega(\zeta_r)$ if the domain is not too \emph{eccentric}.
\begin{proposition}\label{improvementonkappamorethanr}
Let $f \colon \mathbb{H}^d \to \reals$ be differentiable, $\mu$-strongly g-convex globally, and suppose $\nabla f(x^*) = 0$.  Let $r > 2$, and consider a g-convex subset $D$ containing $x^*$ such that 
\begin{enumerate}
\item there exists $\tilde D \subseteq D$ such that $B(x, 2) \subseteq D$ for all $x \in \tilde D$, and 
\item there is a $\tilde x \in \tilde D$ so that $\dist(\tilde x, x^*) \geq r - 2$.  
\end{enumerate}
(For example, $D = B(x^*, r)$ satisfies these assumptions.)

Let $L$ be the Lipschitz constant of $\nabla f$ in $D$.  Then $\kappa = \frac{L}{\mu} \geq \frac{1}{10} \zeta_r$.  
\end{proposition}
\begin{proof}[Proof of Proposition~\ref{improvementonkappamorethanr}]
We know that $\|\nabla f(x)\| \leq \frac{\pi\sqrt{5}}{2} L$ for all $x \in \tilde D$, and so 
$$\frac{\pi\sqrt{5}}{2} L \dist(x, x^*) \geq f(x) - f^* \geq \frac{\mu}{2} \dist(x, x^*)^2$$ for all $x \in \tilde D$.  Using $\dist(\tilde x, x^*) \geq r - 2$, $\frac{L}{\mu} \geq \max\{1,\frac{r - 2}{\pi \sqrt{5}}\} \geq \frac{1}{10} \zeta_r$.
\end{proof}

\begin{remark}
From the results in this section, the most natural scalings for smooth g-convex optimization on hyperbolic spaces are arguably \emph{not} the ones given in~\ref{smoothproblem} and~\ref{smoothstronglyproblem}. 
Instead, for~\ref{smoothproblem} one should ask for $x$ such that $f(x) - f^* \leq \epsilon \cdot \frac{1}{2\zeta_r} L r^2$, and for~\ref{smoothstronglyproblem} one should define a new condition number $\kappa' = \frac{L}{\mu \zeta_r}$.
We do not do this here in order to maintain consistency with previous literature.
\end{remark}

\begin{proposition} \label{blahblah}
If $f$ is a $C^3$ g-convex function on a hyperbolic space whose Hessian vanishes at a point $x$ then necessarily $\nabla f(x) = 0$.
\end{proposition}
\begin{proof}
Let $\nabla^3 f(x) \colon \T_x \calM \times \T_x \calM \times \T_x \calM \to \reals$ denote the third derivative of $f$ at $x$ (see \citep[Ex.~10.78]{boumal2020intromanifolds} for a definition).
As observed in~\citep[Remark 3.2]{criscitiello2020accelerated}, the Ricci identity applied to $\nabla f$ implies
$$\|\nabla f(x)\| \leq 2 \|\nabla^3 f(x)\|,$$
since hyperbolic space has constant sectional curvature $-1$.

Let $x \in \calM$, $u, v \in \T_x \calM$ and $\gamma(t) = \exp_x(t v)$.  Define the parallel vector field $U$ along $\gamma$ as $U(t) = P_{x\to \gamma(t)} u$.
By the definition of $\nabla^3 f$ (see \citep[Ex.~10.78]{boumal2020intromanifolds}),
$$\nabla^3 f(u, u, v) = \frac{d}{dt}\Big[\langle U(t), \nabla^2 f(\gamma(t)) U(t) \rangle\Big]_{t=0}.$$
Since $f$ is g-convex, we know that $\langle U(t), \nabla^2 f(\gamma(t)) U(t) \rangle \geq 0$ for all $t$.
Since $\langle U(0), \nabla^2 f(\gamma(0)) U(0) \rangle = 0$, we know $\frac{d}{dt}\Big[\langle U(t), \nabla^2 f(\gamma(t)) U(t) \rangle\Big]_{t=0} = 0$.
We conclude that $\nabla^3 f(u, u, v) = 0$ for all $u, v \in \T_x \calM$.  By symmetry of the first two arguments of $\nabla^3 f$, we conclude $\nabla^3 f(x) = 0$.
Hence, $\|\nabla f(x)\| \leq 0$.
\end{proof}
\begin{remark}
The assumption that $f$ is three times differentiable in Proposition~\ref{blahblah} is necessary.
Indeed, let $y\in\calM, g \in \T_y\calM$ with $\|g\|=1$, and consider the function
$$f(x) = \langle g, \log_y(x) \rangle + \frac{\tau}{6} \dist(x, y)^3.$$
If $\tau$ is large enough (e.g., $\tau = 2$), then one can check, like in the proof of Proposition~\ref{boundedhessaffine}, that this function is globally g-convex.  However, $\nabla^2 f(x) = 0$ and $\nabla f(x) = g$.
\end{remark}

\section{Reduction between strongly and nonstrongly g-convex settings} \label{reductionblah}

\begin{proposition} \label{reductionstronglyconvextoconvex}
%
Fix $\epsilon \in (0,1), L > 0$ and $r > 0$.  
Define 
$$\delta = \epsilon \cdot \frac{1}{2} L r^2, \quad\quad \tilde \delta = \frac{1}{2} r^2, \quad \quad \sigma = \frac{\epsilon  L}{2}.$$
Let $\calA$ be a first-order deterministic algorithm.
There is a first-order deterministic algorithm $\calA'$ such that: for all $f \in \calF_{r, \infty, 0, L}$, if we define
$$\tilde{f}(x) = \frac{1}{\sigma}f(x) + \frac{1}{2} \dist(x, \xorigin)^2,$$ 
then $T_{\delta}(\calA', f) \leq T_{\tilde \delta}(\calA, \tilde f)$ (recall Section ~\ref{problemclasses}).

In particular, if $r' \geq 2 r$ and $\kappa = \frac{L}{\sigma} + \zeta_{r'} = \frac{2}{\epsilon} + \zeta_{r'}$, then $T_{\epsilon, r} \leq T_{\epsilon', r', \kappa}$ where $\epsilon' = \frac{r^2}{\kappa (r')^2}$.
\end{proposition}
\begin{proof}
Let $f \in \mathcal{F}_{r, \infty, 0, L}$ be an $L$-smooth g-convex function with minimizer $x^* \in B(\xorigin, r)$ satisfying $\nabla f(x^*) = 0$.
The function $\tilde{f}$ is $1$-strongly g-convex in $\calM$.
Thus, it has a unique minimizer $\tilde{x}^* \in \calM$ satisfying $\nabla \tilde f(\tilde{x}^*) = 0$ and $\dist(\tilde{x}^*, \xorigin) \leq 2 \dist(x^*, \xorigin) \leq 2 r$ because
\begin{align*}
\sigma (\dist(\tilde{x}^*, \xorigin) - \dist(x^*, \xorigin)) &\leq \sigma \dist(x^*, \tilde{x}^*) \leq \norm{\nabla \tilde{f}(x^*)} \\
&= \norm{- \sigma \exp_{x^*}^{-1}(\xorigin)} = \sigma \dist(x^*, \xorigin).
\end{align*}
We conclude that $\tilde{x}^* \in B(\xorigin, 2 r)$.

Let $\calA'$ be the algorithm which runs $\calA$ on the function $\tilde f$ (and so outputs the same queries as $\calA$).
Let $T = T_{\tilde \delta}(\calA, \tilde f)$ (we can assume $T$ is finite otherwise there is nothing to prove).
By assumption, algorithm $\calA$ uses at most $T$ queries to find a point $x \in \calM$ with
$\sigma(\tilde f(x) - \tilde {f}^*) \leq \sigma \tilde \delta = \sigma \frac{1}{2} r^2 = \frac{\epsilon}{4} L r^2,$
where $\tilde {f}^* = \tilde f(\tilde {x}^*)$.
Let us show that in fact $f(x) - f^* \leq \epsilon \cdot \frac{1}{2} L r^2$.  Since 
$$\sigma \tilde{f}^* \leq \sigma \tilde{f}(x^*) = f^* + \frac{\epsilon L}{4} \dist(x^*, \xorigin)^2 \leq f^* + \frac{\epsilon L}{4} r^2$$
we have
$$f(x) - f^* - \frac{\epsilon L}{4} r^2 \leq f(x) - \sigma \tilde{f}^* \leq \sigma(\tilde{f}(x) - \tilde{f}^*) \leq \frac{\epsilon}{4} L r^2,$$
that is, $f(x) - f^* \leq \epsilon \cdot \frac{1}{2} L r^2 = \delta$.
Hence, $T_{\delta}(\calA', f) \leq T = T_{\tilde \delta}(\calA, \tilde f)$, as claimed.

Given $r' \geq 2 r$, we know that $\tilde{x}^* \in B(\xorigin, r')$ and $\tilde f$ is $\frac{L}{\sigma} + \zeta_{r'} = \kappa$-smooth in $B(\xorigin, r')$ by Lemma~\ref{lemmaboundhess}.  Hence, $\tilde f \in \calF_{r', \infty, 1, \kappa}$, and $T_{\delta}(\calA', f) \leq T_{\tilde \delta}(\calA, \tilde f) \leq \sup_{f' \in \calF_{r', \infty, 1, \kappa}} T_{\tilde \delta}(\calA, f').$
Recall that this holds for all $f \in \calF_{r, \infty, 0, L}$, so $\sup_{f \in \calF_{r, \infty, 0, L}} T_{\delta}(\calA', f) \leq \sup_{f' \in \calF_{r', \infty, 1, \kappa}} T_{\tilde \delta}(\calA, f')$, so 
$$T_{\epsilon, r} = \inf_{\tilde \calA} \sup_{f \in \calF_{r, \infty, 0, L}} T_{\delta}(\tilde \calA, f) \leq \sup_{f' \in \calF_{r', \infty, 1, \kappa}} T_{\tilde \delta}(\calA, f')$$
where the infimum is over all algorithms $\tilde \calA$.
This holds for all algorithms $\calA$, so we conclude
$$T_{\epsilon, r} = \inf_{\tilde \calA} \sup_{f \in \calF_{r, \infty, 0, L}} T_{\delta}(\tilde \calA, f) \leq \inf_{\calA} \sup_{f' \in \calF_{r', \infty, 1, \kappa}} T_{\tilde \delta}(\calA, f') = T_{\epsilon', r', \kappa}$$
using that $\tilde \delta = \frac{1}{2} r^2 = \frac{r^2}{\kappa (r')^2} \cdot \frac{1}{2} \kappa (r')^2 = \epsilon' \cdot \frac{1}{2} \kappa (r')^2$.
\end{proof}
We then have the following proof of Theorem~\ref{thisremark}.
\begin{proof}[Proof of Theorem~\ref{thisremark}]
We have two cases:

Case I: $r' \geq 2$.  Taking $r = 1$ and $\epsilon = \frac{2}{\kappa - \zeta_{r'}}$, Proposition~\ref{reductionstronglyconvextoconvex} and Theorem~\ref{thmsmoothlowerbound} imply $$q_{\kappa, r'} \geq \frac{T_{\epsilon', r', \kappa}}{\log(1/\epsilon')} \geq \frac{T_{\epsilon, r}}{\log(1/\epsilon')} \geq \frac{1}{\log(\kappa (r')^2)\sqrt{8 \zeta_r^2 \epsilon}} = \frac{\sqrt{\kappa - \zeta_{r'}}}{4 \zeta_r \log(\kappa (r')^2)} \geq \frac{\sqrt{\kappa - \zeta_{r'}}}{8 \log(\kappa (r')^2)}.$$

Case II: $r' < 2$.  Taking $r = r'/2$ and $\epsilon = \frac{2}{\kappa - \zeta_{r'}}$, Proposition~\ref{reductionstronglyconvextoconvex} and Theorem~\ref{thmsmoothlowerbound} imply
$$q_{\kappa, r'} \geq \frac{T_{\epsilon', r', \kappa}}{\log(1/\epsilon')} \geq \frac{T_{\epsilon, r}}{\log(1/\epsilon')} \geq \frac{1}{\log(4 \kappa)\sqrt{8 \zeta_r^2 \epsilon}} 
= \frac{\sqrt{\kappa - \zeta_{r'}}}{4 \zeta_r \log(4 \kappa)} \geq \frac{\sqrt{\kappa - \zeta_{r'}}}{8 \log(4 \kappa)}.$$
\end{proof}

\section{Width-bounded separators: Working out dependence on $d$ in Lemma~\ref{lemmafrombak}} \label{applemmafrombak}

\citet[Lemma 9 (ii)]{hyperbolicintersectiongraphs}, shows that $\mathbb{E}[\mathcal{S}] \leq \sum_{j=1}^J \mathbb{P}[T \cap H \neq \emptyset | T \in L_j] |L_j|$ where
\begin{itemize}
\item $L_j = \{T \in \mathcal{T} : \dist(p, T) \in [(j-1) \tau, j \tau)\}$.  In particular, using that $c_0 \tau / 2 \geq 4 \log(d)$ and Lemma~\ref{Lemmavolumes},
$$|L_j| \leq V_d((j+1) \tau) / V_d(c_0 \tau / 2) \leq 4 e^{(d-1) \tau (j+1 - c_0/2)}.$$

\item $J \geq 2$ is such that the number of tiles $T$ which intersect the interior of $B(p, (J-1)\tau)$ is at most $|\mathcal{I}|$.  In particular, this implies that $V_{d}((J-1)\tau) / V_d(\tau / 2) \geq |\mathcal{I}|$.  Using~\ref{Lemmavolumes}, we know that
$$V_{d}((J-1)\tau) / V_d(\tau / 2) \geq e^{(d-1)(J-1)\tau} e^{-\tau(d-1)/2-2}$$
and so $(J-1)\tau \leq \frac{1}{d-1}(\log(|\mathcal{I}|) + \tau(d-1)/2 + 2)$.
\end{itemize}
\citet[Lemma 8]{hyperbolicintersectiongraphs}, also shows that
\begin{align*}
\mathbb{P}[T \cap H \neq \emptyset | T \in L_j] &\leq 4 \frac{\sinh(2\tau)}{\sinh(j \tau)} \cdot \frac{\sigma_{d-2}}{\sigma_{d-1}} 
 \leq \frac{4 \sqrt{d}}{2 \pi} \cdot \frac{\sinh(2\tau)}{\sinh(j \tau)} \leq \frac{8 \sqrt{d}}{2 \pi} e^{-\tau (j-2)}.
\end{align*}
Plugging in these bounds and using $\tau \geq \frac{2}{3(d-1)} \log(\frac{64 e^2}{\sqrt{2\pi}} \sqrt{d})$ for the last inequality below, we find
\begin{align*}
\mathbb{E}[\mathcal{S}] &\leq \frac{32 \sqrt{d}}{\sqrt{2 \pi}} e^{(d-1) \tau (3 - c_0/2)} \sum_{j=1}^J e^{(d-2) \tau (j-2)} \\
&= \frac{32 \sqrt{d}}{\sqrt{2 \pi}} e^{(d-1) \tau (3 - c_0/2)}[-e^{-(d-2)\tau} + \frac{e^{(d-2)\tau (J-1)} - 1}{e^{(d-2)\tau} -1}] \\
&\leq \frac{64 \sqrt{d}}{\sqrt{2 \pi}} e^{(d-1) \tau (3)} e^{(d-2)\tau(J-1)} \\
&\leq \frac{64 \sqrt{d}}{\sqrt{2 \pi}} e^{3 (d-1) \tau} e^{\frac{d-2}{d-1}(\tau(d-1)/2 + 2)} |\mathcal{I}|^{\frac{d-2}{d-1}}\leq \frac{1}{2} e^{5(d-1) \tau} |\mathcal{I}|^{\frac{d-2}{d-1}}.
\end{align*}

Next, we need to workout $c_0$.  \citet[Lemma 5(ii)]{hyperbolicintersectiongraphs}, shows there exists a $(c_0 \tau/2, \tau/2)$-nice tiling if $c_0 \leq \tau^{-1} \log\Big(1 + \frac{2}{d}(\cosh(\tau/2)-1)\Big).$
Using that $\cosh(\tau/2)-1 \geq e^{\tau/2}/4$ if $\tau \geq 3$ and $\frac{1}{2 d}e^{\tau/2} \geq e^{\tau / 4}$ if $\tau \geq 8 \log(d)$, we find that
$$\tau^{-1} \log\Big(1 + \frac{2}{d}(\cosh(\tau/2)-1)\Big) \geq \tau^{-1} \log\Big(\frac{1}{2 d} e^{\tau/2}\Big) \geq \frac{1}{4}.$$
Hence, we can take $c_0 = 1/4$.

\end{document}

%% file: preamble.tex
\usepackage{mathtools}
\usepackage{mathrsfs} 
\makeatletter\def\@seccntformat#1{\protect\makebox[0pt][r]{\csname the#1\endcsname\hspace{12pt}}}\makeatother
\definecolor{mycolor1}{rgb}{0.105882,0.619608,0.466667}
\definecolor{mycolor2}{rgb}{0.85098,0.372549,0.00784314}
\definecolor{mycolor3}{rgb}{0.458824,0.439216,0.701961}
\definecolor{mycolor4}{rgb}{0.905882,0.160784,0.541176}
\definecolor{mycolor5}{rgb}{0.4,0.65098,0.117647}
\definecolor{mycolor6}{rgb}{0.65098,0.462745,0.113725}
\definecolor{mycolor7}{rgb}{0.901961,0.670588,0.00784314}
\definecolor{mycolor8}{rgb}{0.4,0.4,0.4}
\definecolor{mycolor9}{rgb}{0.301961,0,0.294118}
\definecolor{mycolor10}{rgb}{0.0313725,0.25098,0.505882}

\DeclareMathOperator{\arccosh}{arccosh}

\DeclareMathOperator{\arctanh}{arctanh}

\usepackage{bigfoot}

\newcommand{\cutchunk}[1]{}

\newcommand{\removesafe}[1]{}

\newcommand{\calA}{\mathcal{A}}

\newcommand{\calP}{\mathcal{P}}

\newcommand{\calF}{\mathcal{F}}

\newcommand{\calH}{\mathcal{H}}

\newcommand{\calM}{\mathcal{M}}

\newcommand{\calO}{\mathcal{O}}

\newcommand{\diag}{\mathrm{diag}}
\newcommand{\Vol}{\mathrm{Vol}}

\newcommand{\dist}{\mathrm{dist}}

\newcommand{\inner}[2]{\left\langle{#1},{#2}\right\rangle}

\newcommand{\Kup}{{K_{\mathrm{up}}}}
\newcommand{\Klo}{{K_{\mathrm{lo}}}}

\newcommand{\norm}[1]{\left\|{#1}\right\|}

\newcommand{\reals}{{\mathbb{R}}}

\newcommand{\spann}{\mathrm{span}}
\newcommand{\interior}{\mathrm{int}}
\newcommand{\gspan}{\mathrm{gspan}}

\newcommand{\T}{\mathrm{T}}

\newcommand{\xorigin}{x_{\mathrm{ref}}}

\newcommand{\aref}[1]{\hyperref[#1]{A\ref{#1}}}
\newtheorem{assumption}{A\ignorespaces} 